\pdfoutput=1
\RequirePackage{silence}
\documentclass[a4paper,11pt]{amsart}
\usepackage[hmarginratio={1:1},vmarginratio={1:1},lmargin=60.0pt,tmargin=60.0pt]{geometry}

\allowdisplaybreaks

\usepackage[utf8]{inputenc}

\usepackage{latexsym,exscale,mathtools,textcomp}
\usepackage{amssymb,amsmath,amsthm,amsfonts,mathrsfs,enumitem}
\IfFileExists{stmaryrd.sty}{\usepackage{stmaryrd}}{}
\IfFileExists{dsfont.sty}{\usepackage{dsfont}}{\newcommand{\mathds}[1]{\mathbf{##1}}}
\usepackage[table]{xcolor}
\usepackage{graphicx}
\usepackage{mathtools}
\IfFileExists{ytableau.sty}{\usepackage{ytableau}}{}
\usepackage{relsize}
\usepackage{xfrac}
\usepackage{caption}
\usepackage{float}
\usepackage{xparse}
\usepackage{graphbox}
\usepackage{verbatim}
\usepackage{listings}
\IfFileExists{algorithm.sty}{\usepackage{algorithm}}{%
\newfloat{algorithm}{tbp}{loa}
\floatname{algorithm}{Algorithm}}
\IfFileExists{algcompatible.sty}{\usepackage{algcompatible}}{}
\IfFileExists{algpseudocode.sty}{\usepackage[noend]{algpseudocode}}{%
\newcount\algdepth
\newcommand{\algpad}{\hspace*{\dimexpr1.15em*\algdepth\relax}}
\newenvironment{algorithmic}{\scriptsize\begin{list}{}{\leftmargin=0pt\itemsep=0pt\parsep=0pt\topsep=1pt}}{\end{list}}
\newcommand{\State}{\item[]\algpad}
\newcommand{\Statex}{\item[]}
\newcommand{\Require}{\item[]\textbf{Input:} }
\newcommand{\Function}[2]{\item[]\algpad\textbf{function} ##1(##2)\advance\algdepth by 1}
\newcommand{\EndFunction}{\advance\algdepth by -1}
\newcommand{\If}[1]{\item[]\algpad\textbf{if} ##1 \textbf{then}\advance\algdepth by 1}
\newcommand{\ElsIf}[1]{\advance\algdepth by -1\item[]\algpad\textbf{else if} ##1 \textbf{then}\advance\algdepth by 1}
\newcommand{\Else}{\advance\algdepth by -1\item[]\algpad\textbf{else}\advance\algdepth by 1}
\newcommand{\EndIf}{\advance\algdepth by -1}
\newcommand{\For}[1]{\item[]\algpad\textbf{for} ##1 \textbf{do}\advance\algdepth by 1}
\newcommand{\ForAll}[1]{\item[]\algpad\textbf{for all} ##1 \textbf{do}\advance\algdepth by 1}
\newcommand{\EndFor}{\advance\algdepth by -1}
\newcommand{\Return}{\textbf{return} }
\newcommand{\Comment}[1]{\hfill$\triangleright$ ##1}
\newcommand{\algstore}[1]{}
\newcommand{\algrestore}[1]{}
}

\usepackage{scalefnt}

\usepackage{dynkin-diagrams}

\usepackage{array}
\newcolumntype{C}{>{$}c<{$}}

\definecolor{mygray}{gray}{0.6}
\definecolor{mygraydark}{gray}{0.4}
\definecolor{mygraylight}{gray}{0.85}
\definecolor{spinach}{RGB}{46,139,87}
\definecolor{tomato}{RGB}{255,99,71}
\definecolor{orchid}{RGB}{143,40,194}
\definecolor{neon}{RGB}{77,77,255}
\definecolor{pumpkin}{RGB}{224,180,80}
\definecolor{citron}{RGB}{190,180,90}

\definecolor{lava}{RGB}{207,16,32}
\definecolor{cream}{RGB}{255,253,208}
\definecolor{verdigris}{RGB}{67,179,174}
\definecolor{Black}{RGB}{0,0,0}
\definecolor{mydarkblue}{RGB}{10,10,170}
\definecolor{darkspinach}{RGB}{20,70,20}
\definecolor{darktomato}{RGB}{155,40,30}
\definecolor{darkorchid}{RGB}{50,10,100}
\definecolor{darklava}{RGB}{150,8,16}

\usepackage{todonotes}

\usepackage{enumitem}
\setlist[enumerate]{itemsep=0.15cm,label=\emph{\upshape(\alph*)}}
\setlist[enumerate,2]{itemsep=0.15cm,label=\emph{\upshape(\roman*)}}
\setlist[enumerate,3]{itemsep=0.15cm,label=\emph{\upshape(\Alph*)}}

\let\emph\relax
\DeclareTextFontCommand{\emph}{\bfseries\em}

\newcommand{\placeholder}{{}_{-}}

\renewcommand{\dots}{\text{...}}

\let\<=\langle
\let\>=\rangle

\renewcommand{\dots}{\text{...}}
\renewcommand{\vdots}{\rotatebox{90}{\text{...}}}

\DeclarePairedDelimiterX{\set}[1]{\{}{\}}{\setargs{#1}}
\NewDocumentCommand{\setargs}{>{\SplitArgument{1}{|}}m}{\setargsaux#1}
\NewDocumentCommand{\setargsaux}{mm}
{\IfNoValueTF{#2}{#1} {#1\,\delimsize|\,\mathopen{}#2}}

\newcommand{\N}{\mathbb{Z}_{\geq 0}}

\newcommand{\algebra}[1][A]{\mathscr{#1}}

\newcommand{\cellbasis}[1][\mathscr{A}]{B_{\mathscr{A}}}

\newcommand{\rad}[1][\jcell]{\mathrm{Rad}}

\newcommand{\jcell}{\mathcal{J}}

\newcommand{\jideal}[1][\lambda]{\algebra^{>_{lr}\lambda}}

\newcommand{\onemon}{{1}}

\newcommand{\bstirling}[2]{\begin{Bmatrix}#1\\#2\end{Bmatrix}}

\usepackage[all]{xy}
\usepackage{tikz}
\usetikzlibrary{cd}
\usetikzlibrary{decorations}
\usetikzlibrary{decorations.markings}
\usetikzlibrary{decorations.pathreplacing}
\usetikzlibrary{decorations.pathmorphing}
\usetikzlibrary{arrows.meta,shapes,positioning,matrix,calc}
\usetikzlibrary{shapes.callouts}
\usetikzlibrary{tqft}
\tikzset{
anchorbase/.style={baseline={([yshift=#1]current bounding box.center)}},
anchorbase/.default={-0.5ex},
tinynodes/.style={font=\tiny,text height=0.25ex,text depth=0.05ex},
smallnodes/.style={font=\scriptsize,text height=0.75ex,text depth=0.15ex},
mor/.style={line width=0.75,color=black,fill=cream},
mor2/.style={line width=0.75,color=black,fill=tomato},
mor3/.style={line width=0.75,color=black,fill=spinach},
usual/.style={line width=1.2,color=black},
crossline/.style={preaction={draw=white,line width=5.0pt,-},preaction={draw=black,line width=0.9pt,-}},
dot/.style = {
decoration={markings,
post length=0.25mm,
pre length=0.25mm,
mark=at position #1 with {\node[circle,radius=0.15cm,inner sep=-1.2pt,color=black,fill=black]{};}
},
postaction={decorate}
},
dot/.default=1,
dotg/.style = {
decoration={markings,
post length=0.25mm,
pre length=0.25mm,
mark=at position #1 with {\node[circle,radius=0.5cm,inner sep=-2.5pt,color=neon,fill=neon]{};}
},
postaction={decorate}
},
dotg/.default=0.5,
}
\tikzstyle directed=[postaction={decorate,decoration={markings,
mark=at position #1 with {\arrow[line width=0.25mm, black]{>}}}}]

\usepackage{pgfplots}
\pgfplotsset{compat=1.18}

\usepackage{aliascnt,etoolbox}
\def\NewTheorem#1{%
\newaliascnt{#1}{equation}%
\newtheorem{#1}[#1]{#1}%
\aliascntresetthe{#1}%
\expandafter\def\csname #1autorefname\endcsname{#1}%
}
\def\equationautorefname~#1\null{(#1)\null}

\numberwithin{equation}{subsection}

\NewTheorem{Proposition}
\NewTheorem{Theorem}
\NewTheorem{Corollary}
\AtEndEnvironment{Corollary}{\null\hfill$\square$}%
\NewTheorem{Lemma}
\NewTheorem{Conjecture}
\NewTheorem{Speculation}
\NewTheorem{Observation}
\NewTheorem{Assumption}
\NewTheorem{Condition}
\theoremstyle{definition}
\NewTheorem{Definition}
\AtEndEnvironment{Definition}{\null\hfill$\Diamond$}%
\NewTheorem{Classification Problem}
\AtEndEnvironment{Classification Problem}{\null\hfill$\Diamond$}%
\NewTheorem{Notation}
\AtEndEnvironment{Notation}{\null\hfill$\Diamond$}%
\NewTheorem{Example}
\AtEndEnvironment{Example}{\null\hfill$\Diamond$}%
\NewTheorem{Examples}
\AtEndEnvironment{Examples}{\vskip-10mm\null\hfill$\Diamond$}%

\theoremstyle{remark}
\NewTheorem{Remark}
\AtEndEnvironment{Remark}{\null\hfill$\Diamond$}%
\NewTheorem{Question}
\AtEndEnvironment{Question}{\null\hfill$\Diamond$}%

\usepackage[stretch=16,expansion=false]{microtype}
\usepackage[T1]{fontenc}
\usepackage{geometry}
\usepackage{graphicx}
\usepackage{footmisc}
\usepackage{datetime}
\usepackage[sort&compress,numbers]{natbib}
\usepackage{tikz}
\usepackage{tikz-3dplot}
\usetikzlibrary{automata,positioning,math,svg.path,tikzmark,arrows,trees,external,decorations.pathreplacing,calligraphy,patterns,patterns.meta,shapes.geometric}
\usepackage{tikz-cd}
\usepackage{float}
\usepackage{subcaption}
\usepackage{appendix}
\IfFileExists{faktor.sty}{\usepackage{faktor}}{}
\usetikzlibrary{decorations.markings}
\usepackage{etoolbox}
\usepackage{pdflscape}
\usepackage{afterpage}
\usepackage{rotating}
\usepackage{totcount}
\usepackage{suffix}
\usepackage{enumitem}
\usepackage{multirow}
\usepackage{tabularray}
\usepackage{multicol}
\usepackage{scalerel}
\usepackage{array}
\IfFileExists{matlab-prettifier.sty}{\usepackage[numbered,framed]{matlab-prettifier}}{}

\usepackage{amsmath} 
\usepackage{amsfonts} 
\IfFileExists{mathdots.sty}{\usepackage{mathdots}}{}
\usepackage{amssymb}
\IfFileExists{physics.sty}{\usepackage{physics}}{} 
\usepackage{mathrsfs} 
\usepackage{mathtools}
\usepackage{amsthm}
\usepackage{thmtools}
\IfFileExists{mathdots.sty}{\usepackage{mathdots}}{}
\IfFileExists{bbm.sty}{\usepackage{bbm}}{}
\usepackage{pifont} 
\usepackage{esint}
\usepackage{mleftright} 

\newcounter{proofs}
\newcounter{proofLevel}

\newcounter{mcases}
\newcounter{msubcases}[mcases]

\newcounter{msubsubcases}[msubcases]

\newcounter{msubsubsubcases}[msubsubcases]

\newcounter{msubsubsubsubcases}[msubsubsubcases]

\makeatletter
\newcommand{\key}[2][1]{
\quad%
\setlength{\fboxrule}{0.5pt}%
\setlength{\fboxsep}{0.4em}%
\ifnum#1=1
\fbox{%
\begin{tblr}{
columns = {c},
rows = {mode=math}
}%
\SetCell[c=2]{c}\text{\underline{Key}} & \\[1.5ex]%
#2
\end{tblr}%
}%
\else
\ifnum#1=2
\fbox{%
\begin{tblr}{
columns = {c},
rows = {mode=math}
}%
\SetCell[c=4]{c}\text{\underline{Key}} & & & \\[1.5ex]%
#2
\end{tblr}%
}%
\else
\fbox{%
\begin{tblr}{
columns = {c},
rows = {mode=math}
}%
\SetCell[c=6]{c}\text{\underline{Key}} & & & & & \\[1.5ex]%
#2
\end{tblr}%
}%
\fi%
\fi%
\,%
}

\let\oldoverrightarrow\overrightarrow
\renewcommand{\overrightarrow}[1]{\oldoverrightarrow{#1\rule{0pt}{1.9ex}}}

\mleftright 

\let\originalmiddle=\middle
\def\middle#1{\mathrel{}\originalmiddle#1\mathrel{}}

\makeatletter
\newsavebox{\@brx}
\newcommand{\llangle}[1][]{\savebox{\@brx}{\(\m@th{#1\langle}\)}%
\mathopen{\copy\@brx\mkern2mu\kern-0.9\wd\@brx\usebox{\@brx}}}
\newcommand{\rrangle}[1][]{\savebox{\@brx}{\(\m@th{#1\rangle}\)}%
\mathclose{\copy\@brx\mkern2mu\kern-0.9\wd\@brx\usebox{\@brx}}}
\makeatother

\newlength{\LactHeight}

\newlength{\halfEquals}
\usepackage[hypertexnames=false]{hyperref}
\usepackage{bookmark}
\hypersetup{
pdftoolbar=true,
pdfmenubar=true,
pdffitwindow=false,
pdfstartview={FitH},
pdftitle={Fast factorization in diagram monoids},
pdfauthor={Matthias Fresacher, Willow Stewart and Daniel Tubbenhauer},
pdfsubject={},
pdfcreator={Matthias Fresacher, Willow Stewart and Daniel Tubbenhauer},
pdfproducer={Matthias Fresacher, Willow Stewart and Daniel Tubbenhauer},
pdfkeywords={},
pdfnewwindow=true,
colorlinks=true,
linkcolor=mydarkblue,
citecolor=teal,
filecolor=magenta,
urlcolor=orchid,
linkbordercolor=lava,
citebordercolor=teal,
urlbordercolor=orchid,
linktocpage=true
}

\def\makeautorefname#1#2{\csdef{#1autorefname}{#2}}

\makeautorefname{section}{Section}%
\makeautorefname{subsection}{Section}%
\makeautorefname{subsubsection}{Section}%

\makeautorefname{algorithm}{Algorithm}%

\usepackage[capitalise,nameinlink,noabbrev]{cleveref}
\crefname{equation}{}{}
\crefname{enumi}{}{}
\crefname{corol}{Corollary}{Corollaries}
\Crefname{corol}{Corollary}{Corollaries}
\crefname{corolTOC}{Corollary}{Corollaries}
\Crefname{corolTOC}{Corollary}{Corollaries}
\crefname{task}{Task}{Tasks}
\Crefname{task}{Task}{Tasks}
\crefname{question}{Question}{Questions}
\Crefname{question}{Question}{Questions}
\crefname{comment}{Comment}{Comments}
\Crefname{comment}{Comment}{Comments}
\crefname{ass}{Assumption}{Assumptions}
\Crefname{ass}{Assumption}{Assumptions}
\crefname{assTOC}{Assumption}{Assumptions}
\Crefname{assTOC}{Assumption}{Assumptions}
\crefname{defi}{Definition}{Definitions}
\Crefname{defi}{Definition}{Definitions}
\crefname{chpDefi}{Chapter Definition}{Chapter Definitions}
\Crefname{chpDefi}{Chapter Definition}{Chapter Definitions}
\crefname{prop}{Proposition}{Propositions}
\Crefname{prop}{Proposition}{Propositions}
\crefname{propNoTOC}{Proposition}{Propositions}
\Crefname{propNoTOC}{Proposition}{Propositions}
\crefname{thm}{Theorem}{Theorems}
\Crefname{thm}{Theorem}{Theorems}
\crefname{thmNoTOC}{Theorem}{Theorems}
\Crefname{thmNoTOC}{Theorem}{Theorems}
\crefname{con}{Condition}{Conditions}
\Crefname{con}{Condition}{Conditions}
\crefname{stat}{Statement}{Statements}
\Crefname{stat}{Statement}{Statements}
\crefname{nota}{Notation}{Notations}
\Crefname{nota}{Notation}{Notations}
\crefname{mcases}{Case}{Cases}
\Crefname{mcases}{Case}{Cases}
\crefalias{msubcases}{mcases}
\crefalias{msubsubcases}{mcases}
\crefalias{msubsubsubcases}{mcases}
\crefalias{msubsubsubsubcases}{mcases}

\let\oldcref\cref
\renewcommand{\cref}[1]{%
\begingroup\hypersetup{linkcolor=black}\oldcref{#1}\endgroup
}
\let\oldCref\Cref
\renewcommand{\Cref}[1]{%
\begingroup\hypersetup{linkcolor=black}\oldCref{#1}\endgroup
}

\usepackage{fancyvrb}
\SaveVerb{webpage}'staff.cdms.westernsydney.edu.au/~mfresacher/'

\title[Fast factorization in diagram monoids]{Fast factorization in diagram monoids}
\author[M. Fresacher, W. Stewart, D. Tubbenhauer]{Matthias Fresacher, Willow Stewart, Daniel Tubbenhauer}
\address{M.F.: Western Sydney University, Centre for Research in Mathematics and Data Science, Locked Bag 1797, Penrith NSW 2751, Australia, \href{https://staff.cdms.westernsydney.edu.au/~mfresacher/}{\UseVerb[fontfamily=cmr]{webpage}}, \href{https://orcid.org/0000-0003-0677-3701}{ORCID: 0000-0003-0677-3701}}
\email{M.Fresacher@westernsydney.edu.au}
\address{W.S.: The University of Sydney, School of Mathematics and Statistics F07, Office Carslaw 807, NSW 2006, Australia, \href{https://www.maths.usyd.edu.au/ut/people?who=W_Stewart}{www.maths.usyd.edu.au/ut/people?who=W\_Stewart}, \href{https://orcid.org/0009-0000-2854-2256}{ORCID: 0009-0000-2854-2256}}
\email{willow.stewart@sydney.edu.au}
\address{D.T.: The University of Sydney, School of Mathematics and Statistics F07, Office Carslaw 827, NSW 2006, Australia, \href{http://www.dtubbenhauer.com}{www.dtubbenhauer.com}, \href{https://orcid.org/0000-0001-7265-5047}{ORCID 0000-0001-7265-5047}}
\email{daniel.tubbenhauer@sydney.edu.au}

\begin{document}

\begin{abstract}
We give explicit algorithms that factor elements of the standard diagram
monoids into their usual generators. These algorithms generalize sorting
from permutations to partial matchings and set partitions. In every case
the worst-case complexity is $n^2$, which is optimal for
algorithms that explicitly list the factors. 
We also determine the average complexity of our algorithms.
\end{abstract}

\subjclass[2020]{Primary 20M20, 68W40; Secondary 05A16, 20M05.}
\keywords{Diagram monoids, factorization algorithms, sorting, average-case complexity,
Temperley--Lieb monoids, Brauer monoids, partition monoids}

\maketitle
\tableofcontents

\section{Introduction}

This paper studies factorization into their usual generators in the standard diagram monoids such as the Temperley--Lieb, Brauer, partition monoids and others.

\subsection{Motivation}

Sorting is one of the basic problems of algorithms. Ordered data is
easier to search, compare and manipulate, and sorting therefore sits
behind a remarkable number of computational processes. It is also a model problem
for understanding what ``efficient'' actually means: one can ask for
the shortest sequence of elementary operations, the best possible
worst-case runtime, or the expected runtime of a particular algorithm.
These are related questions, but they are not the same.

For the symmetric group this distinction is especially transparent.
Let
\[
s_i=(i,i+1)\leftrightsquigarrow
\begin{tikzpicture}[anchorbase]
\draw[usual] (2,0) node[below]{$i$} to (1.5,0.5) node[above]{$i$};
\draw[usual] (1.5,0) node[below]{$i{+}1$} to (2,0.5) node[above]{$i{+}1$};
\end{tikzpicture}
,\qquad 1\leq i<n,
\]
be the adjacent transpositions with their familiar crossing picture swapping the $i$-th and $(i+1)$-th strand. 
Sorting a permutation $\pi\in S_n$ by
adjacent swaps amounts to finding a factorization
\[
\pi=s_{i_1}\cdots s_{i_r},\quad\text{e.g. }
s_1s_2s_1=
\begin{tikzpicture}[anchorbase]
\draw[usual] (0,0) node[below]{$1$} to (1,1)node[above]{$1$};
\draw[usual] (0.5,0) node[below]{$2$} to[out=135,in=225] (0.5,1)node[above]{$2$};
\draw[usual] (1,0) node[below]{$3$} to (0,1)node[above]{$3$};
\end{tikzpicture}
\leftrightsquigarrow \text{sort $[3,2,1]$ to $[1,2,3]$.}
\]
Fixing the adjacent transpositions as generators of the symmetric group,
there are then three natural quantities to consider:
\begin{gather*}
\ell(\pi)
=
\min\{r\mid \pi=s_{i_1}\cdots s_{i_r}\},\quad
\operatorname{Opt}_n
=
\inf_{A\in Al}\max_{\pi\in S_n}T_A(\pi),\quad
\operatorname{Avg}_n(A)
=
\frac{1}{n!}\sum_{\pi\in S_n}T_A(\pi).
\end{gather*}
Here $Al$ ranges over algorithms that explicitly return a word in the
generators, and $T_A(\pi)$ denotes the runtime of $A$ on $\pi$. Thus
$\ell(\pi)$ is the minimal number of generators, while
$\operatorname{Opt}_n$ is the minimal possible worst-case runtime for
the factorization problem. The latter is what we mean by minimal
runtime. Both are properties of the problem (once the generators,
input format and required output have been fixed).
By contrast, $\operatorname{Avg}_n(A)$ belongs to a particular algorithm
and a particular distribution of inputs; here and below we use the
uniform distribution.

\begin{Remark}
For the symmetric group, all three quantities are well-understood see e.g. our go-to-reference \cite[Section 5]{KT-AlgDesign} (among many other references).
\end{Remark}

All three quantities
are useful, for different reasons. The number
$\ell(\pi)$ measures the intrinsic distance from $\pi$ to the identity.
It is the length of the shortest word in adjacent transpositions for $\pi$.
For $S_n$ this is the
so-called inversion number $\operatorname{inv}(\placeholder)$, and computing it is completely understood, see e.g. \cite[Section 5]{KT-AlgDesign} (among many other references).

Minimal runtime asks a different question: how quickly can one find
some valid factorization in the worst case? A short factorization length need
not be easy to find, and a fast algorithm need not return a shortest
factorization length. In the symmetric group these two issues happen to cooperate.
The reverse permutation has $\binom{n}{2}$ inversions, so every explicit
factorization of it uses at least $\binom{n}{2}$ crossings, e.g.
\[
\begin{tikzpicture}[
anchorbase,
x=1cm,
y=0.5cm,
line join=miter,
line cap=round
]
\draw[usual] (0,5) node[above]{$5$} -- (4,0) node[below]{$5$};
\draw[usual] (4,5) node[above]{$1$} -- (0,0) node[below]{$1$};
\draw[usual] (1,5) node[above]{$4$} -- (0,3.75) -- (3,0) node[below]{$4$};
\draw[usual] (3,5) node[above]{$2$} -- (0,1.25) -- (1,0) node[below]{$2$};
\draw[usual] (2,5) node[above]{$3$} -- (0,2.5) -- (2,0) node[below]{$3$};
\end{tikzpicture}
\text{ for }n=5.
\]
Consequently, any algorithm that prints a factorization has worst-case
runtime $\Omega(n^2)$. 

\begin{Notation}
Here and throughout, we use capital $O$ notation (also called Bachmann--Landau notation): 
$\Omega$ denotes an asymptotic lower bound, $O$ an upper bound and $\Theta$ a tight bound (both an upper and lower bound). We also use $\sim$ to mean asymptotically equal.
\end{Notation}

Standard sorting algorithms give the matching
upper bound \cite[Section 5]{KT-AlgDesign}, and hence
\[
\operatorname{Opt}_n\in\Theta(n^2).
\]
This is an output-size lower bound: one can avoid it only by changing
what is meant by the output.

Average runtime answers another question. A worst-case estimate
gives a guarantee, but it need not describe what an algorithm normally
does. Pathological inputs may be rare, and two algorithms with the same
worst-case complexity may behave differently on typical inputs.
For example, for ordinary sorting, insertion sort and quicksort both have worst-case
runtime $\Theta(n^2)$, but their average runtimes on uniformly random
permutations are $\Theta(n^2)$ and $\Theta(n\log n)$, cf. \cite[Section 5]{KT-AlgDesign}.

Average analysis makes this distinction visible, provided the
algorithm and input distribution are stated explicitly. For a uniformly
random permutation, it is well known that
\[
\mathbb{E}[\operatorname{inv}(\pi)]
=
\tfrac{n(n-1)}{4}.
\]
Thus an algorithm producing a reduced adjacent-transposition-factorization has quadratic average runtime as well. The reverse
permutation is not solely responsible for the quadratic behavior: a
random permutation already contains quadratically many inversions \cite[Section 5]{KT-AlgDesign}.
This is precisely the kind of information that worst-case complexity
alone cannot provide.

The present paper studies the analogues of the latter two questions for
diagram monoids: we determine the optimal worst-case runtime, give explicit
factorization algorithms attaining it, and calculate the average
runtime of these algorithms.

\subsection{Contribution}

A permutation can be represented by a diagram with two rows of $n$
vertices, in which every vertex is joined to exactly one vertex in the
opposite row. Relaxing these conditions gives the standard diagram
monoids. One may allow pairs in the same row, isolated vertices, blocks
of arbitrary size, or impose planarity. This produces, for example, the Temperley--Lieb, Brauer, rook, planar rook, rook--Brauer, Motzkin, partition and
planar partition monoids. For example, the Brauer monoid has pictures of the 
form
\[
\begin{tikzpicture}[anchorbase]
\draw[usual] (0.5,0) to[out=90,in=180] (1.25,0.45) to[out=0,in=90] (2,0);
\draw[usual] (1,0) to[out=90,in=180] (1.25,0.25) to[out=0,in=90] (1.5,0);
\draw[usual] (0,1) to[out=270,in=180] (0.75,0.55) to[out=0,in=270] (1.5,1);
\draw[usual] (1,1) to[out=270,in=180] (1.75,0.55) to[out=0,in=270] (2.5,1);
\draw[usual] (0,0) to (0.5,1);
\draw[usual] (2.5,0) to (2,1);
\end{tikzpicture}
\]
and contains the symmetric group as crossings. 

\begin{Remark}
Terminology is not uniform. For example, the Temperley--Lieb monoid is
also known as the Rumer--Teller--Weyl or Jones--Kauffman monoid.
This abundance of names partly reflects the fact that the same diagram
calculus was discovered independently in several different settings;
see, e.g., \cite{Jo87,Ka87,RTW-Valenztheorie,TL71}. 

The same is true for the other monoids in this paper. In particular, 
the rook monoid is also called the symmetric inverse monoid, while the
planar rook monoid is the monoid of order-preserving partial
permutations. The rook--Brauer monoid is also known as the partial
Brauer monoid. Its planar submonoid, the Motzkin monoid, has appeared
under the names planar partial Brauer, partial Jones and partial
Temperley--Lieb monoid. The words ``monoid'' and ``semigroup'' are also
interchanged rather freely, and each of these objects has a corresponding
diagram algebra, bringing its own supply of notation.

The objects are standard; only their names are not. We make no attempt
to list every independent discovery: the list would be long and, with
near certainty, unfair to someone.
\end{Remark}

\begin{Remark}
There are many more diagram monoids, for example partial transformation monoids \cite{EKMW-MaxSubsemigroup}, annular monoids \cite{HT-AffineDiag}, blob monoids \cite{MS-BlobAlg}, framed monoids \cite{AJP-Framization}, framed blob monoids \cite{JL-FramedBlob}, the full-domain and trivial co-kernel partition monoids \cite{CEF-full-domain,CEF-trivial-cokernel}, and many more. We prioritize the more established diagram monoids, but it is an interesting question to extend our methods to these and others.
\end{Remark}

The natural generalization of sorting is then to factorize a diagram
into its standard local generators (i.e. only acting on up to four adjacent vertices, e.g. adjacent transpositions in the symmetric group). These generators interchange
neighboring strands or create elementary cups, caps, isolated vertices
and larger blocks. For example, for the Brauer monoid the standard generators are (here we mark the positions)
\[
e_i=\quad
\begin{tikzpicture}[anchorbase]
\draw[usual] (0,0) to (0,1);
\draw[usual] (0.5,0.5) node {$\dots$};
\draw[usual] (1,0) to (1,1);
\draw[usual] (1.5,0) to[out=90,in=180] (1.75,0.25) to[out=0,in=90] (2,0);
\draw[usual] (1.5,1) to[out=270,in=180] (1.75,0.75) to[out=0,in=270] (2,1);
\draw[usual] (1.5,1) node [above] {$i$};
\draw[usual] (1.5,0) node [below] {$i$};
\draw[usual] (2.5,0) to (2.5,1);
\draw[usual] (3,0.5) node {$\dots$};
\draw[usual] (3.5,0) to (3.5,1);
\end{tikzpicture}
\quad,\quad s_i=\quad
\begin{tikzpicture}[anchorbase]
\draw[usual] (0,0) to (0,1);
\draw[usual] (0.5,0.5) node {$\dots$};
\draw[usual] (1,0) to (1,1);
\draw[usual] (1.5,0) to (2,1);
\draw[usual] (2,0) to (1.5,1);
\draw[usual] (1.5,1) node [above] {$i$};
\draw[usual] (1.5,0) node [below] {$i$};
\draw[usual] (2.5,0) to (2.5,1);
\draw[usual] (3,0.5) node {$\dots$};
\draw[usual] (3.5,0) to (3.5,1);
\end{tikzpicture}
.
\]
Presentations by generators and relations are well
known, see e.g. \cite{HaRa-partition-algebras}, but a presentation proves only that such a factorization
exists. It does not provide an efficient procedure for finding one. A
computer, with its usual lack of initiative, requires the procedure.

Constructive factorization is useful whenever diagrams must be computed
with. Representations are generally specified on generators, so a
factorization reduces the evaluation of an arbitrary diagram to local
operations. It also provides an independently checkable certificate and
a direct interface between diagrammatic input and algebraic software.

\begin{Remark}
Precomputation is not realistic: the Temperley--Lieb monoid has
approximately $4^n$ elements, and the Brauer monoid has approximately
$(2n/e)^n$ elements. Complete lookup
tables therefore become fictional rather quickly; see \cite[A000108, A001147 and A000110]{oeis}.
\end{Remark}

For every family of diagram monoid considered here, there are diagrams for which every
factorization in the standard generators has length $\Omega(n^2)$.
Any explicit-output algorithm must spend at least this long writing down
the answer. We give matching $O(n^2)$ algorithms and therefore determine
the minimal worst-case runtime:
\[
\operatorname{Opt}_n\in\Theta(n^2)
\]
for all nine families. The words produced need not be shortest, but the
time taken to produce them is asymptotically optimal.

The common worst-case answer hides substantial differences between the
monoids. We therefore also calculate the average runtime of each of our
algorithms under the uniform distribution. The results are
summarized in \cref{tab:factorization-complexities}.

\begin{table}[ht]
\centering
\renewcommand{\arraystretch}{1.15}
\begin{tabular}{|l|c|c|}
\hline
Monoid
&
Worst-case runtime
&
Average runtime
\\
\hline
Temperley--Lieb monoid $TL_n$
& $\Theta(n^2)$ & $\Theta(n^{3/2})$ \\
\hline
Planar rook monoid $pRo_n$
& $\Theta(n^2)$ & $\Theta(n^{3/2})$ \\
\hline
Motzkin monoid $Mo_n$
& $\Theta(n^2)$ & $\Theta(n^{3/2})$ \\
\hline
Planar partition monoid $pPa_n$
& $\Theta(n^2)$ & $\Theta(n^{3/2})$ \\
\hline
Symmetric group $S_n$
& $\Theta(n^2)$ & $\Theta(n^2)$ \\
\hline
Brauer monoid $Br_n$
& $\Theta(n^2)$ & $\Theta(n^2)$ \\
\hline
Rook monoid $Ro_n$
& $\Theta(n^2)$ & $\Theta(n^2)$ \\
\hline
Rook--Brauer monoid $RoBr_n$
& $\Theta(n^2)$ & $\Theta(n^2)$ \\
\hline
Partition monoid $Pa_n$
& $\Theta(n^2)$ & $\Theta(n^2)$ \\
\hline
\end{tabular}
\caption{Worst-case runtimes and average runtimes of our
algorithms. More precise formulas, including leading constants and
lower-order terms, are given in \autoref{S:Planar} and \autoref{S:Symmetric}.}
\label{tab:factorization-complexities}
\end{table}

The average runtimes separate into two regimes. They are
$\Theta(n^{3/2})$ for the planar monoids and $\Theta(n^2)$ for the symmetric monoids. Thus the quadratic worst case is typical in some
families and exceptional in others. Planarity restricts the amount of
displacement and nesting in a random diagram, while unrestricted
permutations and set partitions retain enough global disorder for the
quadratic behavior to survive.

\begin{Remark}
The precise averages explain these exponents. Their derivation involves
Dyck paths, Motzkin numbers, Vandermonde identities, Laguerre
polynomials, telephone numbers, Bell numbers and the Lambert $W$
function. The average analysis therefore does more than predict
runtime: it identifies the combinatorial features responsible for the
work performed by the algorithms.
\end{Remark}

Thus, our main result, whose proof occupies the remainder of the paper apart from \autoref{S:Experiment}, is the following:
\begin{enumerate}

\item We show that the minimal runtime of any factorization algorithm is $\Omega(n)$, and give matching fast algorithms. Thus, $\operatorname{Opt}_n\in\Theta(n^2)$. These have been implemented and are available online \cite{St-github-FastFact}.

\item We explicitly analyze the average runtime $\operatorname{Avg}_n$ of our algorithms, as summarized in \autoref{tab:factorization-complexities}.

\end{enumerate}
Finally, \autoref{S:Experiment} provides experimental evidence for the average runtime of implementations of our algorithms.

\subsection{Relation to earlier work}\label{S:Relation}

Normal forms and factorizations in diagram monoids have deep roots in 
classical representation theory and knot theory, tracing back to, e.g.,
\cite{Br37,RTW-Valenztheorie,TL71}, though 
early work largely treated these structures through diagrammatic bases, generators, 
and matrix representations rather than explicit semigroup-theoretic factorizations. 
The canonical diagrammatic word problem was formalized in foundational cases such 
as the Temperley--Lieb monoid \cite{Jo87,Ka87}, 
while \cite{TL-factorization} gives a direct algorithm for recovering a reduced 
factorization from a diagram. 

More recently, structural decompositions, factorizations, and normal forms have also been 
studied across broader diagram families, see e.g., 
\cite{BH-Motzkin,DEG-MotzPartialBrauer,Ea-partition-presentations,Hu-diagram-categories}. 
Explicit complexity and minimal factorizations have only recently been analyzed algorithmically, such as in 
\cite{Br-factorization} for the Brauer monoid.

What seems to have been missing is a systematic study of factorization
as an algorithmic problem. We give explicit algorithms for all the
diagram monoids considered here, using their usual local generators,
and determine their worst-case runtime in terms of $\Theta(\placeholder)$. 
We also choose a diagram
uniformly at random and determine the average length of the word
produced by our algorithm. As far as we know, such averages have not
previously been studied for factorization algorithms in diagram
monoids.
Taken together, this gives the first uniform treatment of factorization
in these standard families which combines explicit algorithms, optimal
worst-case runtime, and uniform-input averages. 

Of course, the symmetric group case is classical
and serves as the model for the general picture.

\subsection{A few additional remarks}

We finish this introduction with some comments on directions we do not explore in this paper.

\begin{Remark}
Changing the generating set changes the factorization problem. For
$S_n$, a classical alternative to all adjacent transpositions is the
two-element generating set consisting of $s=s_1$ and the long cycle
$c$ corresponding to sorting the list $[2,\dots,n,1]$. For $n=5$, these are
\[ 
\begin{tikzpicture}[anchorbase, x=1cm, y=0.5cm, line join=miter, line cap=round] 
\draw[usual] (1,1) node[above]{$1$} -- (0,0) node[below]{$1$}; \draw[usual] (0,1) node[above]{$2$} -- (1,0) node[below]{$2$}; \draw[usual] (2,1) node[above]{$3$} -- (2,0) node[below]{$3$}; \draw[usual] (3,1) node[above]{$4$} -- (3,0) node[below]{$4$}; \draw[usual] (4,1) node[above]{$5$} -- (4,0) node[below]{$5$}; \end{tikzpicture} 
\text{ and } 
\begin{tikzpicture}[anchorbase, x=1cm, y=0.5cm, line join=miter, line cap=round] 
\draw[usual] (4,1) node[above]{$1$} -- (0,0) node[below]{$1$}; \draw[usual] (0,1) node[above]{$2$} -- (1,0) node[below]{$2$}; \draw[usual] (1,1) node[above]{$3$} -- (2,0) node[below]{$3$}; \draw[usual] (2,1) node[above]{$4$} -- (3,0) node[below]{$4$}; \draw[usual] (3,1) node[above]{$5$} -- (4,0) node[below]{$5$}; \end{tikzpicture} 
\text{ for }n=5. 
\]
The two choices measure complexity differently, e.g. the long cycle has $n-1$ crossings but length one in
$\{s,c\}$. Neither generating set is uniformly better:
crossings are adapted to local diagrammatic operations
and representations, while cycle-based generators are better suited to
compact encodings and canonical normal forms, including those related
to Lehmer codes \cite{Lehmer-code}. The optimal worst-case runtime is still bounded by $n^2$, see, e.g., \cite{BKL89}.

For the Brauer monoid, the analogous small generating set is obtained
by adjoining one elementary cup-cap $e_1$ to $s$ and $c$. The
generation statement is immediate; however, the resulting factorization problem
is not. It would be interesting to determine the optimal worst-case
runtime and the average runtime for this generating set, and to compare
them with the local generators used here; it is at least $\Omega(n^3)$, as any Brauer diagram with $\Omega(n)$ cups requires $\Omega(n)$ permutations for a factorization in these generators. Similar questions arise for other generating sets and
all the diagram monoids considered in this paper.
\end{Remark}

\begin{Remark}
The geodesic problem of finding the minimal number of generators has
been studied for the symmetric group, Temperley--Lieb monoid and Brauer
monoid \cite{KT-AlgDesign,TL-factorization,Br-factorization}. For example, the known algorithm for a minimal Brauer factorization has complexity
$O(n^4)$ \cite{Br-factorization}, whereas our nonminimal algorithm runs in $\Theta(n^2)$.
\end{Remark}

\begin{Remark}
There is also a secondary motivation from cryptography. Diagram monoids
have been proposed as platforms for noncommutative protocols \cite{Arms-Motzkin-gap,FST-GeneralizedRepGap,khovanov-monoidal-2024,Liu-cyclic-diagram,ST-RepGapRigidPlanar}, and
such proposals require basic algorithms for working with the platform
monoids. A platform that nobody can compute with is secure in much the
same sense as a password that nobody, including its owner, remembers.
We stress that our algorithms neither solve the factorization search
problem nor establish the security of any protocol. They provide one
piece of the required computational infrastructure, but exploring this further is a main future goal.
\end{Remark}

\subsection{Paper structure}

The next section recalls the diagram monoids and fixes their standard
generators. We then give and analyze the factorization algorithms,
starting with the planar families and the symmetric group before turning
to the remaining nonplanar cases. We finish by comparing the theoretical
average complexities with experimental runtimes.

\noindent\textbf{Acknowledgments.}
WS was supported by the
Postgraduate Research Scholarship in Mathematics and Statistics
(SC4238), and DT by ARC Future Fellowship FT230100489.

\noindent\textbf{AI declaration.}
OpenAI's GPT~5.6 through Microsoft Copilot
assisted with random sampling and several of the longer computations and, unlike DT,
showed no visible signs of regret.

\section{Diagram monoids}\label{S:DiagMon}

We assume the reader is familiar with diagram monoids and refer to \cite{FST-GeneralizedRepGap} for a selective, if admittedly biased, list of references.

\subsection{Basics}

We now setup our notation.
Write
\[
[n]=\{1,\dots,n\}
\qquad\text{and}\qquad
[n]'=\{1',\dots,n'\},
\]
where the two sets are disjoint. A \textit{partition diagram} of
degree $n$ is a set partition of $[n]\sqcup[n]'$, drawn using two rows
of vertices. The vertices in the top row are labeled $1,\dots,n$ and
those in the bottom row by $1',\dots,n'$. Vertices belong to the same
block precisely when they lie in the same connected component of the
diagram. Only connectivity matters: the shape of the curves carries no
information, and crossings are not additional vertices. We depict
singleton blocks by short marked stubs, called dots.

Given two diagrams $\alpha$ and $\beta$, their product
$\alpha\beta$ is obtained by placing $\alpha$ above $\beta$, identifying
the bottom row of $\alpha$ with the top row of $\beta$, and retaining
the induced partition of the two outer rows. Any component contained
entirely in the middle row is discarded. This operation is associative,
and its identity is the diagram
\[
\mathds{1}=
\begin{tikzpicture}[anchorbase]
\draw[usual] (0,0) to (0,1);
\draw[usual] (0.5,0.5) node {$\dots$};
\draw[usual] (1,0) to (1,1);
\end{tikzpicture}
.
\]
The resulting monoid is the \textit{partition monoid}.

A block meeting both rows is called \textit{transversal}; all other
blocks are \textit{nontransversal}. The \textit{rank} of a diagram is
the number of its transversal blocks. A diagram is \textit{planar} if
it can be drawn inside the rectangle formed by its vertices without edges crossing. The nonplanar diagram monoids listed here are \textit{symmetric}, since they contain the symmetric group.

For $n\in\N$, the diagram monoids considered in this paper are as
follows.

\begin{enumerate}[label=$\bullet$]

\item The \textit{partition monoid} $Pa_n$ consists of all partition
diagrams of degree $n$. The \textit{planar partition monoid} $pPa_n$
is its planar submonoid.
\begin{gather*}
\begin{tikzpicture}[anchorbase]
\draw[usual] (0.5,0) to[out=90,in=180] (1.25,0.45) to[out=0,in=90] (2,0);
\draw[usual] (0.5,0) to[out=90,in=180] (1,0.35) to[out=0,in=90] (1.5,0);
\draw[usual] (0,1) to[out=270,in=180] (0.75,0.55) to[out=0,in=270] (1.5,1);
\draw[usual] (1.5,1) to[out=270,in=180] (2,0.55) to[out=0,in=270] (2.5,1);
\draw[usual] (0,0) to (0.5,1);
\draw[usual] (1,0) to (1,1);
\draw[usual] (2.5,0) to (2.5,1);
\draw[usual,dot] (2,1) to (2,0.8);
\end{tikzpicture}
\in Pa_6
,\quad
\begin{tikzpicture}[anchorbase]
\draw[usual] (0.5,0) to[out=90,in=180] (1.25,0.45) to[out=0,in=90] (2,0);
\draw[usual] (0.5,0) to[out=90,in=180] (1,0.35) to[out=0,in=90] (1.5,0);
\draw[usual] (0.5,1) to[out=270,in=180] (1,0.55) to[out=0,in=270] (1.5,1);
\draw[usual] (1.5,1) to[out=270,in=180] (2,0.55) to[out=0,in=270] (2.5,1);
\draw[usual] (0,0) to (0,1);
\draw[usual] (2.5,0) to (2.5,1);
\draw[usual,dot] (1,0) to (1,0.2);
\draw[usual,dot] (1,1) to (1,0.8);
\draw[usual,dot] (2,1) to (2,0.8);
\end{tikzpicture}
\in pPa_6
.
\end{gather*}

\item The \textit{rook--Brauer monoid} $RoBr_n$, also called the
\textit{partial Brauer monoid}, consists of all diagrams whose blocks
have size at most two. Its planar submonoid
$pRoBr_n=Mo_n$ is the \textit{Motzkin monoid}.
\begin{gather*}
\begin{tikzpicture}[anchorbase]
\draw[usual] (1,0) to[out=90,in=180] (1.25,0.25) to[out=0,in=90] (1.5,0);
\draw[usual] (1,1) to[out=270,in=180] (1.75,0.55) to[out=0,in=270] (2.5,1);
\draw[usual] (0,0) to (0.5,1);
\draw[usual] (2.5,0) to (2,1);
\draw[usual,dot] (0.5,0) to (0.5,0.2);
\draw[usual,dot] (2,0) to (2,0.2);
\draw[usual,dot] (0,1) to (0,0.8);
\draw[usual,dot] (1.5,1) to (1.5,0.8);
\end{tikzpicture}
\in RoBr_6
,\quad
\begin{tikzpicture}[anchorbase]
\draw[usual] (0.5,0) to[out=90,in=180] (1.25,0.5) to[out=0,in=90] (2,0);
\draw[usual] (1,0) to[out=90,in=180] (1.25,0.25) to[out=0,in=90] (1.5,0);
\draw[usual] (2,1) to[out=270,in=180] (2.25,0.75) to[out=0,in=270] (2.5,1);
\draw[usual] (0,0) to (1,1);
\draw[usual,dot] (2.5,0) to (2.5,0.2);
\draw[usual,dot] (0,1) to (0,0.8);
\draw[usual,dot] (0.5,1) to (0.5,0.8);
\draw[usual,dot] (1.5,1) to (1.5,0.8);
\end{tikzpicture}
\in Mo_6
.
\end{gather*}

\item The \textit{Brauer monoid} $Br_n$ consists of all diagrams whose
blocks have size two. Its planar submonoid $pBr_n=TL_n$ is the
\textit{Temperley--Lieb monoid}, also called the \textit{Jones monoid}.
\begin{gather*}
\begin{tikzpicture}[anchorbase]
\draw[usual] (0.5,0) to[out=90,in=180] (1.25,0.45) to[out=0,in=90] (2,0);
\draw[usual] (1,0) to[out=90,in=180] (1.25,0.25) to[out=0,in=90] (1.5,0);
\draw[usual] (0,1) to[out=270,in=180] (0.75,0.55) to[out=0,in=270] (1.5,1);
\draw[usual] (1,1) to[out=270,in=180] (1.75,0.55) to[out=0,in=270] (2.5,1);
\draw[usual] (0,0) to (0.5,1);
\draw[usual] (2.5,0) to (2,1);
\end{tikzpicture}
\in Br_6
,\quad
\begin{tikzpicture}[anchorbase]
\draw[usual] (0.5,0) to[out=90,in=180] (1.25,0.5) to[out=0,in=90] (2,0);
\draw[usual] (1,0) to[out=90,in=180] (1.25,0.25) to[out=0,in=90] (1.5,0);
\draw[usual] (0,1) to[out=270,in=180] (0.25,0.75) to[out=0,in=270] (0.5,1);
\draw[usual] (2,1) to[out=270,in=180] (2.25,0.75) to[out=0,in=270] (2.5,1);
\draw[usual] (0,0) to (1,1);
\draw[usual] (2.5,0) to (1.5,1);
\end{tikzpicture}
\in TL_6
.
\end{gather*}

\item The \textit{rook monoid}, or \textit{symmetric inverse monoid},
$Ro_n$ consists of all diagrams whose blocks are either singletons or
transversal pairs. Equivalently, its elements are partial bijections of
$[n]$. The \textit{planar rook monoid} $pRo_n$ is the corresponding
planar submonoid; equivalently, it consists of the order-preserving
partial bijections.
\begin{gather*}
\begin{tikzpicture}[anchorbase]
\draw[usual] (0,0) to (1,1);
\draw[usual] (0.5,0) to (0,1);
\draw[usual] (2,0) to (2,1);
\draw[usual] (2.5,0) to (0.5,1);
\draw[usual,dot] (1,0) to (1,0.2);
\draw[usual,dot] (1.5,0) to (1.5,0.2);
\draw[usual,dot] (1.5,1) to (1.5,0.8);
\draw[usual,dot] (2.5,1) to (2.5,0.8);
\end{tikzpicture}
\in Ro_6
,\quad
\begin{tikzpicture}[anchorbase]
\draw[usual] (0,0) to (0.5,1);
\draw[usual] (0.5,0) to (1,1);
\draw[usual] (2,0) to (1.5,1);
\draw[usual] (2.5,0) to (2.5,1);
\draw[usual,dot] (1,0) to (1,0.2);
\draw[usual,dot] (1.5,0) to (1.5,0.2);
\draw[usual,dot] (0,1) to (0,0.8);
\draw[usual,dot] (2,1) to (2,0.8);
\end{tikzpicture}
\in pRo_6
.
\end{gather*}

\item The \textit{symmetric group} $S_n$ consists of all diagrams whose
blocks are transversal pairs. Thus every top vertex is joined to
exactly one bottom vertex and conversely. The only planar permutation
diagram is the identity, so $pS_n\cong\onemon$.
\begin{gather*}
\begin{tikzpicture}[anchorbase]
\draw[usual] (0,0) to (1,1);
\draw[usual] (0.5,0) to (0,1);
\draw[usual] (1,0) to (1.5,1);
\draw[usual] (1.5,0) to (2.5,1);
\draw[usual] (2,0) to (2,1);
\draw[usual] (2.5,0) to (0.5,1);
\end{tikzpicture}
\in S_6
,\quad
\begin{tikzpicture}[anchorbase]
\draw[usual] (0,0) to (0,1);
\draw[usual] (0.5,0) to (0.5,1);
\draw[usual] (1,0) to (1,1);
\draw[usual] (1.5,0) to (1.5,1);
\draw[usual] (2,0) to (2,1);
\draw[usual] (2.5,0) to (2.5,1);
\end{tikzpicture}
\in pS_6
.
\end{gather*}

\end{enumerate}

For $n\geq2$, all inclusions among these monoids are generated by the
following diagram. The eight monoids with blocks of size at most two
form a cube, while $pPa_n$ and $Pa_n$ add one further face.
\[
\begin{tikzpicture}[
anchorbase,
x=1.5cm,
y=1.15cm,
every node/.style={inner sep=2pt},
every path/.style={->,>=stealth}
]
\node (TL)   at (-2,-1) {$TL_n$};
\node (Mo)   at (-2,1) {$Mo_n$};
\node (RoBr) at (2,1) {$RoBr_n$};
\node (Br)   at (2,-1) {$Br_n$};
\draw (TL) -- (Mo);
\draw (Mo) -- (RoBr);
\draw (Br) -- (RoBr);
\draw (TL) -- (Br);
\node (pS)  at (-0.7,-0.4) {$pS_n$};
\node (pRo) at (-0.7,0.4) {$pRo_n$};
\node (Ro)  at (0.7,0.4) {$Ro_n$};
\node (S)   at (0.7,-0.4) {$S_n$};
\draw (pS) -- (pRo);
\draw (pRo) -- (Ro);
\draw (S) -- (Ro);
\draw (pS) -- (S);
\draw (pS) -- (TL);
\draw (pRo) -- (Mo);
\draw (Ro) -- (RoBr);
\draw (S) -- (Br);
\node (pPa) at (-2,2.35) {$pPa_n$};
\node (Pa)  at (2,2.35) {$Pa_n$};
\draw (Mo) -- (pPa);
\draw (pPa) -- (Pa);
\draw (RoBr) -- (Pa);
\end{tikzpicture}
.
\]
Every arrow denotes inclusion; all further inclusions follow by
transitivity. For $n=0,1$, this diagram majorly collapses.

We use the following standard local generators. All vertices not
shown locally are joined vertically. For $1\leq i<n$, the generator
$p_i$ places $i$, $i+1$, $i'$, and $(i+1)'$ in one block; $e_i$ has the two blocks
$\{i,i+1\}$ and $\{i',(i+1)'\}$; and $s_i$ interchanges the two
neighboring transversal strings. For $1\leq i\leq n$, the generator
$d_i$ makes $i$ and $i'$ into separate singleton blocks.
\begin{gather*}
p_i=\quad
\begin{tikzpicture}[anchorbase]
\draw[usual] (0,0) to (0,1);
\draw[usual] (0.5,0.5) node {$\dots$};
\draw[usual] (1,0) to (1,1);
\draw[usual] (1.5,0) to (1.5,1);
\draw[usual] (1.5,0) to[out=90,in=180] (1.75,0.25) to[out=0,in=90] (2,0);
\draw[usual] (1.5,1) to[out=270,in=180] (1.75,0.75) to[out=0,in=270] (2,1);
\draw[usual] (1.5,1) node[above] {$i$};
\draw[usual] (1.5,0) node[below] {$i$};
\draw[usual] (2.5,0) to (2.5,1);
\draw[usual] (3,0.5) node {$\dots$};
\draw[usual] (3.5,0) to (3.5,1);
\end{tikzpicture}
\quad,\quad
d_i=\quad
\begin{tikzpicture}[anchorbase]
\draw[usual] (0,0) to (0,1);
\draw[usual] (0.5,0.5) node {$\dots$};
\draw[usual] (1,0) to (1,1);
\draw[usual,dot] (1.5,0) to (1.5,0.2);
\draw[usual,dot] (1.5,1) to (1.5,0.8);
\draw[usual] (1.5,1) node[above] {$i$};
\draw[usual] (1.5,0) node[below] {$i$};
\draw[usual] (2,0) to (2,1);
\draw[usual] (2.5,0.5) node {$\dots$};
\draw[usual] (3,0) to (3,1);
\end{tikzpicture}
\quad ,
\\
e_i=\quad
\begin{tikzpicture}[anchorbase]
\draw[usual] (0,0) to (0,1);
\draw[usual] (0.5,0.5) node {$\dots$};
\draw[usual] (1,0) to (1,1);
\draw[usual] (1.5,0) to[out=90,in=180] (1.75,0.25) to[out=0,in=90] (2,0);
\draw[usual] (1.5,1) to[out=270,in=180] (1.75,0.75) to[out=0,in=270] (2,1);
\draw[usual] (1.5,1) node[above] {$i$};
\draw[usual] (1.5,0) node[below] {$i$};
\draw[usual] (2.5,0) to (2.5,1);
\draw[usual] (3,0.5) node {$\dots$};
\draw[usual] (3.5,0) to (3.5,1);
\end{tikzpicture}
\quad,\quad
s_i=\quad
\begin{tikzpicture}[anchorbase]
\draw[usual] (0,0) to (0,1);
\draw[usual] (0.5,0.5) node {$\dots$};
\draw[usual] (1,0) to (1,1);
\draw[usual] (1.5,0) to (2,1);
\draw[usual] (2,0) to (1.5,1);
\draw[usual] (1.5,1) node[above] {$i$};
\draw[usual] (1.5,0) node[below] {$i$};
\draw[usual] (2.5,0) to (2.5,1);
\draw[usual] (3,0.5) node {$\dots$};
\draw[usual] (3.5,0) to (3.5,1);
\end{tikzpicture}
\quad .
\end{gather*}
For the planar rook and Motzkin monoids we also use
\[
r_i=d_is_i,
\qquad
l_i=s_id_i.
\]
Although these expressions involve $s_i$, the resulting diagrams are
planar:
\begin{gather*}
r_i=d_is_i=\quad
\begin{tikzpicture}[anchorbase]
\draw[usual] (0,0) to (0,1);
\draw[usual] (0.5,0.5) node {$\dots$};
\draw[usual] (1,0) to (1,1);
\draw[usual] (1.5,0) to (2,1);
\draw[usual,dot] (1.5,1) to (1.5,0.8);
\draw[usual,dot] (2,0) to (2,0.2);
\draw[usual] (1.5,1) node[above] {$i$};
\draw[usual] (1.5,0) node[below] {$i$};
\draw[usual] (2.5,0) to (2.5,1);
\draw[usual] (3,0.5) node {$\dots$};
\draw[usual] (3.5,0) to (3.5,1);
\end{tikzpicture}
\quad,\quad
l_i=s_id_i=\quad
\begin{tikzpicture}[anchorbase]
\draw[usual] (0,0) to (0,1);
\draw[usual] (0.5,0.5) node {$\dots$};
\draw[usual] (1,0) to (1,1);
\draw[usual] (2,0) to (1.5,1);
\draw[usual,dot] (1.5,0) to (1.5,0.2);
\draw[usual,dot] (2,1) to (2,0.8);
\draw[usual] (1.5,1) node[above] {$i$};
\draw[usual] (1.5,0) node[below] {$i$};
\draw[usual] (2.5,0) to (2.5,1);
\draw[usual] (3,0.5) node {$\dots$};
\draw[usual] (3.5,0) to (3.5,1);
\end{tikzpicture}
\quad .
\end{gather*}
For $n\geq2$, the generating sets used throughout the paper are
\begin{align*}
Pa_n
&=\langle p_i,d_j,s_i\mid
1\leq i<n,\ 1\leq j\leq n\rangle,\\
pPa_n
&=\langle p_i,d_j\mid
1\leq i<n,\ 1\leq j\leq n\rangle,\\
RoBr_n
&=\langle d_j,e_i,s_i\mid
1\leq i<n,\ 1\leq j\leq n\rangle,\\
Mo_n
&=\langle e_i,r_i,l_i\mid
1\leq i<n\rangle,\\
Br_n
&=\langle e_i,s_i\mid
1\leq i<n\rangle,\\
TL_n
&=\langle e_i\mid
1\leq i<n\rangle,\\
Ro_n
&=\langle d_j,s_i\mid
1\leq i<n,\ 1\leq j\leq n\rangle,\\
pRo_n
&=\langle r_i,l_i\mid
1\leq i<n\rangle,\\
S_n
&=\langle s_i\mid
1\leq i<n\rangle.
\end{align*}

For the algorithms, we represent the top vertex $i$ by $i$ and the
bottom vertex $i'$ by $-i$. Thus a diagram is stored as a list of
blocks
\[
[a_1,\dots,a_k],
\qquad
a_j\in\{\pm1,\dots,\pm n\}.
\]
For example, $[2,-5]$ is a transversal pair, while $[3]$ and $[-4]$
are top and bottom singleton blocks. Every input contains exactly
$2n$ signed labels. We use the order
\[
1<-1<2<-2<\cdots<n<-n.
\]
The vertices within each block are written in this order, and the blocks
themselves are then ordered lexicographically. This notation and ordering is precisely that used to represent diagrams as bipartitions in the GAP \cite{GAP} package ``Semigroups'' \cite{Semigroups-GAP}. 

We include the cost of writing the resulting word. In particular, an
algorithm returning $r$ generators takes at least $\Omega(r)$ time.

\subsection{General methodology used below}\label{S:GenMethod}

For a diagram $X$, call a block a top block if all its vertices lie in
the top row, a bottom block if they all lie in the bottom row, and a
transversal block otherwise. Our general method is to factor one block
at a time: top blocks first, transversal blocks in the middle, and bottom
blocks last. Planarity makes the required order visible. The nonplanar monoids also require permutations of the top and bottom and the partition
monoid needs one extra step because its blocks can have arbitrary size.

For the planar diagram monoids, the picture is
\begin{gather*}
\begin{tikzpicture}[anchorbase,scale=1.5]
\draw[mor] (0,0) to (0.25,0.5) to (0.75,0.5) to (1,0) to (0,0);
\node at (0.5,0.25){Bot};
\draw[mor] (0,1.5) to (0.25,1) to (0.75,1) to (1,1.5) to (0,1.5);
\node at (0.5,1.25){Top};
\draw[mor] (0.25,0.5) to (0.25,1) to (0.75,1) to (0.75,0.5) to (0.25,0.5);
\node at (0.5,0.75){Mid};
\end{tikzpicture},
\end{gather*}
while for the symmetric diagram monoids we also need permutation words
above and below:
\begin{gather*}
\begin{tikzpicture}[anchorbase,scale=1.5]
\draw[mor] (0,1.5) to (1,1.5) to (1,2) to (0,2) to (0,1.5);
\node at (0.5,1.75){Sym};
\draw[mor] (0,1.5) to (0.25,1) to (0.75,1) to (1,1.5) to (0,1.5);
\node at (0.5,1.25){Top};
\draw[mor] (0,0.5) to (0.25,1) to (0.75,1) to (1,0.5) to (0,0.5);
\node at (0.5,0.75){Bot};
\draw[mor] (0,0.5) to (1,0.5) to (1,0) to (0,0) to (0,0.5);
\node at (0.5,0.25){Sym};
\end{tikzpicture}.
\end{gather*}

\begin{Remark}
For the reader familiar with similar pictures, these do not describe (sandwich) cellular
factorizations as in \cite{Br-gen-matrix-algebras,GrLe-cellular,Tu-sandwich}. They only indicate the order in which the lists of
factors are assembled. However, the planar case suggests efficient algorithms to compute the bases of cellular algebras, for example, using the description in \cite{AnStTu-cellular-tilting}.
\end{Remark}

The resulting words need not be minimal, but their worst case lengths
have the optimal order of growth, and the algorithms are easy to
implement.

\begin{Remark}
The algorithms have been implemented in GAP; see \cite{St-github-FastFact} for the code.
\end{Remark}

\begin{Remark}
We ignore the identity in this paper, since any practical use of these algorithms will not involve the identity.
\end{Remark}

\subsection{Notation summary}

For an input diagram $X$, we write $L(X)$ for the length of the word
returned by the algorithm under consideration, and $T(X)$ for its
runtime. We include the cost of writing the output, so in particular
$T(X)\in\Omega(L(X))$.

For each of our algorithms, we write
\[
\overline L_n^{M}
=
\frac{1}{|M_n|}\sum_{X\in M_n}L(X),
\qquad
\operatorname{Avg}_n^M
=
\frac{1}{|M_n|}\sum_{X\in M_n}T(X)
\]
for its average output length and average runtime, respectively, where
$M$ denotes the relevant diagram monoid family. All averages in this
paper are taken with respect to the uniform distribution.
For a diagram monoid $M$, we write
\[
\operatorname{Opt}_n^M
=
\inf_A\max_{X\in M}T_A(X),
\]
where $A$ ranges over algorithms that explicitly return a word in the
fixed generators and $T_A(X)$ is the runtime of $A$ on $X$. Thus
$\operatorname{Opt}_n^M$ is the minimal possible worst-case runtime
for the factorization problem. When the family or algorithm is clear from context,
we simply write $\operatorname{Opt}_n$ etc.

\section{Factorization algorithms: planar case}\label{S:Planar}

This is the first main section of the paper.

\subsection{Temperley--Lieb}\label{S:TL}

First up is everyone's favorite diagram monoid. We will cover this case in full detail, then keep subsequent monoids brief since their proofs and arguments are essentially identical.

\begin{Proposition}\label{P:TLlower}
Any algorithm that factors every Temperley--Lieb diagram $X\in TL_n$ into
a product of the generators $\{e_i\}$ has worst case
complexity $\Omega(n^2)$, thus $\operatorname{Opt}_n\in\Omega(n^2)$.
\end{Proposition}

\begin{proof}
Any such algorithm must write its output. For the rainbow diagram in
\autoref{fig:TLRainbow}, let $c_r$ be the number of strings crossing the
vertical cut between columns $r$ and $r+1$.
\begin{figure}[H]
\begin{tikzpicture}[anchorbase]
\draw[usual] (0,3) to[out=270,in=180] (2,2) to[out=0,in=270] (4,3);
\draw[usual] (0.5,3) to[out=270,in=180] (2,2.25) to[out=0,in=270] (3.5,3);
\draw[usual] (0.875,3) node {$\dots$};
\draw[usual] (1.25,3) to[out=270,in=180] (2,2.5) to[out=0,in=270] (2.75,3);
\draw[usual] (1.75,3) to[out=270,in=180] (2,2.75) to[out=0,in=270] (2.25,3);
\draw[usual] (3.125,3) node {$\dots$};
\draw[usual] (0,0) to[out=90,in=180] (2,1) to[out=0,in=90] (4,0);
\draw[usual] (0.5,0) to[out=90,in=180] (2,0.75) to[out=0,in=90] (3.5,0);
\draw[usual] (0.875,0) node {$\dots$};
\draw[usual] (1.25,0) to[out=90,in=180] (2,0.5) to[out=0,in=90] (2.75,0);
\draw[usual] (1.75,0) to[out=90,in=180] (2,0.25) to[out=0,in=90] (2.25,0);
\draw[usual] (3.125,0) node {$\dots$};
\end{tikzpicture}
\quad\quad ,\quad\quad
\begin{tikzpicture}[anchorbase]
\draw[usual] (-1,3) to (4,0);
\draw[usual] (0,3) to[out=270,in=180] (2,2) to[out=0,in=270] (4,3);
\draw[usual] (0.5,3) to[out=270,in=180] (2,2.25) to[out=0,in=270] (3.5,3);
\draw[usual] (0.875,3) node {$\dots$};
\draw[usual] (1.25,3) to[out=270,in=180] (2,2.5) to[out=0,in=270] (2.75,3);
\draw[usual] (1.75,3) to[out=270,in=180] (2,2.75) to[out=0,in=270] (2.25,3);
\draw[usual] (3.125,3) node {$\dots$};
\draw[usual] (-1,0) to[out=90,in=180] (1,1) to[out=0,in=90] (3,0);
\draw[usual] (-0.5,0) to[out=90,in=180] (1,0.75) to[out=0,in=90] (2.5,0);
\draw[usual] (-0.125,0) node {$\dots$};
\draw[usual] (0.25,0) to[out=90,in=180] (1,0.5) to[out=0,in=90] (1.75,0);
\draw[usual] (0.75,0) to[out=90,in=180] (1,0.25) to[out=0,in=90] (1.25,0);
\draw[usual] (2.125,0) node {$\dots$};
\end{tikzpicture}
.
\caption{Rainbow diagrams requiring $\Omega(n^2)$ factors. Left: even $n$; right: odd $n$.}
\label{fig:TLRainbow}
\end{figure}
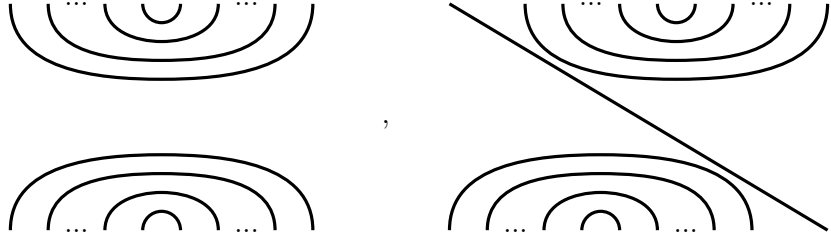
Every occurrence of $e_r$
supplies only two passages through this cut. Hence every word for the
diagram contains at least $c_r/2$ occurrences of $e_r$. Summing over
$r$ gives $n^2/4$ for even $n$ and $(n^2-1)/4$ for odd $n$; this is also
the reduced length, as follows from \cite[Corollary 4.8]{TL-factorization}.

For even $n$, the following reduced decomposition of the top half makes
the bound visible. Commuting generators are placed on the same line.
\begin{gather*}
\begin{tikzpicture}[anchorbase]
\draw[usual] (-0.5,0) to (-0.5,1);
\draw[usual] (-0.5,2) to (-0.5,3);
\draw[usual] (-0.5,3) to (-0.5,4);
\draw[usual] (0,2) to (0,3);
\draw[usual] (0,3) to (0,4);
\draw[usual] (0.5,3) to (0.5,4);
\draw[usual] (2,3) to (2,4);
\draw[usual] (2.5,2) to (2.5,3);
\draw[usual] (2.5,3) to (2.5,4);
\draw[usual] (3,0) to (3,1);
\draw[usual] (3,2) to (3,3);
\draw[usual] (3,3) to (3,4);
\draw[usual] (-0.25,3) node {$\dots$};
\draw[usual] (1.25,-0.5) node {$\dots$};
\draw[usual] (1.25,0.5) node {$\dots$};
\draw[usual] (1.25,1.75) node {$\vdots$};
\draw[usual] (2.75,3) node {$\dots$};
\draw[usual] (-0.5,-1) to[out=90,in=180] (-0.25,-0.75) to[out=0,in=90] (0,-1);
\draw[usual] (-0.5,0) to[out=270,in=180] (-0.25,-0.25) to[out=0,in=270] (0,0);
\draw[usual] (0,0) to[out=90,in=180] (0.25,0.25) to[out=0,in=90] (0.5,0);
\draw[usual] (0,1) to[out=270,in=180] (0.25,0.75) to[out=0,in=270] (0.5,1);
\draw[usual] (0.5,-1) to[out=90,in=180] (0.75,-0.75) to[out=0,in=90] (1,-1);
\draw[usual] (0.5,0) to[out=270,in=180] (0.75,-0.25) to[out=0,in=270] (1,0);
\draw[usual] (0.5,1) to[out=90,in=180] (0.75,1.25) to[out=0,in=90] (1,1);
\draw[usual] (0.5,2) to[out=90,in=180] (0.75,2.25) to[out=0,in=90] (1,2);
\draw[usual] (0.5,3) to[out=270,in=180] (0.75,2.75) to[out=0,in=270] (1,3);
\draw[usual] (1,3) to[out=90,in=180] (1.25,3.25) to[out=0,in=90] (1.5,3);
\draw[usual] (1,4) to[out=270,in=180] (1.25,3.75) to[out=0,in=270] (1.5,4);
\draw[usual] (1.5,-1) to[out=90,in=180] (1.75,-0.75) to[out=0,in=90] (2,-1);
\draw[usual] (1.5,0) to[out=270,in=180] (1.75,-0.25) to[out=0,in=270] (2,0);
\draw[usual] (1.5,1) to[out=90,in=180] (1.75,1.25) to[out=0,in=90] (2,1);
\draw[usual] (1.5,2) to[out=90,in=180] (1.75,2.25) to[out=0,in=90] (2,2);
\draw[usual] (1.5,3) to[out=270,in=180] (1.75,2.75) to[out=0,in=270] (2,3);
\draw[usual] (2,0) to[out=90,in=180] (2.25,0.25) to[out=0,in=90] (2.5,0);
\draw[usual] (2,1) to[out=270,in=180] (2.25,0.75) to[out=0,in=270] (2.5,1);
\draw[usual] (2.5,-1) to[out=90,in=180] (2.75,-0.75) to[out=0,in=90] (3,-1);
\draw[usual] (2.5,0) to[out=270,in=180] (2.75,-0.25) to[out=0,in=270] (3,0);
\draw[usual] (4,-0.5) node {$n/2$};
\draw[usual] (5,-0.5) node {factors};
\draw[usual] (4,0.5) node {$n/2-1$};
\draw[usual] (4,1.75) node {$\vdots$};
\draw[usual] (4,2.5) node {$2$};
\draw[usual] (4,3.5) node {$1$};
\end{tikzpicture}
.
\end{gather*}
The proof is complete.
\end{proof}

We now give a simpler factorization algorithm, see \autoref{alg:TLFact}. Its words need not be
reduced, but their length is still $O(n^2)$.

\begin{algorithm}
\caption{Temperley--Lieb factorization algorithm}\label{alg:TLFact}
\begin{algorithmic}
\Require {$X \in TL_n$ and the generators $e_1,\dots,e_{n-1}$}
\Function{Factorize $TL_n$}{$X$}
\State top, $F$, bot, $B$ $\gets$ empty lists
\State blocks $\gets$ blocks of $X$ in the order fixed above
\For {block $[a,b]\in$ blocks}

\If{$a>0$ and $b>0$}
\State odd $\gets$ empty list
\State even $\gets$ empty list
\For{$i=a,\dots,b-1$}
\If{$i$ and $a$ have same parity}
\State Add $e_i$ to even 
\Else
\State Add $e_i$ to odd
\EndIf
\EndFor
\State $F\gets$ concatenate $F$ and odd  
\State $F\gets$ concatenate $F$ and even
\ElsIf{$a>0$ and $b<0$}
\If{$a \neq |b|$}
\State odd $\gets$ empty list
\State even $\gets$ empty list
\If{$a>|b|$}
\For{$i=|b|,\dots,a-1$}
\If{$i$ and $a$ have same parity}
\State Add $e_i$ to even 
\Else
\State Add $e_i$ to odd
\EndIf
\EndFor
\State bot $\gets$ concatenate bot and even
\State bot $\gets$ concatenate bot and odd
\Else
\For{$i=a,\dots,|b|-1$}
\If{$i$ and $a$ have same parity}
\State Add $e_i$ to even 
\Else
\State Add $e_i$ to odd
\EndIf
\EndFor
\State top $\gets$ concatenate top and even
\State top $\gets$ concatenate top and odd
\EndIf
\Else 
\State \textbf{continue}
\EndIf
\Else
\State odd $\gets$ empty list
\State even $\gets$ empty list 
\For{$i=|a|,\dots,|b|-1$}
\If{$i$ and $|a|$ have the same parity}
\State Add $e_i$ to even 
\Else
\State Add $e_i$ to odd
\EndIf
\EndFor
\State $B\gets$ concatenate $B$ and even 
\State $B\gets$ concatenate $B$ and odd 
\EndIf
\EndFor
\State Reverse $F$
\State Reverse top
\State \textbf{return} concatenate $F$, top, bot, and $B$
\EndFunction
\end{algorithmic}
\end{algorithm}

The four lists separate cups, the two directions of transversal strings,
and caps. Alternating indices enlarge a neighboring cup or cap by one
column. Reversing $F$ and top makes inner cups and the rightmost
right-moving strings appear first. The proof of \autoref{P:TLFact} explains the algorithm in more detail.

\begin{Example}
Consider the following diagram.
\begin{gather*}
\begin{tikzpicture}[anchorbase]
\draw[usual] (0.5,2) to[out=270,in=180] (0.75,1.75) to[out=0,in=270] (1,2);
\draw[usual] (1.5,2) to (0.5,0);
\draw[usual] (2,2) to (3,0);
\draw[usual] (2.5,2) to[out=270,in=180] (2.75,1.75) to[out=0,in=270] (3,2);
\draw[usual] (1,0) to[out=90,in=180] (1.75,0.5) to[out=0,in=90] (2.5,0);
\draw[usual] (1.5,0) to[out=90,in=180] (1.75,0.25) to[out=0,in=90] (2,0);
\end{tikzpicture}
=\quad
\begin{tikzpicture}[anchorbase]
\draw[usual] (0,-1) to (0,0);
\draw[usual] (0.5,-1) to (0.5,0);
\draw[usual] (1,-1) to[out=90,in=180] (1.25,-0.75) to[out=0,in=90] (1.5,-1);
\draw[usual] (1,0) to[out=270,in=180] (1.25,-0.25) to[out=0,in=270] (1.5,0);
\draw[usual] (2,-1) to (2,0);
\draw[usual] (2.5,-1) to (2.5,0);
\draw[usual] (0,0) to (0,1);
\draw[usual] (0.5,0) to (0.5,1);
\draw[usual] (1,0) to[out=90,in=180] (1.25,0.25) to[out=0,in=90] (1.5,0);
\draw[usual] (1,1) to[out=270,in=180] (1.25,0.75) to[out=0,in=270] (1.5,1);
\draw[usual] (2,0) to (2,1);
\draw[usual] (2.5,0) to (2.5,1);
\draw[usual] (0,1) to (0,2);
\draw[usual] (0.5,1) to[out=90,in=180] (0.75,1.25) to[out=0,in=90] (1,1);
\draw[usual] (0.5,2) to[out=270,in=180] (0.75,1.75) to[out=0,in=270] (1,2);
\draw[usual] (1.5,1) to[out=90,in=180] (1.75,1.25) to[out=0,in=90] (2,1);
\draw[usual] (1.5,2) to[out=270,in=180] (1.75,1.75) to[out=0,in=270] (2,2);
\draw[usual] (2.5,1) to (2.5,2);
\draw[usual] (0,2) to (0,3);
\draw[usual] (0.5,3) to[out=270,in=180] (0.75,2.75) to[out=0,in=270] (1,3);
\draw[usual] (0.5,2) to[out=90,in=180] (0.75,2.25) to[out=0,in=90] (1,2);
\draw[usual] (1.5,3) to[out=270,in=180] (1.75,2.75) to[out=0,in=270] (2,3);
\draw[usual] (1.5,2) to[out=90,in=180] (1.75,2.25) to[out=0,in=90] (2,2);
\draw[usual] (2.5,2) to (2.5,3);
\draw[usual] (0,4) to[out=270,in=180] (0.25,3.75) to[out=0,in=270] (0.5,4);
\draw[usual] (0,3) to[out=90,in=180] (0.25,3.25) to[out=0,in=90] (0.5,3);
\draw[usual] (1,3) to (1,4);
\draw[usual] (1.5,3) to (1.5,4);
\draw[usual] (2,4) to[out=270,in=180] (2.25,3.75) to[out=0,in=270] (2.5,4);
\draw[usual] (2,3) to[out=90,in=180] (2.25,3.25) to[out=0,in=90] (2.5,3);
\draw[usual] (0,5) to[out=270,in=180] (0.25,4.75) to[out=0,in=270] (0.5,5);
\draw[usual] (0,4) to[out=90,in=180] (0.25,4.25) to[out=0,in=90] (0.5,4);
\draw[usual] (1,4) to (1,5);
\draw[usual] (1.5,4) to (1.5,5);
\draw[usual] (2,5) to[out=270,in=180] (2.25,4.75) to[out=0,in=270] (2.5,5);
\draw[usual] (2,4) to[out=90,in=180] (2.25,4.25) to[out=0,in=90] (2.5,4);
\end{tikzpicture}
\end{gather*}
$=(e_5)(e_1)(e_5e_4)(e_1e_2)(e_2e_4e_3)(e_3)$, where parentheses
group the contributions from different blocks. Commuting generators are
drawn on the same line.
\end{Example}

\begin{Proposition}\label{P:TLFact}
\autoref{alg:TLFact} factors every $X\in TL_n$ and runs in
$O(n+L(X))$, where $L(X)$ is the length of the returned word.
\end{Proposition}

\begin{proof}
There are $n$ blocks and $2n$ signed labels. The algorithm reads each
once and writes each generator once, which gives the stated bound.

For correctness, follow a string through the product. The alternating
indices in the local list for $[a,b]$ move one endpoint through the
columns between $|a|$ and $|b|$, and hence contribute exactly
$\lvert |a|-|b|\rvert$ generators. The only issue is the order of
different blocks. Planarity leaves four possibilities, shown below.

\begin{enumerate}
\item Cups are processed from the inside out. Disjoint cups commute.
\begin{gather*}
\begin{tikzpicture}[anchorbase]
\draw[usual] (0,3) to[out=270,in=180] (2,2) to[out=0,in=270] (4,3);
\draw[usual] (0.5,3) to[out=270,in=180] (2,2.25) to[out=0,in=270] (3.5,3);
\draw[usual] (0.875,3) node {$\dots$};
\draw[usual] (1.25,3) to[out=270,in=180] (2,2.5) to[out=0,in=270] (2.75,3);
\draw[usual] (1.75,3) to[out=270,in=180] (2,2.75) to[out=0,in=270] (2.25,3);
\draw[usual] (3.125,3) node {$\dots$};
\draw[usual] (2,1) node {$\dots$};
\end{tikzpicture}
\quad = \quad
\begin{tikzpicture}[anchorbase]
\draw[usual] (-0.5,-1) to (-0.5,0);
\draw[usual] (-0.5,0) to (-0.5,1);
\draw[usual] (-0.5,2) to (-0.5,3);
\draw[usual] (-0.5,3) to (-0.5,4);
\draw[usual] (-0.5,4) to (-0.5,5);
\draw[usual] (0,2) to (0,3);
\draw[usual] (0,3) to (0,4);
\draw[usual] (0,4) to (0,5);
\draw[usual] (0.5,3) to (0.5,4);
\draw[usual] (0.5,4) to (0.5,5);
\draw[usual] (2,3) to (2,4);
\draw[usual] (2,4) to (2,5);
\draw[usual] (2.5,2) to (2.5,3);
\draw[usual] (2.5,3) to (2.5,4);
\draw[usual] (2.5,4) to (2.5,5);
\draw[usual] (3,-1) to (3,0);
\draw[usual] (3,0) to (3,1);
\draw[usual] (3,2) to (3,3);
\draw[usual] (3,3) to (3,4);
\draw[usual] (3,4) to (3,5);
\draw[usual] (-0.25,3) node {$\dots$};
\draw[usual] (1.25,-4) node {$\dots$};
\draw[usual] (1.25,-1.5) node {$\dots$};
\draw[usual] (1.25,0) node {$\dots$};
\draw[usual] (1.25,1.75) node {$\vdots$};
\draw[usual] (2.75,3) node {$\dots$};
\draw[usual] (-0.5,-3) to[out=90,in=180] (-0.25,-2.75) to[out=0,in=90] (0,-3);
\draw[usual] (-0.5,-2) to[out=270,in=180] (-0.25,-2.25) to[out=0,in=270] (0,-2);
\draw[usual] (-0.5,-2) to[out=90,in=180] (-0.25,-1.75) to[out=0,in=90] (0,-2);
\draw[usual] (-0.5,-1) to[out=270,in=180] (-0.25,-1.25) to[out=0,in=270] (0,-1);
\draw[usual] (0,-1) to[out=90,in=180] (0.25,-0.75) to[out=0,in=90] (0.5,-1);
\draw[usual] (0,0) to[out=270,in=180] (0.25,-0.25) to[out=0,in=270] (0.5,0);
\draw[usual] (0,0) to[out=90,in=180] (0.25,0.25) to[out=0,in=90] (0.5,0);
\draw[usual] (0,1) to[out=270,in=180] (0.25,0.75) to[out=0,in=270] (0.5,1);
\draw[usual] (0.5,-3) to[out=90,in=180] (0.75,-2.75) to[out=0,in=90] (1,-3);
\draw[usual] (0.5,-2) to[out=270,in=180] (0.75,-2.25) to[out=0,in=270] (1,-2);
\draw[usual] (0.5,-2) to[out=90,in=180] (0.75,-1.75) to[out=0,in=90] (1,-2);
\draw[usual] (0.5,-1) to[out=270,in=180] (0.75,-1.25) to[out=0,in=270] (1,-1);
\draw[usual] (0.5,1) to[out=90,in=180] (0.75,1.25) to[out=0,in=90] (1,1);
\draw[usual] (0.5,2) to[out=90,in=180] (0.75,2.25) to[out=0,in=90] (1,2);
\draw[usual] (0.5,3) to[out=270,in=180] (0.75,2.75) to[out=0,in=270] (1,3);
\draw[usual] (1,3) to[out=90,in=180] (1.25,3.25) to[out=0,in=90] (1.5,3);
\draw[usual] (1,4) to[out=270,in=180] (1.25,3.75) to[out=0,in=270] (1.5,4);
\draw[usual] (1,4) to[out=90,in=180] (1.25,4.25) to[out=0,in=90] (1.5,4);
\draw[usual] (1,5) to[out=270,in=180] (1.25,4.75) to[out=0,in=270] (1.5,5);
\draw[usual] (1.5,-3) to[out=90,in=180] (1.75,-2.75) to[out=0,in=90] (2,-3);
\draw[usual] (1.5,-2) to[out=270,in=180] (1.75,-2.25) to[out=0,in=270] (2,-2);
\draw[usual] (1.5,-2) to[out=90,in=180] (1.75,-1.75) to[out=0,in=90] (2,-2);
\draw[usual] (1.5,-1) to[out=270,in=180] (1.75,-1.25) to[out=0,in=270] (2,-1);
\draw[usual] (1.5,1) to[out=90,in=180] (1.75,1.25) to[out=0,in=90] (2,1);
\draw[usual] (1.5,2) to[out=90,in=180] (1.75,2.25) to[out=0,in=90] (2,2);
\draw[usual] (1.5,3) to[out=270,in=180] (1.75,2.75) to[out=0,in=270] (2,3);
\draw[usual] (2,-1) to[out=90,in=180] (2.25,-0.75) to[out=0,in=90] (2.5,-1);
\draw[usual] (2,0) to[out=270,in=180] (2.25,-0.25) to[out=0,in=270] (2.5,0);
\draw[usual] (2,0) to[out=90,in=180] (2.25,0.25) to[out=0,in=90] (2.5,0);
\draw[usual] (2,1) to[out=270,in=180] (2.25,0.75) to[out=0,in=270] (2.5,1);
\draw[usual] (2.5,-3) to[out=90,in=180] (2.75,-2.75) to[out=0,in=90] (3,-3);
\draw[usual] (2.5,-2) to[out=270,in=180] (2.75,-2.25) to[out=0,in=270] (3,-2);
\draw[usual] (2.5,-2) to[out=90,in=180] (2.75,-1.75) to[out=0,in=90] (3,-2);
\draw[usual] (2.5,-1) to[out=270,in=180] (2.75,-1.25) to[out=0,in=270] (3,-1);
\end{tikzpicture}
.
\end{gather*}

\item Transversal strings running down and to the right are processed from right to left.
\begin{gather*}
\begin{tikzpicture}[anchorbase]
\draw[usual] (0,1) to (1,0);
\draw[usual] (0.5,1) to (1.5,0);
\draw[usual] (1.5,1) node {$\dots$};
\draw[usual] (0,0) node {$\dots$};
\end{tikzpicture}
\quad =\quad
\begin{tikzpicture}[anchorbase]
\draw[usual] (0,1) to[out=270,in=180] (0.25,0.75) to[out=0,in=270] (0.5,1);
\draw[usual] (0,0) to[out=90,in=180] (0.25,0.25) to[out=0,in=90] (0.5,0);
\draw[usual] (1,1) to[out=270,in=180] (1.25,0.75) to[out=0,in=270] (1.5,1);
\draw[usual] (1,0) to[out=90,in=180] (1.25,0.25) to[out=0,in=90] (1.5,0);
\draw[usual] (2.5,1) to[out=270,in=180] (2.75,0.75) to[out=0,in=270] (3,1);
\draw[usual] (2.5,0) to[out=90,in=180] (2.75,0.25) to[out=0,in=90] (3,0);
\draw[usual] (3.5,0) to (3.5,1);
\draw[usual] (4,0) to (4,1);
\draw[usual] (0,1) to (0,2);
\draw[usual] (0.5,1) to[out=90,in=180] (0.75,1.25) to[out=0,in=90] (1,1);
\draw[usual] (0.5,2) to[out=270,in=180] (0.75,1.75) to[out=0,in=270] (1,2);
\draw[usual] (3,1) to[out=90,in=180] (3.25,1.25) to[out=0,in=90] (3.5,1);
\draw[usual] (3,2) to[out=270,in=180] (3.25,1.75) to[out=0,in=270] (3.5,2);
\draw[usual] (4,1) to (4,2);
\draw[usual] (2,2) node {$\dots$};
\draw[usual] (0,2) to (0,3);
\draw[usual] (0.5,2) to[out=90,in=180] (0.75,2.25) to[out=0,in=90] (1,2);
\draw[usual] (0.5,3) to[out=270,in=180] (0.75,2.75) to[out=0,in=270] (1,3);
\draw[usual] (3,2) to[out=90,in=180] (3.25,2.25) to[out=0,in=90] (3.5,2);
\draw[usual] (3,3) to[out=270,in=180] (3.25,2.75) to[out=0,in=270] (3.5,3);
\draw[usual] (4,2) to (4,3);
\draw[usual] (0,3) to (0,4);
\draw[usual] (0.5,3) to (0.5,4);
\draw[usual] (1,3) to[out=90,in=180] (1.25,3.25) to[out=0,in=90] (1.5,3);
\draw[usual] (1,4) to[out=270,in=180] (1.25,3.75) to[out=0,in=270] (1.5,4);
\draw[usual] (2.5,3) to[out=90,in=180] (2.75,3.25) to[out=0,in=90] (3,3);
\draw[usual] (2.5,4) to[out=270,in=180] (2.75,3.75) to[out=0,in=270] (3,4);
\draw[usual] (3.5,3) to[out=90,in=180] (3.75,3.25) to[out=0,in=90] (4,3);
\draw[usual] (3.5,4) to[out=270,in=180] (3.75,3.75) to[out=0,in=270] (4,4);
\end{tikzpicture}
.
\end{gather*}
\item Transversal strings running down and to the left are processed from left to right.
\begin{gather*}
\begin{tikzpicture}[anchorbase]
\draw[usual] (0,0) to (1,1);
\draw[usual] (0.5,0) to (1.5,1);
\draw[usual] (1.5,0) node {$\dots$};
\draw[usual] (0,1) node {$\dots$};
\end{tikzpicture}
\quad =\quad
\begin{tikzpicture}[anchorbase]
\draw[usual] (4,1) to[out=270,in=0] (3.75,0.75) to[out=180,in=270] (3.5,1);
\draw[usual] (4,0) to[out=90,in=0] (3.75,0.25) to[out=180,in=90] (3.5,0);
\draw[usual] (3,1) to[out=270,in=0] (2.75,0.75) to[out=180,in=270] (2.5,1);
\draw[usual] (3,0) to[out=90,in=0] (2.75,0.25) to[out=180,in=90] (2.5,0);
\draw[usual] (1.5,1) to[out=270,in=0] (1.25,0.75) to[out=180,in=270] (1,1);
\draw[usual] (1.5,0) to[out=90,in=0] (1.25,0.25) to[out=180,in=90] (1,0);
\draw[usual] (0.5,0) to (0.5,1);
\draw[usual] (0,0) to (0,1);
\draw[usual] (4,1) to (4,2);
\draw[usual] (3.5,1) to[out=90,in=0] (3.25,1.25) to[out=180,in=90] (3,1);
\draw[usual] (3.5,2) to[out=270,in=0] (3.25,1.75) to[out=180,in=270] (3,2);
\draw[usual] (1,1) to[out=90,in=0] (0.75,1.25) to[out=180,in=90] (0.5,1);
\draw[usual] (1,2) to[out=270,in=0] (0.75,1.75) to[out=180,in=270] (0.5,2);
\draw[usual] (0,1) to (0,2);
\draw[usual] (2,2) node {$\dots$};
\draw[usual] (4,2) to (4,3);
\draw[usual] (3.5,2) to[out=90,in=0] (3.25,2.25) to[out=180,in=90] (3,2);
\draw[usual] (3.5,3) to[out=270,in=0] (3.25,2.75) to[out=180,in=270] (3,3);
\draw[usual] (0.5,2) to[out=90,in=180] (0.75,2.25) to[out=0,in=90] (1,2);
\draw[usual] (0.5,3) to[out=270,in=180] (0.75,2.75) to[out=0,in=270] (1,3);
\draw[usual] (0,2) to (0,3);
\draw[usual] (4,3) to (4,4);
\draw[usual] (3.5,3) to (3.5,4);
\draw[usual] (3,3) to[out=90,in=0] (2.75,3.25) to[out=180,in=90] (2.5,3);
\draw[usual] (3,4) to[out=270,in=0] (2.75,3.75) to[out=180,in=270] (2.5,4);
\draw[usual] (1.5,4) to[out=270,in=0] (1.25,3.75) to[out=180,in=270] (1,4);
\draw[usual] (0.5,3) to[out=90,in=0] (0.25,3.25) to[out=180,in=90] (0,3);
\draw[usual] (0.5,4) to[out=270,in=0] (0.25,3.75) to[out=180,in=270] (0,4);
\draw[usual] (1,3) to[out=90,in=180] (1.25,3.25) to[out=0,in=90] (1.5,3);
\end{tikzpicture}
.
\end{gather*}

\item Caps are processed from the outside in.
\end{enumerate}

In each case the displayed order prevents a later block from changing an
endpoint already in place. The product is therefore $X$.
\end{proof}

Thus
\[
L(X)=\sum_{[a,b]\in X}\big\lvert |a|-|b|\big\rvert .
\]
This statistic is simpler than the reduced length used in the lower bound.

\begin{Proposition}
The longest factorization produced by \autoref{alg:TLFact} is attained by
the rainbow diagram in \autoref{fig:TLRainbow} and has length
\[
\left\lfloor\frac{n^2}{2}\right\rfloor.
\]
\end{Proposition}

\begin{proof}
Let $c_r(X)$ be the number of blocks crossing the vertical cut between
columns $r$ and $r+1$. In the following example, two blocks cross the
cut, so $c_r(X)=2$.
\begin{gather*}
\begin{tikzpicture}[anchorbase]
\draw[usual,dashed] (2.5,0) to (2.5,2); 
\draw[usual] (1,2) to (3.5,0);
\draw[usual] (0,2) node [above] {$1$};
\draw[usual] (0.5,2) node [above] {$2$};
\draw[usual] (1,2) node [above] {$\dots$};
\draw[usual] (4,2) node [above] {$\dots$};
\draw[usual] (5,2) node [above] {$n$};
\draw[usual] (1.5,2) to[out=270,in=180] (2.75,1.5) to[out=0,in=270] (4,2);
\draw[usual] (2,2) node [above] {$r$};
\draw[usual] (3.25,2) node [above] {$r+1$};
\draw[usual] (0.5,2) to (1.5,0);
\end{tikzpicture}\quad.
\end{gather*} 
Double counting block--cut incidences gives
\[
L(X)=\sum_{[a,b]\in X}\big\lvert |a|-|b|\big\rvert
=\sum_{r=1}^{n-1}c_r(X).
\] 
There are $2r$ vertices to the left of the cut and $2(n-r)$ to the
right, hence
\[
L(X)\leq 2\sum_{r=1}^{n-1}\min(r,n-r)
=\left\lfloor\frac{n^2}{2}\right\rfloor.
\] 
The rainbow diagram attains equality at every cut.
\end{proof}
Thus \autoref{alg:TLFact} has worst case complexity $\Theta(n^2)$.
\begin{Proposition}\label{P:TLAvg}
The average output length of \autoref{alg:TLFact} is
\[
\overline L_n^{TL}
=\frac{(n+1)\left(2^{2n-1}-\binom{2n}{n}\right)}
{\binom{2n}{n}}.
\]
With the step count \(T(X)=2n+2L(X)\), the average runtime is
\[
\operatorname{Avg}_n
=\frac{2^{2n}(n+1)}{\binom{2n}{n}}-2
\sim\sqrt{\pi}\cdot n^{3/2}.
\]
\end{Proposition}

\begin{proof}
We use the well-known relation between Dyck path and Temperley--Lieb diagrams, see e.g. \cite{SW-constructive}. (There are several equivalent conventions for this correspondence.)

For a nonidentity diagram, the preceding proof gives
\(L(X)=\sum_{r=1}^{n-1}c_r(X)\). Encode \(X\) by reading
\(1,-1,2,-2,\dots,n,-n\): the first endpoint of each block gives a
right step and the second gives an up step. This is a Dyck path from
\((0,0)\) to \((n,n)\). For example,
\begin{gather*}
\begin{tikzpicture}[anchorbase]
\draw[usual] (0,0) to (0,1);
\draw[usual] (0.5,1) to[out=270,in=180] (0.75,0.75) to[out=0,in=270] (1,1);
\draw[usual] (0.5,0) to (1.5,1);
\draw[usual] (1,0) to[out=90,in=180] (1.25,0.25) to[out=0,in=90] (1.5,0);
\end{tikzpicture}
\quad \longleftrightarrow
\begin{tikzpicture}[anchorbase]
\draw[usual,ultra thin] (0,0) to (2,0);
\draw[usual,ultra thin] (0,0) to (0,2);
\draw[usual,ultra thick] (0,0) to (0.5,0);
\draw[usual,ultra thick] (0.5,0) to (0.5,0.5);
\draw[usual,ultra thick] (0.5,0.5) to (1,0.5);
\draw[usual,ultra thick] (1,0.5) to (1.5,0.5);
\draw[usual,ultra thick] (1.5,0.5) to (1.5,1);
\draw[usual,ultra thick] (1.5,1) to (2,1);
\draw[usual,ultra thick] (2,1) to (2,1.5);
\draw[usual,ultra thick] (2,1.5) to (2,2);
\draw[usual] (0,0) node [below] {$(0,0)$};
\draw[usual] (2,2) node [above] {$(2,2)$};
\draw[usual,dotted] (0.5,0) to (0.5,2);
\draw[usual,dotted] (1,0) to (1,2);
\draw[usual,dotted] (1.5,0) to (1.5,2);
\draw[usual,dotted] (2,0) to (2,2);
\draw[usual,dotted] (0,0.5) to (2,0.5);
\draw[usual,dotted] (0,1) to (2,1);
\draw[usual,dotted] (0,1.5) to (2,1.5);
\draw[usual,dotted] (0,2) to (2,2);
\end{tikzpicture}
.
\end{gather*}
The height after \(2r\) steps is \(c_r(X)\). If this height is \(2j\),
the reflection principle gives
\[
d_{r,j}=\binom{2r}{r-j}-\binom{2r}{r-j-1}
\]
choices for the first \(2r\) steps and \(d_{n-r,j}\) choices for the
rest. Hence
\[
\sum_{X\in TL_n}L(X)
=\sum_{j\geq1}\sum_{r=j}^{n-j}2j\,d_{r,j}d_{n-r,j}.
\]
The Rothe--Hagen convolution as in, e.g., \cite[(5.63), Table 202]{GKP-Concrete} gives
\[
\sum_{r=j}^{n-j}d_{r,j}d_{n-r,j}
=\frac{2j+1}{n+1}\binom{2n+2}{n-2j}.
\]
Using
\[
\frac{2j+1}{n+1}\binom{2n+2}{n-2j}
=\binom{2n+1}{n-2j}-\binom{2n+1}{n-2j-1}
\]
and summing by parts now yields
\[
\sum_{X\in TL_n}L(X)=2^{2n-1}-\binom{2n}{n}.
\]
There are \(C_n=\frac1{n+1}\binom{2n}{n}\) diagrams. Divide by \(C_n\) to obtain \(\overline L_n^{TL}\).
The runtime formula follows from the chosen step count, and Stirling's
formula gives the asymptotic.
\end{proof}

\subsection{Planar Rook}

Next, the planar rook monoid.

\begin{Proposition}\label{P:pRolower}
Any algorithm that factors every planar rook diagram into the chosen
generators has worst case complexity \(\Omega(n^2)\), where
$\{r_i,l_i\}$
is the generating set. Thus $\operatorname{Opt}_n\in\Omega(n^2)$.
\end{Proposition}

\begin{proof}
For $X\in pRo_n$, put
\[
D(X)=\sum_{[a,-b]\in X}|a-b|.
\]
Multiplication by one generator moves at most one surviving string by one
column, so every word for $X$ has length at least $D(X)$. The following
family of shift diagrams \autoref{Eq:shift} has a positive proportion of its strings displaced by a positive
proportion of $n$ columns.
\begin{gather}\label{Eq:shift}
\begin{tikzpicture}[anchorbase]
\draw[usual,dot] (0,0) to (0,0.2);
\draw[usual,dot] (0.5,0) to (0.5,0.2);
\draw[usual] (1,0) node {$\dots$};
\draw[usual,dot] (1.5,0) to (1.5,0.2);
\draw[usual] (0,2) to (2,0);
\draw[usual] (0.5,2) to (2.5,0);
\draw[usual] (2,0) node [below] {$\lfloor\frac{n}{2}\rfloor+3$};
\draw[usual] (1.875,1) node {$\dots$}; 
\draw[usual] (1.25,2) to (3.25,0);
\draw[usual,dot] (2,2) to (2,1.8);
\draw[usual,dot] (2.5,2) to (2.5,1.8);
\draw[usual] (2.875,2) node {$\dots$};
\draw[usual,dot] (3.25,2) to (3.25,1.8);
\end{tikzpicture}
\quad=\quad
\begin{tikzpicture}[anchorbase]
\draw[usual] (0.25,2.375) node {$\vdots$};
\draw[usual] (1,2.375) node {$\dots$};
\draw[usual,dotted] (2,0.5) to (2.125,0.25);
\draw[usual] (1.5,-2.5) node {$\vdots$};
\draw[usual] (1,-3.5) node {$\dots$};
\draw[usual] (2.75,3) node {$\dots$};
\draw[usual] (2.875,2) node {$\vdots$};
\draw[usual] (2.75,-3.5) node {$\dots$};
\draw[usual] (0,2.5) to (0,3.5);
\draw[usual] (0,1) to (0,2);
\draw[usual] (0,0) to (0,1);
\draw[usual] (0,0) to (0.5,-1);
\draw[usual,dot] (0.5,0) to (0.5,-0.2);
\draw[usual,dot] (0,-1) to (0,-0.8); 
\draw[usual] (0.5,-1) to (1,-2);
\draw[usual] (0,-1) to (0,-2);
\draw[usual,dot] (1,-1) to (1,-1.2);
\draw[usual,dot] (0.5,-2) to (0.5,-1.8); 
\draw[usual] (0,-3) to (0,-4);
\draw[usual] (0.5,-3) to (0.5,-4);
\draw[usual] (1.5,-3) to (2,-4);
\draw[usual,dot] (1.5,-4) to (1.5,-3.8);
\draw[usual,dot] (2,-3) to (2,-3.2);
\draw[usual] (0.5,2.5) to (0.5,3.5); 
\draw[usual] (0.5,1) to (0.5,2);
\draw[usual] (0.5,1) to (1,0);
\draw[usual,dot] (0.5,0) to (0.5,0.2);
\draw[usual,dot] (1,1) to (1,0.8);
\draw[usual] (1,0) to (1.5,-1);
\draw[usual,dot] (1,-1) to (1,-0.8);
\draw[usual,dot] (1.5,0) to (1.5,-0.2);
\draw[usual] (1.5,-1) to (2,-2);
\draw[usual,dot] (1.5,-2) to (1.5,-1.8);
\draw[usual,dot] (2,-1) to (2,-1.2);
\draw[usual] (2.5,-3) to (2.5,-4);
\draw[usual] (1.5,3.5) to (2,2.5);
\draw[usual,dot] (1.5,2.5) to (1.5,2.7);
\draw[usual,dot] (2,3.5) to (2,3.3);
\draw[usual] (2.5,2.5) to (2.5,3.5);
\draw[usual] (3,2.5) to (3,3.5);
\draw[usual] (2,2.5) to (2.5,1.5);
\draw[usual,dot] (2,1.5) to (2,1.7);
\draw[usual,dot] (2.5,2.5) to (2.5,2.3);
\draw[usual] (3,-1) to (3,-2);
\draw[usual] (3,-3) to (3,-4);
\end{tikzpicture}
.
\end{gather}
Thus $D(X)\in\Omega(n^2)$, and writing the output already takes quadratic time.
\end{proof}

The next algorithm uses the same block order as the Temperley--Lieb case.

\begin{algorithm}
\caption{Planar rook factorization algorithm}\label{alg:pRoFact}
\begin{algorithmic}
\Require {$X\in pRo_n$ and the generators $r_1,\dots,r_{n-1},l_1,\dots,l_{n-1}$}
\Function{Factorize $pRo_n$}{$X$}
\State top, $F$, bot, $B$ $\gets$ empty lists
\State blocks $\gets$ blocks of $X$ in the order fixed above
\For {block $\in$ blocks}
\If{block is a singleton $[a]$}
\If{$a > 0$}
\If{$a = n$}
\State Add $l_{n-1}$ to $F$
\State Add $r_{n-1}$ to $F$
\Else
\State Add $r_a$ to $F$
\State Add $l_a$ to $F$
\EndIf
\ElsIf{$a<0$}
\If{$a = -n$}
\State Add $l_{n-1}$ to $B$
\State Add $r_{n-1}$ to $B$
\Else
\State Add $r_{|a|}$ to $B$
\State Add $l_{|a|}$ to $B$
\EndIf
\EndIf
\ElsIf{block is a through strand $[a,b]$}
\If{$a \neq |b|$}
\State fact $\gets$ empty list
\If{$a>|b|$}
\For{$i=|b|,\dots,a-1$}
\State Add $r_i$ to fact
\EndFor
\State Reverse fact
\State bot $\gets$ concatenate bot and fact
\Else
\For{$i=a,\dots,|b|-1$}
\State Add $l_i$ to fact
\EndFor
\State Reverse fact
\State top $\gets$ concatenate top and fact
\EndIf
\Else 
\State \textbf{continue}
\EndIf
\EndIf
\EndFor
\State Reverse top
\State \textbf{return} concatenate $F$, top, bot, and $B$
\EndFunction
\end{algorithmic}
\end{algorithm}

In words, \autoref{alg:pRoFact} works as follows: for each dot, we use a factor that consists of two dots vertically aligned and vertical through strands everywhere else, which itself is then factored into generators like so:
\begin{gather*}
\begin{tikzpicture}[anchorbase]
\draw[usual] (0,0) to (0,1);
\draw[usual] (0.5,0.5) node {$\dots$};
\draw[usual] (1,0) to (1,1);
\draw[usual,dot] (1.5,0) to (1.5,0.2);
\draw[usual,dot] (1.5,1) to (1.5,0.8);
\draw[usual] (2,0) to (2,1);
\draw[usual] (2.5,0.5) node {$\dots$};
\draw[usual] (3,0) to (3,1);
\end{tikzpicture}
\quad=\quad
\begin{tikzpicture}[anchorbase]
\draw[usual] (0.5,0) node {$\dots$};
\draw[usual] (0,0) to (0,1);
\draw[usual] (1,0) to (1,1);
\draw[usual] (0,-1) to (0,0);
\draw[usual] (1,-1) to (1,0);
\draw[usual,dot] (1.5,1) to (1.5,0.8);
\draw[usual] (1.5,0) to (2,1);
\draw[usual,dot] (2,0) to (2,0.2);
\draw[usual] (1.5,0) to (2,-1);
\draw[usual,dot] (2,0) to (2,-0.2);
\draw[usual,dot] (1.5,-1) to (1.5,-0.8);
\draw[usual] (2.5,0) node {$\dots$};
\draw[usual] (3,-1) to (3,1);
\end{tikzpicture}
\quad ;\quad
\begin{tikzpicture}[anchorbase]
\draw[usual] (0,0) to (0,1);
\draw[usual] (0.5,0.5) node {$\dots$};
\draw[usual] (1,0) to (1,1);
\draw[usual,dot] (1.5,0) to (1.5,0.2);
\draw[usual,dot] (1.5,1) to (1.5,0.8);
\end{tikzpicture}
\quad =\quad 
\begin{tikzpicture}[anchorbase]
\draw[usual] (0.5,0) node {$\dots$};
\draw[usual] (0,-1) to (0,1);
\draw[usual] (1.5,0) to (1,1);
\draw[usual] (1.5,0) to (1,-1);
\draw[usual,dot] (1,0) to (1,0.2);
\draw[usual,dot] (1,0) to (1,-0.2);
\draw[usual,dot] (1.5,1) to (1.5,0.8);
\draw[usual,dot] (1.5,-1) to (1.5,-0.8);
\end{tikzpicture}
.
\end{gather*}
For a singleton in column \(i\), the algorithm uses
\[
d_i=r_il_i\quad(i<n),\qquad d_n=l_{n-1}r_{n-1}.
\]
A string from \(a\) to \(b>a\) contributes
\(l_a l_{a+1}\cdots l_{b-1}\); a string from \(a\) to \(b<a\)
contributes \(r_{a-1}r_{a-2}\cdots r_b\). The reversals in the
pseudocode place these words in the order forced by planarity.

\begin{Proposition}
\autoref{alg:pRoFact} factors every \(X\in pRo_n\) and runs in
\(O(n+L(X))\), where \(L(X)\) is the length of its output.
\end{Proposition}

\begin{proof}
The singleton identities above give the required top and bottom dots.
The two displayed movement words give the required transversal strings.
Since the partial bijection is order preserving, right-moving strings
must be processed from right to left and left-moving strings from left to
right. This is precisely the order in \autoref{alg:pRoFact}. The algorithm
reads \(2n\) labels and writes \(L(X)\) generators, and the result follows.
\end{proof}

If \(X\) has \(k\) strings, then
\[
L(X)=4(n-k)+\sum_{[a,-b]\in X}|a-b|.
\]

\begin{Proposition}
For \(n\geq3\), the maximal output length of \autoref{alg:pRoFact} is
\[
\begin{cases}
\left(\dfrac n2+2\right)^2,&n\text{ even},\\[2mm]
\dfrac{(n+3)(n+5)}4,&n\text{ odd}.
\end{cases}
\]
For \(n=2\), the maximum is \(8\).
\end{Proposition}

\begin{proof}
Write the top and bottom endpoints as
\(a_1<\cdots<a_k\) and \(b_1<\cdots<b_k\). Then
\[
|a_i-b_i|\leq n-k
\]
for every \(i\). Consequently
\[
L(X)\leq (n-k)(k+4).
\]
Equality is attained by taking \(k\) consecutive strings shifted as far
as possible, as in the following picture and similarly as in \autoref{Eq:shift}.
\begin{gather*}
\begin{tikzpicture}[anchorbase]
\draw[usual,dot] (0,0) to (0,0.2);
\draw[usual,dot] (0.5,0) to (0.5,0.2);
\draw[usual] (1,0) node {$\dots$};
\draw[usual,dot] (1.5,0) to (1.5,0.2);
\draw[usual] (0,2) to (2,0);
\draw[usual] (0.5,2) to (2.5,0);
\draw[usual] (2,0) node [below] {$i$};
\draw[usual] (1.875,1) node {$\dots$};
\draw[usual] (1.25,2) to (3.25,0);
\draw[usual,dot] (2,2) to (2,1.8);
\draw[usual,dot] (2.5,2) to (2.5,1.8);
\draw[usual] (2.875,2) node {$\dots$};
\draw[usual] (3.25,0) node [below] {$j\leq n$};
\draw[usual] (3.75,0) node {$\dots$};
\end{tikzpicture}
.
\end{gather*}
Maximizing the quadratic \((n-k)(k+4)\) over integers \(0\leq k\leq n\)
gives the formula.
\end{proof}

Thus \autoref{alg:pRoFact} has worst case complexity \(\Theta(n^2)\).

\begin{Proposition}
The average output length of \autoref{alg:pRoFact} is
\[
\overline L_n^{pRo}
=2n+\frac{2^{2n-3}(n-1)}{\binom{2n}{n}}.
\]
With the step count \(T(X)=2n+2L(X)\), the average runtime is
\[
\operatorname{Avg}_n
=6n+\frac{2^{2n-2}(n-1)}{\binom{2n}{n}}
\sim\frac{\sqrt{\pi}}4\cdot n^{3/2}.
\]
\end{Proposition}

\begin{proof}
There are \(\binom{2n}{n}\) diagrams. A rank \(k\) diagram is obtained
by choosing its \(k\) top and \(k\) bottom endpoints, so Vandermonde's
identity gives average rank
\[
\frac{\sum_k k\binom nk^2}{\binom{2n}{n}}=\frac n2.
\]
The average contribution from dots is therefore \(2n\).

For the strings, cut between columns \(r\) and \(r+1\), and let
\(C_r(X)\) be the number of top endpoints to the left of the cut minus
the number of bottom endpoints there. Since the strings are order
preserving, \(|C_r(X)|\) is exactly the number crossing the cut. If
\(C_r(X)=j\), Vandermonde's identity gives
\[
\binom{2r}{r+j}\binom{2(n-r)}{n-r-j}
\]
diagrams. Hence the total displacement over all diagrams is
\[
\sum_{r=1}^{n-1}\sum_j |j|
\binom{2r}{r+j}\binom{2(n-r)}{n-r-j}.
\]
A second use of Vandermonde, after separating \(j>0\) and \(j<0\), gives
\[
\sum_{r=1}^{n-1}\frac{(n-r+1)(r+1)}n
\binom{2r}{r+1}\binom{2(n-r)}{n-r-1}
=2^{2n-3}(n-1).
\]
Divide by \(\binom{2n}{n}\).
The runtime formula follows from the step count, and the asymptotic from
the central binomial coefficient.
\end{proof}

\subsection{Motzkin}

The Motzkin algorithm combines the preceding two constructions. This is
close in spirit to the standard decomposition into right planar rook,
Temperley--Lieb, and left planar rook parts, cf. \autoref{S:Relation}, but here we keep the
blockwise procedure because its output is easy to count.

\begin{Proposition}
Every factorization algorithm for \(Mo_n\) over the generators
\(\{e_i,r_i,l_i\}\) has worst case output length \(\Omega(n^2)\). Thus $\operatorname{Opt}_n\in\Omega(n^2)$.
\end{Proposition}

\begin{proof}
Take the Temperley--Lieb rainbow diagram \autoref{fig:TLRainbow}, with nested cups and caps and,
when \(n\) is odd, one central transversal string. Its edges cross the
vertical cuts between adjacent columns a total of
\(n^2/2\) times for even \(n\) and \((n^2-1)/2\) times for odd \(n\).
In a stacked word, every such crossing must occur in an elementary
layer. A factor \(e_i\) supplies at most two crossings and a factor
\(r_i\) or \(l_i\) at most one. Consequently every word for this
diagram has at least \(\lfloor n^2/4\rfloor\) factors.
\end{proof}

We first make the parity correction precise. Let \(w^*\) denote vertical
reflection of a word: reverse its order, fix every \(e_i\), and interchange
\(r_i\) with \(l_i\). For \(1\leq a<b\leq n\), put
\[
\operatorname{Cup}(a,b)=
\begin{cases}
(e_{a+1}e_{a+3}\cdots e_{b-2})
(e_ae_{a+2}\cdots e_{b-1}),&b-a\text{ odd},\\[1mm]
r_{b-1}(e_{a+1}e_{a+3}\cdots e_{b-3})
(e_ae_{a+2}\cdots e_{b-2})e_{b-1},&b-a\text{ even},
\end{cases}
\]
where an empty product is omitted, and set
\(\operatorname{Cap}(a,b)=\operatorname{Cup}(a,b)^*\).
The second line is the extra Motzkin case; locally it looks as follows.
\begin{gather*}
\begin{tikzpicture}[anchorbase]
\draw[usual] (0,1) to[out=270,in=180] (1,0.5) to[out=0,in=270] (2,1);
\draw[usual] (0.5,1) to[out=270,in=180] (0.75,0.75) to[out=0,in=270] (1,1);
\draw[usual,dot] (1.5,1) to (1.5,0.8); 
\draw[usual] (1,0) node {$\dots$};
\end{tikzpicture}
\quad =\quad
\begin{tikzpicture}[anchorbase]
\draw[usual] (0,1) to[out=270,in=180] (0.25,0.75) to[out=0,in=270] (0.5,1);
\draw[usual] (0,1) to (0,2);
\draw[usual] (0.5,1) to[out=90,in=180] (0.75,1.25) to[out=0,in=90] (1,1);
\draw[usual] (1,1) to[out=270,in=180] (1.25,0.75) to[out=0,in=270] (1.5,1);
\draw[usual] (1.5,1) to (2,2);
\draw[usual,dot] (1.5,2) to (1.5,1.8);
\draw[usual,dot] (2,1) to (2,1.2);
\draw[usual] (1,0) node {$\dots$};
\end{tikzpicture}
\,.
\end{gather*}
Also put
\[
D(i)=
\begin{cases}
r_i l_i,&i<n,\\
l_{n-1}r_{n-1},&i=n.
\end{cases}
\]
Thus \(D(i)=d_i\), although \(d_i\) is not included in our chosen
generating set.

\begin{algorithm}[H]
\caption{Motzkin factorization algorithm}\label{alg:MoFact}
\begin{algorithmic}
\Require {$X\in Mo_n$ and generators $e_i,r_i,l_i$}
\Function{Factorize $Mo_n$}{$X$}
\State topdot, cup, right, left, cap, botdot \(\gets\) empty lists
\State blocks \(\gets\) blocks of \(X\) in the order fixed above
\For{block in blocks}
\If{block is a singleton \([a]\)}
\If{\(a>0\)}
\State topdot $\gets$ concatenate topdot and \(D(a)\) 
\Else
\State botdot $\gets$ concatenate botdot and \(D(|a|)\)
\EndIf
\ElsIf{block is a cup \([a,b]\)}
\State cup $\gets$ concatenate \(\operatorname{Cup}(a,b)\) and cup
\ElsIf{block is a transversal pair \([a,-b]\)}
\If{\(a<b\)}
\State right $\gets$ \(l_a, l_{a+1},\dots, l_{b-1}\) and right
\ElsIf{\(a>b\)}
\State left $\gets$ concatenate left and \(r_{a-1},r_{a-2},\dots, r_b\)
\EndIf
\Else
\Comment{block is a cap \([-a,-b]\), \(a<b\)}
\State Add \(\operatorname{Cap}(a,b)\) to cap
\EndIf
\EndFor
\State \textbf{return} concatenate topdot, cup, right, left, cap, and botdot
\EndFunction
\end{algorithmic}
\end{algorithm}

\begin{Proposition}
\autoref{alg:MoFact} factors every \(X\in Mo_n\) and runs in
\(O(n+L(X))\).
\end{Proposition}

\begin{proof}
The identities \(D(i)=d_i\) give the singleton blocks. The two cases in
\(\operatorname{Cup}(a,b)\) give a cup with the prescribed endpoints;
the even case uses \(r_{b-1}\) to pass the unmatched middle endpoint.
Vertical reflection proves the cap formula. The transversal words are
the planar rook words from the preceding subsection.

It remains to order the blocks. Nested cups are processed from the
inside out, caps from the outside in, right-moving strings from right to
left, and left-moving strings from left to right. These are exactly the
six lists in the algorithm. Planarity then prevents different local
words from interfering. The input scan is linear and each output
generator is written once.
\end{proof}

A singleton contributes \(2\), a transversal pair \([a,-b]\) contributes
\(a-b\), and a cup or cap with endpoints \(a<b\) contributes \(b-a\).

\begin{Proposition}\label{P:MoMax}
For odd \(n\geq3\), the maximal output length of \autoref{alg:MoFact} is
\[
\frac{n^2+15}{2}.
\]
For even \(n\geq6\), it is
\[
\frac{n^2}{2}+8.
\]
The maxima for \(n=2,4\) are \(8,16\), respectively.
\end{Proposition}

\begin{proof}
Let \(X\in Mo_n\), and let \(k\) be the number of pairs in \(X\).
Then \(X\) has \(2n-2k\) singleton blocks. Since every pair
\([a,b]\) contributes
\(\lvert |a|-|b|\rvert\) generators and every singleton contributes
two, we have
\[
L(X)
=
\sum_{\substack{[a,b]\in X\\ |[a,b]|=2}}
\big\lvert |a|-|b|\big\rvert
+4(n-k).
\]
For \(1\leq r<n\), let \(c_r(X)\) be the number of pairs crossing the
vertical cut between columns \(r\) and \(r+1\). Double counting
pair-cut incidences gives
\[
L(X)=\sum_{r=1}^{n-1}c_r(X)+4(n-k).
\]
There are \(2r\) vertices to the left of the cut and
\(2(n-r)\) to the right, and there are only \(k\) pairs in total.
Hence
\[
c_r(X)\leq \min\{2r,2(n-r),k\}.
\]
Therefore
\[
L(X)
\leq
\sum_{r=1}^{n-1}\min\{2r,2(n-r),k\}+4(n-k).
\]
Splitting the sum at the two points where \(2r\) crosses \(k\) gives
\[
\sum_{r=1}^{n-1}\min\{2r,2(n-r),k\}
=
kn-\left\lceil\frac{k^2}{2}\right\rceil.
\]
Consequently,
\[
\begin{aligned}
L(X)
&\leq
kn-\left\lceil\frac{k^2}{2}\right\rceil+4(n-k)\\
&\leq
kn-\frac{k^2}{2}+4(n-k)\\
&=
\frac{n^2+16}{2}
-\frac{(k-(n-4))^2}{2}.
\end{aligned}
\]
Since \(L(X)\) is an integer,
\[
L(X)\leq
\left\lfloor\frac{n^2+16}{2}\right\rfloor
=
\begin{cases}
\dfrac{n^2+15}{2},&n\text{ odd},\\[2mm]
\dfrac{n^2}{2}+8,&n\text{ even}.
\end{cases}
\]
These bounds are attained by the Motzkin rainbow diagrams in
\autoref{fig:MoRainbow}. 
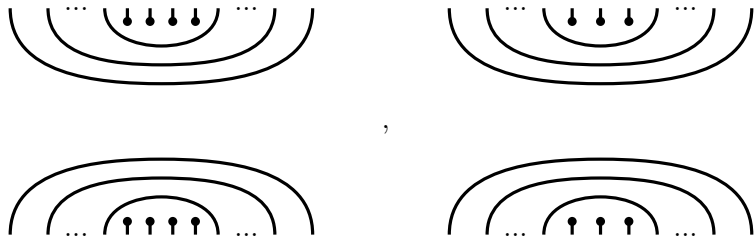
\begin{figure}[H]
\begin{tikzpicture}[anchorbase]
\draw[usual] (0,3) to[out=270,in=180] (2,2) to[out=0,in=270] (4,3);
\draw[usual] (0.5,3) to[out=270,in=180] (2,2.25) to[out=0,in=270] (3.5,3);
\draw[usual] (0.875,3) node {$\dots$};
\draw[usual] (1.25,3) to[out=270,in=180] (2,2.5) to[out=0,in=270] (2.75,3);
\draw[usual,dot] (1.55,3) to (1.55,2.8);
\draw[usual,dot] (1.85,3) to (1.85,2.8);
\draw[usual,dot] (2.15,3) to (2.15,2.8);
\draw[usual,dot] (2.45,3) to (2.45,2.8);
\draw[usual] (3.125,3) node {$\dots$};
\draw[usual] (0,0) to[out=90,in=180] (2,1) to[out=0,in=90] (4,0);
\draw[usual] (0.5,0) to[out=90,in=180] (2,0.75) to[out=0,in=90] (3.5,0);
\draw[usual] (0.875,0) node {$\dots$};
\draw[usual] (1.25,0) to[out=90,in=180] (2,0.5) to[out=0,in=90] (2.75,0);
\draw[usual,dot] (1.55,0) to (1.55,0.2);
\draw[usual,dot] (1.85,0) to (1.85,0.2);
\draw[usual,dot] (2.15,0) to (2.15,0.2);
\draw[usual,dot] (2.45,0) to (2.45,0.2);
\draw[usual] (3.125,0) node {$\dots$};
\end{tikzpicture}
\quad\quad ,\quad\quad
\begin{tikzpicture}[anchorbase]
\draw[usual] (0,3) to[out=270,in=180] (2,2) to[out=0,in=270] (4,3);
\draw[usual] (0.5,3) to[out=270,in=180] (2,2.25) to[out=0,in=270] (3.5,3);
\draw[usual] (0.875,3) node {$\dots$};
\draw[usual] (1.25,3) to[out=270,in=180] (2,2.5) to[out=0,in=270] (2.75,3);
\draw[usual,dot] (1.625,3) to (1.625,2.8);
\draw[usual,dot] (2,3) to (2,2.8);
\draw[usual,dot] (2.375,3) to (2.375,2.8);
\draw[usual] (3.125,3) node {$\dots$};
\draw[usual] (0,0) to[out=90,in=180] (2,1) to[out=0,in=90] (4,0);
\draw[usual] (0.5,0) to[out=90,in=180] (2,0.75) to[out=0,in=90] (3.5,0);
\draw[usual] (0.875,0) node {$\dots$};
\draw[usual] (1.25,0) to[out=90,in=180] (2,0.5) to[out=0,in=90] (2.75,0);
\draw[usual,dot] (1.625,0) to (1.625,0.2);
\draw[usual,dot] (2,0) to (2,0.2);
\draw[usual,dot] (2.375,0) to (2.375,0.2);
\draw[usual] (3.125,0) node {$\dots$};
\end{tikzpicture}
.
\caption{Motzkin Rainbow diagrams. Left: even $n$; right: odd $n$.}
\label{fig:MoRainbow}
\end{figure}
For even \(n\), take four central singletons in each
row and nest the remaining vertices in cups and caps. The pair
lengths in each row are
$n-1,n-3,\ldots,5$,
so the total output length is
$2\cdot\frac{n^2-16}{4}+16
=
\frac{n^2}{2}+8$.
For odd \(n\), take three central singletons in each row. The pair
lengths in each row are
$n-1,n-3,\ldots,4$,
and hence the output length is
$2\cdot\frac{n^2-9}{4}+12
=
\frac{n^2+15}{2}$.
The cases \(n=2,4\) are immediate.
\end{proof}

Thus, we again get the expected worst case complexity \(\Theta(n^2)\).

The average calculation is most transparent if it is kept exact until
the last step. Let \(M_k\) be the \(k\)th Motzkin number (see e.g. \cite{Be-catalan-motzkin}), with
\(M_0=1\), and define
\[
\begin{aligned}
A_n={}&4nM_{2n-1}\\
&+2\sum_{d=1}^{n-1}
d(n-d)
M_{d-1}M_{2n-d-1}\\
&+\sum_{i=1}^n\sum_{j=1}^n
|i-j|M_{i+j-2}M_{2n-i-j}.
\end{aligned}
\]

\begin{Proposition}\label{P:MoAvg}
The average output length of \autoref{alg:MoFact} is
\[
\overline L_n^{Mo}=\frac{A_n}{M_{2n}}.
\]
With the step count \(T(X)=2n+2L(X)\), the average runtime satisfies
\[
\operatorname{Avg}_n
=2n+2\frac{A_n}{M_{2n}}
\sim
\frac{4\sqrt3(9-4\sqrt2)}{9\sqrt\pi}\cdot n^{3/2}.
\]
\end{Proposition}

\begin{proof}
Fixing one singleton leaves an arbitrary Motzkin matching on the other
\(2n-1\) boundary vertices:
\begin{gather*}
\begin{tikzpicture}[anchorbase]
\draw[usual,dot] (0,1) to (0,0.8);
\draw[usual] (1,0) node {$M_{2n-1}$};
\draw[usual] (0,1) node [above] {$i$};
\end{tikzpicture}
.
\end{gather*}
Thus all singleton blocks contribute \(4nM_{2n-1}\).

For a cup of span \(d\), the vertices inside and outside the cup support
independent Motzkin matchings:
\begin{gather*}
\begin{tikzpicture}[anchorbase]
\draw[usual] (1,2) node {$1\leq d \leq n-1$};
\draw[usual] (0,1) to[out=270,in=180] (1,0.25) to[out=0,in=270] (2,1);
\draw[usual] (0,1) node [above] {$i$};
\draw[usual] (2,1) node [above] {$i+d$};
\draw[usual] (1,-1) node {$M_{2n-d-1}$};
\draw[usual] (1,0.85) node {$M_{d-1}$};
\end{tikzpicture}
.
\end{gather*}
There are \(n-d\) possible positions, and caps contribute the same
amount. This gives the second line of \(A_n\), including the parity
correction. Finally, a transversal pair from \(i\) to \(j\) leaves
intervals of sizes \(i+j-2\) and \(2n-i-j\):
\begin{gather*}
\begin{tikzpicture}[anchorbase]
\draw[usual] (0,2) to (1.5,0);
\draw[usual] (0,0.25) node {$M_{2n-i-j}$};
\draw[usual] (0,2) node [above] {$i$};
\draw[usual] (1.5,0) node [below] {$-j$};
\draw[usual] (2,1) node {$M_{i+j-2}$};
\end{tikzpicture}
.
\end{gather*}
This gives the last line of \(A_n\). Divide by \(M_{2n}\).

For the leading term, use \cite{Be-catalan-motzkin}
\[
M_m\sim\frac1{2\sqrt{\pi}}\left(\frac3m\right)^{3/2}3^m.
\]
The singleton and parity terms are \(O(n)\). The distance part of the
cup and cap sum contributes
\[
\frac{2}{\sqrt{3\pi}}\,n^{3/2}+O(n),
\]
while the transversal sum contributes
\[
\frac{1}{2\sqrt{3\pi}}
\left(\int_0^1\int_0^1
\frac{|x-y|}{(x+y)^{3/2}}\,dy\,dx\right)n^{3/2}+O(n).
\]
Splitting the square along \(x=y\) gives
\[
\int_0^1\int_0^1\frac{|x-y|}{(x+y)^{3/2}}\,dy\,dx
=8-\frac{16\sqrt2}{3}.
\]
Doubling the output length contribution gives the stated runtime
constant.
\end{proof}

\subsection{Planar Partition}

We use the standard monoid isomorphism
\[
\Phi:pPa_n\longrightarrow TL_{2n}
\]
from, e.g., \cite[(1.5)]{HaRa-partition-algebras}. It is the usual fattening
of a noncrossing partition. For a signed vertex \(v\), define its first
and second copies in clockwise boundary order by
\[
\lambda(i)=2i-1,\quad \rho(i)=2i,\qquad
\lambda(-i)=-2i,\quad \rho(-i)=-(2i-1).
\]
If \(B=(v_1,\dots,v_m)\) is a block listed clockwise, fattening replaces
it by the pairs
\[
[\rho(v_j),\lambda(v_{j+1})],
\qquad 1\leq j\leq m,
\]
where \(v_{m+1}=v_1\). This single cyclic rule includes top, bottom,
transversal, and singleton blocks.

\begin{algorithm}[H]
\caption{Planar partition fattening}\label{alg:pPatoTL}
\begin{algorithmic}
\Require {$X\in pPa_n$}
\Function{Fatten}{$X$}
\State \(X'\gets\) empty list
\For{block \(B=(v_1,\dots,v_m)\) of \(X\), in clockwise order}
\For{\(j=1,\dots,m\)}
\State Add \([\rho(v_j),\lambda(v_{j+1})]\) to \(X'\),
with \(v_{m+1}=v_1\)
\EndFor
\EndFor
\State \textbf{return} \(X'\)
\EndFunction
\end{algorithmic}
\end{algorithm}

\begin{Proposition}\label{P:pPartComplexity}
\autoref{alg:pPatoTL} computes \(\Phi(X)\) in \(O(n)\) time.
\end{Proposition}

\begin{proof}
The blocks contain \(2n\) vertices in total, and the algorithm performs
constant work at each vertex.
\end{proof}

The inverse conversion on generators is particularly simple.

\begin{Lemma}\label{L:pPaTLgenerators}
For \(1\leq i\leq n\) and \(1\leq i<n\),
\[
\Phi(d_i)=e_{2i-1},
\qquad
\Phi(p_i)=e_{2i}.
\]
Consequently,
\[
\Phi^{-1}(e_j)=
\begin{cases}
d_{(j+1)/2},&j\text{ odd},\\
p_{j/2},&j\text{ even}.
\end{cases}
\]
\end{Lemma}

\begin{proof}
For \(d_i\), fattening joins the two copies of \(i\) on each row and
leaves all other copies vertical. This is \(e_{2i-1}\). For \(p_i\),
the four vertices in its nontrivial block fatten to the top and bottom
pairs in columns \(2i,2i+1\), with the neighboring copies vertical.
This is \(e_{2i}\).
\end{proof}

\begin{Proposition}\label{P:PPalower}
Any algorithm that factors every planar partition diagram $X\in pPa_n$ into
a product of the generators $\{p_i,d_j\}$ has worst case
complexity $\Omega(n^2)$. Thus $\operatorname{Opt}_n\in\Omega(n^2)$.
\end{Proposition}

\begin{proof}
The Temperley--Lieb rainbow lower bound
transfers through \(\Phi^{-1}\).
\end{proof}

Now our algorithm.

\begin{algorithm}[H]
\caption{Planar partition factorization algorithm}\label{alg:pPaFact}
\begin{algorithmic}
\Require {$X\in pPa_n$ and generators $p_1,\dots,p_{n-1},d_1,\dots,d_n$}
\Function{Factorize $pPa_n$}{$X$}
\State \(X'\gets\Phi(X)\) using \autoref{alg:pPatoTL}
\State \(F\gets\) output of \autoref{alg:TLFact} on \(X'\in TL_{2n}\)
\For{\(e_j\) in \(F\)}
\State Replace \(e_j\) by \(d_{(j+1)/2}\) if \(j\) is odd,
and by \(p_{j/2}\) if \(j\) is even
\EndFor
\State \textbf{return} \(F\)
\EndFunction
\end{algorithmic}
\end{algorithm}

\begin{Proposition}
\autoref{alg:pPaFact} factors every \(X\in pPa_n\) and has worst case
complexity \(O(n^2)\).
\end{Proposition}

\begin{proof}
By \autoref{L:pPaTLgenerators}, \(\Phi\) and its inverse preserve both
products and word length in the chosen generators. Correctness follows
from \autoref{alg:TLFact}. Fattening is linear, while factorization in
\(TL_{2n}\) is \(O(n^2)\).
\end{proof}

Thus, we again get \(\Theta(n^2)\), and move to the average.

\begin{Proposition}
The average output length of \autoref{alg:pPaFact} is
\[
\overline L_n^{pPa}
=\frac{(2n+1)\left(2^{4n-1}-\binom{4n}{2n}\right)}
{\binom{4n}{2n}}.
\]
With the preceding step convention, including \(4n\) steps for
fattening,
\[
\operatorname{Avg}_n^{pPa}
=4n+\frac{2^{4n}(2n+1)}{\binom{4n}{2n}}-2
\sim2\sqrt{2\pi}\cdot n^{3/2}.
\]
\end{Proposition}

\begin{proof}
The map \(\Phi\) is a bijection from the uniform distribution on
\(pPa_n\) to the uniform distribution on \(TL_{2n}\), and the generator
replacement in \autoref{L:pPaTLgenerators} preserves output length.
Apply \autoref{P:TLAvg} with \(2n\) in place of \(n\), and add the linear
fattening cost.
\end{proof}


\section{Factorization algorithms: symmetric case}\label{S:Symmetric}

We now include crossings.

\subsection{Symmetric Group}\label{S:SymFact}

We start with the symmetric group, which is the model for all the nonplanar cases below. Here the relevant statistic is the inversion number: every adjacent transposition removes at most one inversion, while the reverse permutation has
$\binom{n}{2}=\frac{n(n-1)}{2}$
inversions. Thus the quadratic lower bound is already visible before we choose an algorithm.

\begin{Proposition}
Any algorithm that factors every permutation in $S_n$ into the adjacent transpositions $\{s_i\}$ has worst case complexity $\Omega(n^2)$. Thus $\operatorname{Opt}_n\in\Omega(n^2)$.
\end{Proposition}

\begin{proof}
Note that the reverse permutation $[2,\dots,n,1]$ has $\binom{n}{2}$ inversions, and every adjacent transposition changes the inversion number by at most one. Any explicit factorization must therefore contain at least $\binom{n}{2}$ factors in the worst case.
\end{proof}

We now use the standard sorting construction. An optimal solution is known, see \cite[Section 5]{KT-AlgDesign}; \autoref{alg:SymFact} implements these well-understood ideas with a Fenwick tree \cite{Fenwick}. The tree keeps track of the rank of the next entry to be inserted, while the inner loop writes the corresponding adjacent transpositions.

\begin{algorithm}[H]
\caption{Symmetric group factorization algorithm}\label{alg:SymFact}
\begin{algorithmic}
\Require {$X \in S_n$, generators $s_1,s_2,\dots,s_{n-1}$ of $S_n$}
\Function{Factorize $S_n$}{$X$}
\State $p\gets$ one-line notation of $X$
\State $F$ $\gets$ empty list

\State Construct the array pos such that pos$(x)$ is the position of $x$ in $p$

\State Initialize a Fenwick tree containing $1$ at every position

\For{$x$ in $[n,n-1,\dots,1]$}

\State rank $\gets$ prefix sum of the Fenwick tree at position pos$(x)$

\For{$j$ in $[\text{rank},\dots,x-1]$}
\State Add $s_j$ to $F$
\EndFor

\State Remove $x$ from the Fenwick tree

\EndFor

\State Reverse $F$ \Comment{Required for diagram monoid multiplication convention.}

\State \Return $F$
\EndFunction
\end{algorithmic}
\end{algorithm}

\begin{Proposition}\label{P:SymComplexity}
The complexity of \autoref{alg:SymFact} is $n\log(n)+O(\text{\#factors})$, where $\#$factors equals the number of adjacent transpositions, which is the inversion number.
\end{Proposition}

\begin{proof}
See \cite[Section 5]{KT-AlgDesign}. The Fenwick tree operations contribute $n\log(n)$, and the inner loops write one generator for each inversion.
\end{proof}

\begin{Proposition}
The worst-case complexity of \autoref{alg:SymFact} is $\Theta(n^2)$, and its maximal output length is $\binom{n}{2}$.
\end{Proposition}

\begin{proof}
By \autoref{P:SymComplexity}, the output length is the inversion number. Its maximum is $\binom{n}{2}$, attained by the reverse permutation, and the preceding lower bound shows that this is optimal.
\end{proof}

\begin{Remark}
For the generating set consisting of $s=s_1$ and the long cycle $c$,
any algorithm factoring elements of $S_n$ into these generators has
worst-case complexity $\Theta(n^2)$; see \cite{BKL89}.
\end{Remark}

The average is equally classical.

\begin{Proposition}\label{P:SymAvg}
The average complexity of \autoref{alg:SymFact} is $\operatorname{Avg}_n=n\log(n) + \frac{n(n-1)}{4}$. 
\end{Proposition}

\begin{proof}
For any pair $1\leq i<j \leq n$ in a permutation $x$, it contributes $1$ to the inversion number if $x(i)>x(j)$ and $0$ otherwise. Under the uniform distribution the two possibilities are equally likely, so each pair contributes $\frac{1}{2}$ on average. There are $\binom{n}{2}$ pairs, hence the average inversion number is
\[
\frac{1}{2}\binom{n}{2}=\frac{n(n-1)}{4}.
\]
Combining this with \autoref{P:SymComplexity} gives the result.
\end{proof}

We will use \autoref{alg:SymFact} as a subroutine for all the remaining symmetric diagram monoids. The same reverse permutation also gives their common quadratic lower bound.

\begin{Proposition}\label{P:SymDiagLower}
Any algorithm that factors any of the symmetric diagram monoids, Brauer, Rook, Rook--Brauer, and partition, with generating set $\{s_i\}$ for the symmetric group will have worst case complexity $\Omega(n^2)$. Thus $\operatorname{Opt}_n\in\Omega(n^2)$ for all of them.
\end{Proposition}

\begin{proof}
Each of these monoids contains the symmetric group, and hence the reverse permutation $[2,\dots,n,1]$. It remains only to check that the additional generators cannot shorten a factorization of this permutation. The reverse permutation has the maximal number of through strands $n$, whereas each $p_i,e_i,d_i$ has $<n$ such strands, and multiplication of diagrams cannot increase the number of through strands. Thus no word for the reverse permutation can use one of these generators. Its factorization length is therefore the same as in the symmetric group, and is $\Omega(n^2)$.
\end{proof}

For Brauer, Rook, and Rook--Brauer the matching upper bounds all use the same picture. First place cups, caps, and/or dots in standard positions using the nonpermutation generators. Then use one permutation above and one below to move these pieces, together with the surviving transversal strings, to their required positions; see \autoref{S:GenMethod} for a picture. The partition monoid needs an additional middle step, but follows the same principle.

\subsection{Brauer}

First up is the Brauer monoid. The lower bound is already given by \autoref{P:SymDiagLower}, so it remains to construct a quadratic-time factorization.

Fix the number $k$ of transversal strings. Noting that $n$ and $k$ always share the same parity, we first use $e_1,e_3,\dots,e_{n-k-1}$ to create a standard row of cups and caps, leaving the last $k$ positions for the transversal strings. The two permutations TopSym and BotSym then move the cup and cap endpoints, and the surviving strings, to their prescribed positions. Thus all nontrivial displacement is delegated to the symmetric group routine.

\begin{algorithm}
\caption{Brauer factorization algorithm}\label{alg:BrFact}
\begin{algorithmic}
\Require{$X\in Br_n$, generators $e_1,\dots,e_{n-1},s_1,\dots,s_{n-1}$}
\Function{Factorize $\mathcal{B}_n$}{$X$} 
\State blocks $\gets$ list of blocks of $X$ in lexicographic order
\State $k\gets0$
\State topStrand, botStrand, topCup, botCap, top, $F$, bot $\gets$ empty lists

\For{block $[a,b]\in$ blocks}
\If{$ab<0$}
\State $k\gets k+1$
\State Add $[a]$ to topStrand
\State Add $[b]$ to botStrand
\ElsIf{$a>0$}
\State Add $[a,b]$ to topCup
\Else
\State Add $[a,b]$ to botCap
\EndIf
\EndFor

\State obtainedTopCup, obtainedBotCap $\gets$ empty lists

\For{$i$ in $[1,3,\dots,n-k-1]$}
\State Add $e_i$ to $F$
\State Add $[i,i+1]$ to obtainedTopCup
\State Add $[-i,-i-1]$ to obtainedBotCap
\EndFor

\State obtainedTop $\gets [n-k+1,\dots,n]$
\State obtainedBot $\gets [-(n-k+1),\dots,-n]$

\State TopSym, BotSym $\gets$ empty lists

\For{$i$ in $[1,\dots,\#$topCup]}
\State Add $[\text{topCup}[i][1],-\text{obtainedTopCup}[i][1]]$ to TopSym
\State Add $[\text{topCup}[i][2],-\text{obtainedTopCup}[i][2]]$ to TopSym
\EndFor

\For{$i$ in $[1,\dots,\#$topStrand]}
\State Add $[\text{topStrand}[i][1],-\text{obtainedTop}[i][1]]$ to TopSym
\EndFor

\For{$i$ in $[1,\dots,\#$botCap]}
\State Add $[\text{botCap}[i][1],-\text{obtainedBotCap}[i][1]]$ to BotSym
\State Add $[\text{botCap}[i][2],-\text{obtainedBotCap}[i][2]]$ to BotSym
\EndFor

\For{$i$ in $[1,\dots,\#$botStrand]}
\State Add $[\text{botStrand}[i][1],-\text{obtainedBot}[i][1]]$ to BotSym
\EndFor

\State $\text{FactTopSym}, \text{FactBotSym}$ $\gets$ factorization of $\text{TopSym}, \text{BotSym}$ into adjacent transpositions using \autoref{alg:SymFact}

\State \Return concatenate FactTopSym, $F$, and FactBotSym

\EndFunction
\end{algorithmic}
\end{algorithm}

\begin{Proposition}
\autoref{alg:BrFact} factors every $X\in Br_n$ and has computational complexity $O(n)+O(\text{sym})$, where $O(\text{sym})$ is the cost of factorizing the two symmetric group elements.
\end{Proposition}

\begin{proof}
The factors $e_1,e_3,\dots,e_{n-k-1}$ create the required number of cups and caps in standard positions, with the remaining $k$ vertices joined by transversal strings. By construction, TopSym sends the standard top endpoints to the required top endpoints, while BotSym does the same at the bottom. The resulting product is therefore $X$.

The initial scan, the construction of the standard cups and caps, and the two permutation diagrams all take $O(n)$ time. The remaining work is exactly the factorization of TopSym and BotSym by \autoref{alg:SymFact}, giving $O(n)+O(\text{sym})$.
\end{proof}

Together with \autoref{P:SymDiagLower}, this gives worst case complexity $\Theta(n^2)$.

We now turn to the average. There are two nontrivial contributions: the number of cup/cap generators written before the permutations, and the inversion numbers of the two permutations. All remaining setup is linear. We calculate these two contributions separately.

\begin{Proposition}
The asymptotic average for the cup/cap portion of \autoref{alg:BrFact} is $S^{Br}_n \sim \frac{n}{4}$.
\end{Proposition}

\begin{proof}
We need to find the average number of cups/caps over all diagrams in $Br_n$. We start with a weighted sum, found via a simple counting argument and \cite[Section 5F]{khovanov-monoidal-2024}, and then divide by the total number of diagrams for the monoid. Let
\[
S^{Br}_n = \frac{1}{(2n-1)!!}\sum_{\substack{0\leq k\leq n \\ n-k\text{ even}}} \frac{1}{2}(n-k)k!\left(\binom{n}{k}(n-k-1)!!\right)^2.
\]
Let $n-k=2r \implies k = n-2r$. If $n=2m$ is even:
\[
S^{Br}_n = \frac{2^{2m}(2m)!^3}{(4m)!}\sum_{r=0}^m\frac{r}{4^{r}r!^2(2m-2r)!}
\]
and symbolic simplification gives
\[
S^{Br}_n = \frac{m(2m-1)}{4m-1} = \frac{1}{2}\frac{n(n-1)}{2n-1} \sim \frac{n}{4} \text{ as }n\rightarrow \infty.
\]
If $n=2m+1$ is odd:
\[
S^{Br}_n = \frac{\sqrt{\pi}(2m+1)!^2}{2^{2m-1}\Gamma(2m+1/2)}\sum_{r=0}^{m-1/2}\frac{r}{r!^2(2m-2r+1)!}
\]
and symbolic simplification gives
\[
S^{Br}_n = \frac{1}{2}(n+1)\left(\frac{n}{2n+1}-\frac{\Gamma(1+n/2)^2}{\sqrt{\pi \Gamma(n+3/2)}}\right) \sim \frac{n}{4} \text{ as }n\rightarrow \infty.
\]
This can be verified using, for example, Mathematica. (All the other asymptotics below can be verified in the same way, and we stop making this comment.)
\end{proof}

\begin{Proposition}
The asymptotic average for the permutation portion of \autoref{alg:BrFact} is $P_n^{Br} \sim \frac{3n^2}{8}-\frac{n}{2}$.
\end{Proposition}

\begin{proof}
We proceed similarly to the previous proposition. Fix a number of through strands $k$. After the step of placing all cups, the top half of the product of the factors will look like:
\begin{gather*}
\begin{tikzpicture}[anchorbase]
\draw[usual] (0,1) to[out=270,in=180] (0.25,0.75) to[out=0,in=270] (0.5,1);
\draw[usual] (1,1) to[out=270,in=180] (1.25,0.75) to[out=0,in=270] (1.5,1);
\draw[usual] (2,1) node {$\dots$};
\draw[usual] (2.5,1) to[out=270,in=180] (2.75,0.75) to[out=0,in=270] (3,1);
\draw[usual] (3.5,1) to (3.5,0.5);
\draw[usual] (4,1) node {$\dots$};
\draw[usual] (4.5,1) to (4.5,0.5);
\end{tikzpicture}\quad ,
\end{gather*}
and the next step in the algorithm is to use a permutation in $S_n$ to transform this top half into the desired top half. We want to find the average factorization length of all required permutations, which is equal to the sum of their inversion numbers. 

We break this into several cases. The first type of inversion is when a cup is swapped with another cup:
\begin{gather*}
\begin{tikzpicture}[anchorbase]
\draw[usual] (0,1) to[out=270,in=180] (0.5,0.5) to[out=0,in=270] (1,1);
\draw[usual,dashed] (1.5,1) to[out=270,in=180] (2,0.5) to[out=0,in=270] (2.5,1);
\end{tikzpicture}
\quad \longrightarrow \quad 
\begin{tikzpicture}[anchorbase]
\draw[usual] (0,1) to[out=270,in=180] (0.5,0.5) to[out=0,in=270] (1,1);
\draw[usual,dashed] (1.5,1) to[out=270,in=180] (2,0.5) to[out=0,in=270] (2.5,1);
\end{tikzpicture}
\quad 0 \text{ inversions,}
\end{gather*}
\begin{gather*}
\begin{tikzpicture}[anchorbase]
\draw[usual] (0,1) to[out=270,in=180] (0.5,0.5) to[out=0,in=270] (1,1);
\draw[usual,dashed] (1.5,1) to[out=270,in=180] (2,0.5) to[out=0,in=270] (2.5,1);
\end{tikzpicture}
\quad \longrightarrow \quad 
\begin{tikzpicture}[anchorbase]
\draw[usual] (0,1) to[out=270,in=180] (0.875,0.5) to[out=0,in=270] (1.75,1);
\draw[usual,dashed] (0.875,1) to[out=270,in=180] (1.5,0.5) to[out=0,in=270] (2.5,1);
\end{tikzpicture}
\quad 1 \text{ inversion, }
\end{gather*}
\begin{gather*}
\begin{tikzpicture}[anchorbase]
\draw[usual] (0,1) to[out=270,in=180] (0.5,0.5) to[out=0,in=270] (1,1);
\draw[usual,dashed] (1.5,1) to[out=270,in=180] (2,0.5) to[out=0,in=270] (2.5,1);
\end{tikzpicture}
\quad \longrightarrow \quad 
\begin{tikzpicture}[anchorbase]
\draw[usual] (0,1) to[out=270,in=180] (1.25,0.25) to[out=0,in=270] (2.5,1);
\draw[usual,dashed] (0.75,1) to[out=270,in=180] (1.25,0.5) to[out=0,in=270] (1.75,1);
\end{tikzpicture}
\quad 2 \text{ inversions.}
\end{gather*}
Each scenario of this type is equally likely and on average contributes $1$ per occurrence. There are $\binom{\frac{n-k}{2}}{2}$ pairs of cups, meaning the average contribution from all such pairs is $\binom{\frac{n-k}{2}}{2}$.

The remaining type is when an endpoint of a cup is swapped with a line:
\begin{gather*}
\begin{tikzpicture}[anchorbase]
\draw[usual] (0,1) to[out=270,in=180] (0.5,0.5) to[out=0,in=270] (1,1);
\draw[usual] (1.5,1) to (1.5,0);
\end{tikzpicture}
\quad \longrightarrow \quad 
\begin{tikzpicture}[anchorbase]
\draw[usual] (0,1) to[out=270,in=180] (0.5,0.5) to[out=0,in=270] (1,1);
\draw[usual] (1.5,1) to (1.5,0);
\end{tikzpicture}
\quad 0 \text{ inversions,}
\end{gather*}
\begin{gather*}
\begin{tikzpicture}[anchorbase]
\draw[usual] (0,1) to[out=270,in=180] (0.5,0.5) to[out=0,in=270] (1,1);
\draw[usual] (1.5,1) to (1.5,0);
\end{tikzpicture}
\quad \longrightarrow \quad 
\begin{tikzpicture}[anchorbase]
\draw[usual] (0,1) to[out=270,in=180] (0.75,0.5) to[out=0,in=270] (1.5,1);
\draw[usual] (0.75,1) to (0.75,0);
\end{tikzpicture}
\quad,\quad
\begin{tikzpicture}[anchorbase]
\draw[usual] (0.5,1) to[out=270,in=180] (1,0.5) to[out=0,in=270] (1.5,1);
\draw[usual] (0,1) to (0,0);
\end{tikzpicture}
\quad 1 \text{ inversion.}
\end{gather*}
Each scenario of this type is equally likely and on average contributes $\frac{1}{2}$ per occurrence. There are $\binom{k}{1}\binom{n-k}{1}$ cup/line pairs, meaning the average contribution from all such pairs is $\frac{k(n-k)}{2}$. Counting the number of diagrams with a fixed top half, we therefore have an average contribution
\[
^TP_n^{Br} = \frac{1}{(2n-1)!!}\sum_{\substack{0\leq k\leq n \\ n-k\text{ even}}} \left(\binom{\frac{n-k}{2}}{2} + \frac{k(n-k)}{2} \right)k!\left(\binom{n}{k}(n-k-1)!!\right)^2.
\]
For the bottom half of the diagram we also have another type of inversion, when a line is swapped with another line:
\begin{gather*}
\begin{tikzpicture}[anchorbase]
\draw[usual] (0,0) to (0,0.5);
\draw[usual] (0.5,0) to (0.5,0.5);
\end{tikzpicture}
\quad \longrightarrow \quad 
\begin{tikzpicture}[anchorbase]
\draw[usual] (0,0) to (0,0.5);
\draw[usual] (0.5,0) to (0.5,0.5);
\end{tikzpicture}
\quad 0 \text{ inversions},
\end{gather*}
\begin{gather*}
\begin{tikzpicture}[anchorbase]
\draw[usual] (0,0) to (0,0.5);
\draw[usual] (0.5,0) to (0.5,0.5);
\end{tikzpicture}
\quad \longrightarrow \quad 
\begin{tikzpicture}[anchorbase]
\draw[usual] (0,0) to (0.5,0.5);
\draw[usual] (0.5,0) to (0,0.5);
\end{tikzpicture}
\quad 1 \text{ inversion}.
\end{gather*}
Both scenarios of this type are equally likely and contribute on average $\frac{1}{2}$ per occurrence, and there are $\binom{k}{2}$ of them, so this type contributes $\frac{1}{2}\binom{k}{2}$ to the total inversion number. The average contribution from the bottom half is
\[
^BP_n^{Br} = \frac{1}{(2n-1)!!}\sum_{\substack{0\leq k\leq n \\ n-k\text{ even}}} \left(\frac{1}{2}\binom{k}{2}+\binom{\frac{n-k}{2}}{2} + \frac{k(n-k)}{2} \right)k!\left(\binom{n}{k}(n-k-1)!!\right)^2.
\]
Thus, including a count of the number of diagrams for a given $k$,
\begin{flalign*}
\quad\quad\quad\quad P_n^{Br} &= \prescript{T}{}{P}_n^{Br} + \prescript{B}{}{P}_n^{Br}&&\\\nonumber
&= \frac{1}{(2n-1)!!}\sum_{\substack{0\leq k\leq n \\ n-k\text{ even}}} \left(\frac{1}{2}\binom{k}{2}+2\binom{\frac{n-k}{2}}{2} + 2\frac{k(n-k)}{2} \right)k!\left(\binom{n}{k}(n-k-1)!!\right)^2.&&
\end{flalign*}
Looking at $^BP_n^{Br}$, working similarly to the previous proposition, set $n-k=2r$ and assuming $n=2m$ is even: 
\[
^BP_n^{Br} = \frac{m(2m-1)(14m^2-15m+3)}{32m^2-32m+6} \sim \frac{7n^2}{32}-\frac{n}{4},
\]
and assuming $n=2m+1$ is odd:
\[
^BP_n^{Br} = 2\frac{1}{8}\left(\frac{4m(2m+1)(14m^2-m-1)}{16m^2-1}-\frac{(1-4m^2)\Gamma(m+1)^2}{\sqrt{\pi}\Gamma(2m+3/2)}\right) \sim \frac{7n^2}{32}-\frac{n}{4}.
\]
The same calculation for the top half gives $^TP_n^{Br} \sim \frac{5n^2}{32}-\frac{n}{4}$, and adding these together completes the proof.
\end{proof}

Hence, the total average complexity is asymptotically
\[
\operatorname{Avg}_n\sim\frac{3}{8}n^2-\frac{1}{4}n.
\]

\subsection{Rook}

Next, the rook monoid. Again \autoref{P:SymDiagLower} gives the quadratic lower bound.

The construction is the Brauer construction with dots in place of cups and caps. We first put all top and bottom dots in standard initial positions using the $d_i$. The remaining vertices are vertical strings. The permutations TopSym and BotSym then move the dots and surviving strings to their required positions.

\begin{algorithm}
\caption{Rook factorization algorithm}\label{alg:RoFact}
\begin{algorithmic}
\Require{$X\in Ro_n$, generators $d_1,\dots,d_{n},s_1,\dots,s_{n-1}$}
\Function{Factorize $Ro_n$}{$X$}
\State blocks $\gets$ list of blocks of $X$ in lexicographic order
\State TopExpect, BotExpect, TopDotExpect, BotDotExpect, TopGet, BotGet, TopDotGet, BotDotGet, Top, Bot $\gets$ empty lists
\For{block in blocks}
\If{block is a singleton $[a]$}
\If{$a>0$}
\State Add $[a]$ to TopDotExpect
\Else
\State Add $[a]$ to BotDotExpect
\EndIf
\ElsIf{block is a pair $[a,b]$} 
\State Add $[a]$ to TopExpect
\State Add $[b]$ to BotExpect
\EndIf 
\EndFor 
\For{$i$ in $[1,\dots,\#\text{TopDotExpect}]$}
\State Add $d_i$ to Top
\State Add $[i]$ to TopDotGet
\EndFor
\For{$i$ in $[\#\text{TopDotExpect}+1,\dots,n]$}
\State Add $[i]$ to TopGet
\EndFor
\For{$i$ in $[1,\dots,\#\text{BotDotExpect}]$}
\State Add $d_i$ to Bot
\State Add $[-i]$ to BotDotGet
\EndFor
\For{$i$ in $[\#\text{BotDotExpect}+1,\dots,n]$}
\State Add $[-i]$ to BotGet
\EndFor
\State TopSym, BotSym $\gets$ empty lists
\For{$i$ in $[1,\dots,\#\text{TopDotExpect}]$}
\State Add $[\text{TopDotExpect}[i][1],-\text{TopDotGet}[i][1]]$ to TopSym
\EndFor 
\For{$i$ in $[1,\dots,\#\text{TopExpect}]$}
\State Add $[\text{TopExpect}[i][1],-\text{TopGet}[i][1]]$ to TopSym
\EndFor 
\For{$i$ in $[1,\dots,\#\text{BotDotExpect}]$}
\State Add $[\text{BotDotExpect}[i][1],-\text{BotDotGet}[i][1]]$ to BotSym
\EndFor 
\For{$i$ in $[1,\dots,\#\text{BotExpect}]$}
\State Add $[\text{BotExpect}[i][1],-\text{BotGet}[i][1]]$ to BotSym
\EndFor 
\State $\text{FactTopSym}, \text{FactBotSym}$ $\gets$ factorization of $\text{TopSym}, \text{BotSym}$ into adjacent transpositions using \autoref{alg:SymFact}
\State \textbf{return} concatenate FactTopSym, Top, Bot, and FactBotSym
\EndFunction
\end{algorithmic}
\end{algorithm}

\begin{Proposition}
\autoref{alg:RoFact} factors every $X\in Ro_n$ and has computational complexity $O(n)+O(\text{sym})$.
\end{Proposition}

\begin{proof}
The $d_i$ create the required number of top and bottom dots in standard positions, leaving vertical strings everywhere else. TopSym and BotSym are defined precisely so that their endpoints move these standard dots and strings to those of $X$. Hence the returned product is $X$.

Every list other than the two symmetric group factorizations is constructed by a linear scan of the vertices. Thus the nonpermutation work is $O(n)$, and the remaining cost is the two calls to \autoref{alg:SymFact}.
\end{proof}

Together with \autoref{P:SymDiagLower}, \autoref{alg:RoFact} has worst case complexity $\Theta(n^2)$.

For the average, we again separate the local generators from the permutations. The local part is now the number of dots.

\begin{Proposition}
The asymptotic average for the dot portion of \autoref{alg:RoFact} is $S^{Ro}_n \sim 2\sqrt{n}$.
\end{Proposition}

\begin{proof}
A rank $k$ rook diagram has $n-k$ top dots and $n-k$ bottom dots. Weighting by the number $k!\binom{n}{k}^2$ of rank $k$ diagrams gives
\[
S^{Ro}_n = \dfrac{\sum_{k=0}^n2(n-k)k!\binom{n}{k}^2}{\sum_{k=0}^nk!\binom{n}{k}^2} = 2\dfrac{\sum_{j=0}^nj\binom{n}{j}\frac{1}{j!}}{\sum_{j=0}^n\binom{n}{j}\frac{1}{j!}}
\]
letting $j=n-k$. The Laguerre polynomial $L_n(x)$ is defined by \cite[Section 5.1]{Sz-OrthogonalPolynomials} as
\[
L_n(x)=\sum_{j=0}^n\binom{n}{n-j}\frac{(-x)^j}{j!}.
\]
Differentiating gives
\[
L'_n(x) = \sum_{j=0}^n\binom{n}{n-j}(-j)\frac{(-x)^{j-1}}{j!} = -\frac{1}{-x}\sum_{j=0}^nj\binom{n}{n-j}\frac{(-x)^{j}}{j!},
\]
so
\[
S^{Ro}_n = -2\frac{L'_n(-1)}{L_n(-1)}.
\]
Let $F_n(x) := L_n(-x)$, so $S^{Ro}_n =4\frac{F'_n(x)}{F_n(x)}$, then using \cite[(5.1.2)]{Sz-OrthogonalPolynomials} and the chain rule
\[
xL''_n(x)+(1-x)L'_n(x)+nL_n(x)=0 \implies xF''_n(x)+(1+x)F'_n(x)-nF_n(x)=0.
\]
Let $Q_n(x)=x\frac{F'_n(x)}{F_n(x)}$, then $Q'_n(x)=\frac{F'_n(x)}{F_n(x)}-x\frac{F'_n(x)^2}{F_n(x)^2}+x\frac{F''_n(x)}{F_n(x)}$. Since $x\frac{F''_n(x)}{F_n(x)}=n-(1+x)\frac{F'_n(x)}{F_n(x)}$,
\[
Q'_n(x)=\frac{1}{x}\left(Q_n(x)-Q_n(x)^2 - (1+x)Q_n(x) \right) + n
\]
\begin{equation}\label{eq:pRoLaguerre}
\implies Q'_n(x)+\frac{1}{x}Q_n(x)^2+Q_n(x)-n=0
\end{equation}
We want to find the asymptotic of $2Q_n(1)=S_n^{Ro}$ as $n\rightarrow \infty$. We know from the worst-case complexity that $S_n^{Ro} \leq O(n)$. We also know from \autoref{eq:pRoLaguerre} that $Q_n(1) \leq O(\sqrt{n})$ and $Q'_n(1) \leq O(n)$. We want to show that $Q'_n(1) \leq Q_n(1)$ since this implies $Q_n(1) \geq \Omega(\sqrt{n})$ in order to balance the growth against $n$ in the equation. This reduces to showing that $F''_n(1)F_n(1)-F'_n(1)^2 = L''_n(-1)L_n(-1)-L'_n(-1)^2 \leq 0$. 

By \cite[Theorem 3.31 and Section 5.1]{Sz-OrthogonalPolynomials}, $L_n(x)$ has $n$ positive zeroes so we can write $L_n(x) = A \prod_{r=1}^n (x-\lambda_r)$ for some constant $A$ and roots $\lambda_1,\dots,\lambda_n$. Then $L'_n(x)=A\sum_{r=1}^n\prod_{j\neq r}(x-\lambda_r)$ and if we assume $x\notin \{\lambda_1,\dots,\lambda_n\}$ (which is true for $x=-1$ since the zeroes are positive) we get 
\[
\frac{L'_n(x)}{L_n(x)} = \sum_{r=1}^n\frac{1}{x-\lambda_r}
\]
\[
\implies \frac{L''_n(x)L_n(x)-L'_n(x)^2}{L_n(x)^2} =  \left(\frac{L'_n(x)}{L_n(x)}\right)' = -\sum_{r=1}^n\frac{1}{(x-\lambda_r)^2}
\]
\[
\implies L''_n(-1)L_n(-1)-L'_n(-1)^2 = -L_n(x)^2\sum_{r=1}^n\frac{1}{(x-\lambda_r)^2} \leq 0
\]
as required. Therefore, $Q_n(1)^2$ dominates \autoref{eq:pRoLaguerre} so $Q_n(1) \sim \sqrt{n}$ and thus $S^{Ro}_n \sim 2\sqrt{n}$.
\end{proof}

\begin{Proposition}
The asymptotic average for the permutation portion of \autoref{alg:RoFact} is $P_n^{Ro} \sim \frac{n^2}{4} +\frac{n^{3/2}}{2}-\frac{11n}{8}$.
\end{Proposition}

\begin{proof}
We now sum the inversion numbers of the two permutations required after the dots have been placed. In this case the top half consists of only dots and lines, which can be considered as $k$ distinguishable singletons and $n-k$ indistinguishable singletons of a partition of $n$. 

In the top, only swapping dots and lines contributes to the inversion number, while in the bottom swapping lines with each other also contribute. Each dot--line pair is inverted with probability $\frac12$, and on the bottom each pair of surviving lines is also inverted with probability $\frac12$. Hence
\[
^TP_n^{Ro}=\dfrac{\sum_{k=0}^n\left(\frac{k(n-k)}{2}\right) k!\binom{n}{k}^2}{\sum_{k=0}^nk!\binom{n}{k}^2}
\]
\[
^BP_n^{Ro}=\dfrac{\sum_{k=0}^n\left(\frac{k(n-k)}{2}+\frac{1}{2}\binom{k}{2}\right) k!\binom{n}{k}^2}{\sum_{k=0}^nk!\binom{n}{k}^2}
\]
and
\[
P_n^{Ro}=\dfrac{\sum_{k=0}^n\left(k(n-k)+\frac{1}{2}\binom{k}{2}\right) k!\binom{n}{k}^2}{\sum_{k=0}^nk!\binom{n}{k}^2}=\dfrac{\sum_{j=0}^n\left(j(n-j)+\frac{1}{2}\binom{n-j}{2}\right)\binom{n}{j}\frac{1}{j!}}{\sum_{j=0}^n\binom{n}{j}\frac{1}{j!}}
\]
\[
= \frac{n(n-1)}{4}-\frac{n-1}{2}\frac{L'_n(-1)}{L_n(-1)}-\frac{3}{4}\frac{L''_n(-1)}{L_n(-1)}.
\]
Expressing the resulting sums through the same Laguerre ratios as in the preceding proposition gives the stated asymptotic. 
\end{proof}

Hence, the total average complexity is asymptotically
\begin{gather*}
\operatorname{Avg}_n\sim 4n+4\sqrt{n}+\frac{n^2}{4} +\frac{n^{3/2}}{2}-\frac{11n}{8}+2n\log(n)\cdot \frac{n!}{\sum_{k=0}^n k!\binom{n}{k}^2}
\\
\sim \frac{n^2}{4}+\frac{n^{3/2}}{2}+\frac{21n}{8}+2\sqrt{n}.
\end{gather*}

\subsection{Rook--Brauer}

The rook--Brauer case combines the preceding two constructions. The lower bound is again \autoref{P:SymDiagLower}.

We first place all cups and caps in standard adjacent positions and then place the top and bottom dots immediately after them. The remaining positions carry transversal strings. As before, TopSym and BotSym rearrange this standard configuration into the prescribed diagram. Thus the algorithm is literally the Brauer and rook procedures run in parallel before the two symmetric group factorizations.

\begin{algorithm}
\caption{Rook--Brauer factorization algorithm}\label{alg:RoBrFact}
\begin{algorithmic}
\Require{$X\in RoBr_n$, generators $e_1,\dots,e_{n-1},d_1,\dots,d_n,s_1,\dots,s_{n-1}$}
\Function{Factorize $RoBr_n$}{$X$}
\State blocks $\gets$ list of blocks of $X$ in lexicographic order
\State TopExpect, BotExpect, CupExpect, CapExpect, TopDotExpect, BotDotExpect, TopGet, BotGet, CupGet, CapGet, TopDotGet, BotDotGet, Top, Bot $\gets$ empty lists
\For{block in blocks}
\If{block is a singleton $[a]$}
\If{$a>0$}
\State Add $[a]$ to TopDotExpect
\Else
\State Add $[a]$ to BotDotExpect
\EndIf
\ElsIf{block is a pair $[a,b]$} 
\If{$b>0$}
\State Add $[a,b]$ to CupExpect
\ElsIf{$a<0$}
\State Add $[a,b]$ to CapExpect
\Else
\State Add $[a]$ to TopExpect
\State Add $[b]$ to BotExpect
\EndIf
\EndIf 
\EndFor 
\For{$i$ in $[1,3,\dots,2\#\text{CupExpect}-1]$}
\State Add $e_i$ to Top
\State Add $[i,i+1]$ to CupGet
\EndFor
\For{$i$ in $[2\#\text{CupExpect}+1,\dots,2\#\text{CupExpect}+\#\text{TopDotExpect}]$}
\State Add $d_i$ to Top
\State Add $[i]$ to TopDotGet
\EndFor
\For{$i$ in $[2\#\text{CupExpect}+\#\text{TopDotExpect}+1,\dots,n]$}
\State Add $[i]$ to TopGet
\EndFor
\For{$i$ in $[1,3,\dots,2\#\text{CapExpect}-1]$}
\State Add $e_i$ to Bot
\State Add $[-i,-i-1]$ to CapGet
\EndFor
\For{$i$ in $[2\#\text{CapExpect}+1,\dots,2\#\text{CapExpect}+\#\text{BotDotExpect}]$}
\State Add $d_i$ to Bot
\State Add $[-i]$ to BotDotGet
\EndFor
\For{$i$ in $[2\#\text{CapExpect}+\#\text{BotDotExpect}+1,\dots,n]$}
\State Add $[-i]$ to BotGet
\EndFor
\State TopSym, BotSym $\gets$ empty lists
\For{$i$ in $[1,\dots,\#\text{CupExpect}]$}
\State Add $[\text{CupExpect}[i][1],-\text{CupGet}[i][1]]$ to TopSym
\State Add $[\text{CupExpect}[i][2],-\text{CupGet}[i][2]]$ to TopSym
\EndFor
\For{$i$ in $[1,\dots,\#\text{TopDotExpect}]$}
\State Add $[\text{TopDotExpect}[i][1],-\text{TopDotGet}[i][1]]$ to TopSym
\EndFor 
\For{$i$ in $[1,\dots,\#\text{TopExpect}]$}
\State Add $[\text{TopExpect}[i][1],-\text{TopGet}[i][1]]$ to TopSym
\EndFor 
\For{$i$ in $[1,\dots,\#\text{CapExpect}]$}
\State Add $[\text{CapExpect}[i][1],-\text{CapGet}[i][1]]$ to BotSym
\State Add $[\text{CapExpect}[i][2],-\text{CapGet}[i][2]]$ to BotSym
\EndFor
\For{$i$ in $[1,\dots,\#\text{BotDotExpect}]$}
\State Add $[\text{BotDotExpect}[i][1],-\text{BotDotGet}[i][1]]$ to BotSym
\EndFor 
\For{$i$ in $[1,\dots,\#\text{BotExpect}]$}
\State Add $[\text{BotExpect}[i][1],-\text{BotGet}[i][1]]$ to BotSym
\EndFor 
\State $\text{FactTopSym}, \text{FactBotSym}$ $\gets$ factorization of $\text{TopSym}, \text{BotSym}$ into adjacent transpositions using \autoref{alg:SymFact}

\State \textbf{return} concatenate FactTopSym, Top, Bot, and FactBotSym
\EndFunction
\end{algorithmic}
\end{algorithm}

\begin{Proposition}
\autoref{alg:RoBrFact} factors every $X\in RoBr_n$ and has computational complexity $O(n)+O(\text{sym})$.
\end{Proposition}

\begin{proof}
The first part of the algorithm creates the correct numbers of cups, caps, dots, and transversal strings in standard positions. TopSym and BotSym are then constructed endpoint by endpoint so that the two permutations move this standard diagram to $X$. This proves correctness.

The input scan and all standard-position lists have total size $O(n)$. The only remaining cost is factorizing TopSym and BotSym with \autoref{alg:SymFact}, so the complexity is $O(n)+O(\text{sym})$.
\end{proof}

Together with \autoref{P:SymDiagLower}, the worst case is $\Theta(n^2)$.

For the average, the local contribution now counts cups, caps, and dots, while the permutation contribution combines the inversion types from the Brauer and rook cases.

\begin{Proposition}
The asymptotic average for the cup/cap and dot portion of \autoref{alg:RoBrFact} is $S^{RoBr}_n \sim \frac{n}{2}$.
\end{Proposition}

\begin{proof}
Let $T_m$ be the $m$th telephone number
\[
T_m = \sum_{t=0}^{\lfloor\frac{m}{2}\rfloor}\frac{m!}{(m-2t)!2^t t!}
\]
which has asymptotic (\cite[5.1.4(53)]{art-computer-programming}:
\[
T_m \sim \frac{1}{\sqrt{2}}m^{m/2}e^{-m/2+\sqrt{m}-1/4}(1+\frac{7}{24}m^{-1/2}).
\]
$T_m$ is precisely the number of partial matchings on $m$ elements, so $|RoBr_n|=T_{2n}$, cf. \cite{DEG-MotzPartialBrauer}.
 
Any cup contributes one factor, there are $\binom{n}{2}$ ways to choose a cup, and there are $T_{2n-2}$ remaining ways to match the remaining vertices. Thus, the average contribution from cups is 
\[
\binom{n}{2}\frac{T_{2n-2}}{T_{2n}}
\]
and a similar argument for the dots applies, meaning 
\[
S^{RoBr}_n = 2\left(\binom{n}{2}\frac{T_{2n-2}}{T_{2n}} + n\frac{T_{2n-1}}{T_{2n}}\right).
\]
Using the displayed asymptotic for the telephone numbers, the ratio simplifies to $S_n^{RoBr}\sim\frac{n}{2}$. 
\end{proof}

\begin{Proposition}
The asymptotic average for the permutation portion of \autoref{alg:RoBrFact} is $P_n^{RoBr} \sim \frac{3n^2}{8}+\frac{\sqrt{2}}{8}n^{3/2}-\frac{7n}{8}$.
\end{Proposition}

\begin{proof}
The permutation part combines exactly the inversion types seen in the Brauer and rook cases. Looking first at the bottom permutation, there are five types; after averaging over their relative positions, each contributes $1$ per occurrence:
\begin{enumerate}
\item swapping a cap with another cap $\rightarrow$ there are $3\binom{n}{4}$ ways to place the two caps and $T_{2n-4}$ partial matchings of the rest of the vertices,
\item swapping a cap with a line $\rightarrow$ there are $\binom{n}{2}\binom{n-2}{1}\binom{n}{1}$ ways to place the cap and line and $T_{2n-4}$ partial matchings of the rest of the vertices,
\item swapping a dot with a line $\rightarrow$ there are $\binom{n}{2}\binom{n}{1}$ ways to place the dot and line and $T_{2n-3}$ partial matchings of the rest of the vertices,
\item swapping a dot with a cap $\rightarrow$ there are $\binom{n}{2}\binom{n-2}{1}$ ways to place the dot and cap and $T_{2n-3}$ partial matchings of the rest of the vertices,
\item and swapping a line with another line $\rightarrow$ there are $\binom{n}{2}\binom{n}{2}$ ways to place the two lines and $T_{2n-4}$ partial matchings of the rest of the vertices.
\end{enumerate}
Notice that when counting these, we involve the whole diagram instead of splitting neatly into top and bottom like in the Brauer case. From this,
\[
^BP^{RoBr}_n = 3\binom{n}{4}\frac{T_{2n-4}}{T_{2n}}+\binom{n}{2}n(n-2)\frac{T_{2n-4}}{T_{2n}}+\binom{n}{2}n\frac{T_{2n-3}}{T_{2n}}+\binom{n}{2}(n-2)\frac{T_{2n-3}}{T_{2n}}+\binom{n}{2}^2\frac{T_{2n-4}}{T_{2n}}.
\]
The line-line type of inversion is only counted for the bottom, because this type only requires one permutation, which \cref{alg:RoBrFact} has occurring in the bottom. So,
\[
^TP^{RoBr}_n = 3\binom{n}{4}\frac{T_{2n-4}}{T_{2n}}+\binom{n}{2}n(n-2)\frac{T_{2n-4}}{T_{2n}}+\binom{n}{2}n\frac{T_{2n-3}}{T_{2n}}+\binom{n}{2}(n-2)\frac{T_{2n-3}}{T_{2n}}
\]
Adding the top and bottom contributions and applying the same telephone-number asymptotics gives the stated formula for $P_n^{RoBr}$.
\end{proof}

Hence, the total average complexity is asymptotically 
\[
\operatorname{Avg}_n\sim 
\frac{3}{8}n^2+\frac{\sqrt{2}}{8}n^{3/2}+\frac{45}{8}n.
\]

\subsection{Partition}

Finally, the partition monoid. The quadratic lower bound still comes from \autoref{P:SymDiagLower}, but the construction has one extra layer because blocks may have arbitrary size.

We build the diagram in three pieces. First, the $p_i$ create standard consecutive top and bottom blocks; permutations then move these blocks to the required vertex sets. Second, the middle records which top and bottom pieces belong to the same transversal block: one surviving string is kept for each transversal block and the unused positions become dots. This middle diagram is planar rook and is factorized by \autoref{alg:pRoFact}. Finally, the top, middle, and bottom words are concatenated in the order prescribed by \autoref{S:DiagMon}.

\begin{algorithm}
\caption{Partition factorization algorithm}\label{alg:PaFact}
\begin{algorithmic}
\Require {$X \in Pa_n$, generators $\mathds{1},p_1,\dots,p_{n-1},d_1,\dots,d_n,s_1,\dots,s_{n-1}$}
\Function{Factorize $Pa_n$}{$X$}
\State blocks $\gets$ list of blocks of $X$ in lexicographic order

\State Initialize $\text{TopReq}\gets$ empty list and $\text{BotReq}\gets$ empty list.

\Statex
\Comment{Record what the top and bottom should look like in TopReq/BotReq}

\For{$b\in$ blocks}
\State $\text{TopBlock}\gets$ empty list
\State $\text{BotBlock}\gets$ empty list
\ForAll{$v\in b$}
\If{$v>0$}
\State Add $v$ to TopBlock
\Else
\State Add $v$ to BotBlock
\EndIf
\EndFor
\State TopReq $\gets$ concatenate TopReq and TopBlock
\State BotReq $\gets$ concatenate BotReq and BotBlock 
\EndFor

\Statex
\Comment{Construct the top partition}

\State $\text{counter}\gets1$
\State TopFactors, TopProd, BottomTopProd $\gets$ empty lists

\For{$b\in\text{TopReq}$}
\State Add $[\text{counter},\dots,\text{counter}+\#b-1]$ to BottomTopProd
\If{$b\neq\emptyset$}
\If{$\#b=1$}
\State Add $[\text{counter}]$ to $\text{TopProd}$
\Else
\For{$i$ in $[\text{counter}\dots\text{counter}+\#b-2]$}
\State Add $p_i$ to $\text{TopFactors}$
\EndFor
\State Add $[\text{counter},\dots,\text{counter}+\#b-1]$ to $\text{TopProd}$
\EndIf
\State $\text{counter}\gets\text{counter}+\#b$
\EndIf
\EndFor

\Statex
\Comment{Construct the bottom partition}

\State $\text{counter}\gets1$
\State BotFactors, BotProd, TopBotProd $\gets$ empty list

\For{$b\in\text{BotReq}$}
\State Add $[\text{counter},\dots,\text{counter}+\#b-1]$ to $\text{TopBotProd}$
\If{$b\neq\emptyset$}
\If{$\#b=1$}
\State Add $[-\text{counter}]$ to $\text{BotProd}$
\Else
\For{$i=\text{counter}$ to $\text{counter}+\#b-2$}
\State Add $p_i$ to $\text{BotFactors}$
\EndFor
\State Add $[-\text{counter},\ldots,-(\text{counter}+\#b-1)]$ in $\text{BotProd}$
\EndIf
\State $\text{counter}\gets\text{counter}+\#b$
\EndIf
\EndFor

\Statex
\Comment{Construct permutations to rearrange what we obtained (TopProd/BotProd) to match what we require (TopReq/BotReq)}

\State Flatten $\text{TopProd}$ and $\text{TopReq}$
\State Flatten $\text{BotProd}$ and $\text{BotReq}$

\For{$i$ in $[1,\dots,n]$}
\State Add $[-\text{TopProd}[i],\text{TopReq}[i]]$ to $\text{TopSym}$
\State Add $[-\text{BotProd}[i],\text{BotReq}[i]]$ to $\text{BotSym}$
\EndFor

\State $\text{FactTopSym},\text{FactBotSym}$ $\gets$ factorization of $\text{TopSym},\text{BotSym}$ into adjacent transpositions using \autoref{alg:SymFact}

\algstore{partition}
\end{algorithmic}
\end{algorithm}
\begin{algorithm}
\begin{algorithmic}
\algrestore{partition}

\Statex
\Comment{Construct the middle partition}

\State $\text{MidFactors}\gets$ empty list

\For{each block in blocks}
\If{both corresponding top and bottom parts are empty}
\State \textbf{continue}
\ElsIf{only the top part is empty}
\State Create singletons for every bottom vertex
\ElsIf{only the bottom part is empty}
\State Create singleton blocks for every top vertex
\ElsIf{block is transversal}
\State Create a line pairing the leftmost vertices of the top and bottom, placing singleton blocks everywhere else in block
\EndIf
\EndFor

\State factorize MidFactors using \autoref{alg:pRoFact}, replacing instances of $r_il_i$ or $l_ir_i$ with $d_i$

\Statex
\Comment{Combine the factors}

\State \Return concatenate FactTopSym, TopFactors, MidFactors, BotFactors, and FactBotSym

\EndFunction
\end{algorithmic}
\end{algorithm}

\begin{Example}
The following
\begin{gather*}
\begin{tikzpicture}[anchorbase]
\draw[usual] (0,1) to[out=270,in=180] (0.5,0.5) to[out=0,in=270] (1,1);
\draw[usual] (0,1) to (0,0);
\draw[usual] (0.5,1) to (0.5,0);
\draw[usual,dot] (1.5,1) to (1.5,0.8);
\draw[usual] (0.5,0) to[out=90,in=180] (0.75,0.25) to[out=0,in=90] (1,0);
\draw[usual] (1,0) to[out=90,in=180] (1.25,0.25) to[out=0,in=90] (1.5,0);
\end{tikzpicture}
\quad=\quad
\begin{tikzpicture}[anchorbase]
\draw[usual] (0,0) to (0,1);
\draw[usual] (0.5,0) to (0.5,1);
\draw[usual] (0.5,0) to[out=90,in=180] (0.75,0.25) to[out=0,in=90] (1,0);
\draw[usual] (1,0) to[out=90,in=180] (1.25,0.25) to[out=0,in=90] (1.5,0);
\draw[usual] (0.5,1) to[out=270,in=180] (0.75,0.75) to[out=0,in=270] (1,1);
\draw[usual] (1,1) to[out=270,in=180] (1.25,0.75) to[out=0,in=270] (1.5,1);
\draw[usual] (2.5,0.5) node {$=p_2p_3$};
\draw[usual] (0,1) to (0,2);
\draw[usual] (0.5,1) to (1,2);
\draw[usual,dot] (1,1) to (1,1.2);
\draw[usual,dot] (1.5,1) to (1.5,1.2);
\draw[usual,dot] (0.5,2) to (0.5,1.8);
\draw[usual,dot] (1.5,2) to (1.5,1.8);
\draw[usual] (2.5,1.5) node {$\in pRo_4$};
\draw[usual] (0,2) to (0,3);
\draw[usual] (0,2) to[out=90,in=180] (0.25,2.25) to[out=0,in=90] (0.5,2);
\draw[usual] (0,3) to[out=270,in=180] (0.25,2.75) to[out=0,in=270] (0.5,3);
\draw[usual] (1,2) to (1,3);
\draw[usual] (1.5,2) to (1.5,3);
\draw[usual] (2.5,2.5) node {$=p_1$};
\draw[usual] (0,3) to (0,4);
\draw[usual] (1,3) to (0.5,4);
\draw[usual] (0.5,3) to (1,4);
\draw[usual] (1.5,3) to (1.5,4);
\draw[usual] (2.5,3.5) node {$\in S_4$};
\end{tikzpicture}
\end{gather*}
illustrates the factorization.
\end{Example}

\begin{Proposition}
The worst-case complexity of \autoref{alg:PaFact} is $\Theta(n^2)$.
\end{Proposition}

\begin{proof}
The lower bound is \autoref{P:SymDiagLower}. For the upper bound, the top and bottom constructions use at most $2n$ generators $p_i$ in total, and the bookkeeping needed to form the required blocks is linear. The middle is a planar rook factorization, which is $O(n^2)$ by \autoref{P:pRolower}, while the top and bottom permutations are factorized by \autoref{alg:SymFact} in $O(n^2)$. Hence the total complexity is $O(n^2)$, as claimed.
\end{proof}

Thus the worst-case analysis is complete. For the average, there are three substantive contributions: the $p_i$ used to build the top and bottom blocks, the two permutation words, and the planar-rook word in the middle. We treat them in that order. Let $W$ denote the Lambert $W$ function.

\begin{Remark}
That the function $W$ appears is expected, see e.g.
\cite{GT-GrowthDiagCat}.
\end{Remark}

\begin{Proposition}
The average number of the generators $p_1,\dots,p_{n-1}$ in the factorization of partition diagrams by \autoref{alg:PaFact} is $S_n^{Pa} \sim 2n-\frac{4n}{W(2n)}$.
\end{Proposition}

\begin{proof}
There are $\scalebox{0.5}{\ensuremath{\bstirling{n}{k}}}$ (Stirling number of the second kind) ways to partition the top $n$ vertices into $k$ parts. Fix $k$, and choose $j$ bottom vertices to belong to parts that are not transversal, i.e. remain in the bottom. There are $\binom{n}{j}B_j$ ways to choose and partition these $j$ bottom-only vertices. The remaining $n-j$ vertices can be assigned to transversal blocks in $k^{n-j}$ ways. Thus, there are 
\[
\sum_{j=0}^n \binom{n}{j}k^{n-j}B_j
\]
partition diagrams with a fixed top partitioned into $k$ parts. Each part contributes $|part|-1$ to $S_n^{Pa}$, so summing the contribution over $k$ parts gives $n-k$. Therefore, combining these and noting the bottom and top vertices contribute the same amount, 
\[
S_n^{Pa} = \frac{2}{B_{2n}}\sum_{k=1}^n(n-k)\bstirling{n}{k}\sum_{j=0}^n\binom{n}{j}k^{n-j}B_j = 2n - \frac{2}{B_{2n}}\sum_{k=1}^nk\bstirling{n}{k}\sum_{j=0}^n\binom{n}{j}k^{n-j}B_j
\]
using \cite{Sp-rec-bell}, where $B_m$ denotes the $m$th Bell number. The second term is twice the average number of blocks meeting the top row, equivalently the bottom row. Counting such blocks gives
\[
\frac{B_{2n+1}}{B_{2n}}-1-\frac{1}{B_{2n}}\sum_{j=1}^n\binom{n}{j}B_{2n-j}
\]
which has leading asymptotic $\frac{2n}{W(2n)}$ by \cite[(1.2)]{St-partition-probability}, where $W(m)$ is the solution to $x\log(x)=m$, called the Lambert W function, and has an approximate asymptotic growth of $\log(m)$. The asymptotic for $S_n^{Pa}$ then follows.
\end{proof}

\begin{Proposition}
The asymptotic average complexity for the permutation portion of \autoref{alg:PaFact} is $P_n^{Pa} \sim \frac{n^2}{2}-\frac{n^2}{2W(2n)}-\frac{nW(2n)}{4}-\frac{n}{2} + O(W(n)).$
\end{Proposition}

\begin{proof}
By \cite{St-partition-probability}, a random partition diagram can be modeled by 
\[
\text{Pr}(U=u)=\frac{u^{2n}}{e u!B_{2n}}
\]
where the $2n$ vertices $1,\dots,n,-1,\dots,-n$ are independently placed in one of $u$ blocks. We want to count the expected inversion number, conditional on the number of parts $U=u$, contributed by the permutations in the top and bottom.

Any pair of vertices $(i,j)$, $1\leq i<j\leq n$, in the top contribute an inversion to $P_n^{Pa}$ when they are in different blocks $A_i$ and $A_j$ and $A_i$ appears before $A_j$. For $A_j$ to appear before $A_i$, it needs to appear first in one of the first $i-1$ vertices. The probability that at least one of $A_i$ and $A_j$ appears in the first $i-1$ vertices is $1-\left(1-\frac{2}{u}\right)^{i-1}$, then by symmetry the probability that $A_j$ appears before $A_i$ is $\frac{1}{2}\left(1-\left(1-\frac{2}{u}\right)^{i-1}\right)$. Finally the probability of $A_i$ being distinct from $A_j$ is $1-\frac{1}{u}$ so
\[
\text{Pr}((i,j) \text{ inversion}|U=u) = \frac{1}{2}\left(1-\frac{1}{u}\right)\left(1-\left(1-\frac{2}{u}\right)^{i-1}\right).
\]
There are $\binom{n-i}{1}=n-i$ ways to place $j$ after fixing $i$ in one of $n-1$ positions, so the inversion number contributed by a random partition is 
\[
\sum_{i=1}^{n-1}\frac{1}{2}(n-i)\left(1-\frac{1}{u}\right)\left(1-\left(1-\frac{2}{u}\right)^{i-1}\right)
\]
and summing over all conditional probabilities $U=u$ gives the average inversion number contributed by permutations in the top:
\begin{gather*}
\mathbb{E}[\text{inv}_{\text{top}}] = \sum_{u=1}^{\infty}\frac{u^{2n}}{e u!B_{2n}}\sum_{i=1}^{n-1}\frac{1}{2}(n-i)\left(1-\frac{1}{u}\right)\left(1-\left(1-\frac{2}{u}\right)^{i-1}\right)
\\
=\sum_{u=1}^{\infty}\frac{u^{2n}}{e u!B_{2n}}\frac{1}{2}\left(1-\frac{1}{u}\right)\left(\frac{n(n-1)}{2}-\frac{nu}{2}+\frac{u^2}{4}-\frac{u^2}{4}\left(1-\frac{2}{u}\right)^n\right).
\end{gather*}
Then, by the probability distribution we are using from \cite{St-partition-probability}, we can write this in terms of expectations
\begin{gather*}
\mathbb{E}[\text{inv}_{\text{top}}] = \frac{n^2}{4}-\frac{2n+1}{8}\mathbb{E}[U]+\frac{1}{8}\mathbb{E}[U^2]-\frac{n(n-1)}{4}\mathbb{E}[U^{-1}]-\frac{1}{8}\mathbb{E}\left[U^2\left(1-\frac{1}{U}\right)\left(1-\frac{2}{U}\right)^n\right]
\\
=\frac{n^2}{4}-\frac{2n+1}{8}\frac{B_{2n+1}}{B_{2n}}+\frac{1}{8}\frac{B_{2n+2}}{B_{2n}}-\frac{n(n-1)}{4}\frac{B_{2n-1}}{B_{2n}}-\frac{1}{8}\mathbb{E}\left[U^2\left(1-\frac{1}{U}\right)\left(1-\frac{2}{U}\right)^n\right]
\end{gather*}
using \cite[(2.2)]{St-partition-probability}. \cite[(2.3)]{St-partition-probability} gives that $\mathbb{E}[U] \sim \frac{2n}{W(2n)}$ and \cite[Section 2]{CDKR-partition-limit} gives more general Bell number ratio asymptotics, so the asymptotic of the above is
\begin{gather*}
\sim \frac{n^2}{4}-\frac{2n(2n+1)}{8W(2n)}+\frac{(2n+2)(2n+1)}{8W(2n)^2}-\frac{n(n-1)W(2n)}{8n}-\frac{1}{8}\frac{4n^2}{W(2n)^2}\left(1-\frac{W(2n)}{2n}\right)e^{-W(2n)}
\\
=\frac{n^2}{4}-\frac{nW(2n)}{8}+\frac{W(2n)}{8}-\frac{n^2}{2W(2n)}+\frac{n^2}{4W(2n)^2}-\frac{n}{2W(2n)}+\frac{3n}{4W(2n)^2}+\frac{1}{8}+\frac{1}{4W(2n)^2}
\end{gather*}
Using a very similar, albeit not completely symmetric argument,
\begin{gather*}
\mathbb{E}[\text{inv}_{\text{bot}}] = \sum_{u=1}^{\infty}\frac{u^{2n}}{e u!B_{2n}}\sum_{i=1}^{n-1}\frac{1}{2}(n-i)\left(1-\frac{1}{u}\right)\left(1-\left(1-\frac{2}{u}\right)^{n+i-1}\right)
\\
\sim \frac{n^2}{4}-\frac{nW(2n)}{8}-\frac{n}{2}+\frac{n}{4W(2n)}+\frac{W(2n)}{4}-\frac{1}{4}+\frac{W(2n)}{16n}.
\end{gather*}
Combining the two thus gives
\begin{gather*}
P_n^{Pa} \sim \frac{n^2}{2}-\frac{n^2}{2W(2n)}-\frac{nW(2n)}{4}-\frac{n}{2}+\frac{3W(2n)}{8}+\frac{n^2}{2W(2n)^2}
\\
-\frac{n}{4W(2n)}+\frac{3n}{4W(2n)^2}-\frac{1}{8}+\frac{W(2n)}{16n}+\frac{1}{W(2n)^2}
\\
\sim \frac{n^2}{2}-\frac{n^2}{2W(2n)}-\frac{nW(2n)}{4}-\frac{n}{2} + O(W(n)),
\end{gather*}
and we are done.
\end{proof}

\begin{Proposition}
The asymptotic average for the planar rook portion of \autoref{alg:PaFact}, where we factor the planar rook element, is
\[
M^{Pa}_n \sim 2n+\frac{\sqrt{\pi}}{2}\frac{n^{3/2}}{W(2n)}.
\]
\end{Proposition}

\begin{proof}
Firstly, we know already that \cref{alg:pRoFact} takes $2n+\#$factors steps to run. For a partition $X$, fix the blocks $C_1,\dots,C_k$, and split them by sign, defining $A_i=\#\{\text{top vertices of }C_i\}$ and
$B_i=\#\{\text{bottom vertices of }C_i\}$. For the planar rook diagram we construct in the middle, which we call $X_M$, we need a dot whenever $A_i\neq B_i$. Since each dot contributes one $d_j$ generator, the total contribution from dots must be $\sum_{i=1}^k|A_i-B_i|$. If $C_i$ is transversal, the line connects vertex $\sum_{j<i}A_j$ on the top with vertex $\sum_{j<i}B_j$ on the bottom, making room for all the previous blocks $C_{j<i}$. Thus, the total factor contribution for each block $C_i$ in the factorization of $X_M$ is 
\[
|A_i-B_i|+\mathbf{1}_{A_i,B_i>0}|D_i|
\]
where $D_i := \sum_{j=1}^{i-1}(A_j-B_j)$ so $|D_i|$ is the contribution for transversal blocks. So
\[
M_n^{Pa} = 2n+\mathbb{E}\left[\sum_{i=1}^k|A_i-B_i|+\sum_{i=1}^k\mathbf{1}_{A_i,B_i>0}|D_i|\right].
\]
For the asymptotic, return to the random partition model from \cite{St-partition-probability}:
\[
\text{Pr}(U=u)=\frac{u^{2n}}{e u!B_{2n}}
\]
with average number of blocks $\mathbb{E}[U] \sim \frac{2n}{W(2n)}$. From this, for the average box, the expected number of top and bottom vertices is $\mathbb{E}[A_i] = \mathbb{E}[B_i]\sim \frac{n}{\mathbb{E}[U]} = \frac{W(2n)}{2}$.

$A_i$ and $B_i$ are binomially distributed random variables, which we can approximate as Poisson random variables with mean $\frac{W(2n)}{2}$. Then using the Skellam distribution with mean $0$ and variance $W(2n)$, cf. \cite{Ir-Skellam}:
\[
\mathbb{E}[|A_i-B_i|] = \sum_{u=-\infty}^{\infty} |u|e^{-W(2n)} I_{|u|}(W(2n)) \sim \sqrt{\frac{2W(2n)}{\pi}}.
\]
where $I_{\alpha}(x)$ is a modified Bessel function. Then 
\[
\mathbb{E}[\sum_{i=1}^k|A_i-B_i|] \sim \frac{2n}{W(2n)} \sqrt{\frac{2W(2n)}{\pi}} = \frac{2\sqrt{2}}{\sqrt{\pi}}\frac{n}{\sqrt{W(2n)}}.
\]

For the transversal contribution, condition on $U=u$ and expose the vertices in our lexicographic order. At time $t=ux$, the proportion of boxes already seen is asymptotically $1-e^{-x}$, while new boxes appear with density $ue^{-x}\,dx$. Hence $D_i$ is asymptotically normal with mean $0$ and variance
\[
2ne^{-x}(1-e^{-x}).
\]
Thus
\[
\begin{aligned}
\mathbb{E}\left[\sum_{i=1}^k\mathbf{1}_{A_i,B_i>0}|D_i|\ \middle|\ U=u\right]
&\sim \frac{2u\sqrt n}{\sqrt\pi}
\int_0^\infty e^{-3x/2}\sqrt{1-e^{-x}}\,dx \\
&=\frac{\sqrt\pi}{4}u\sqrt n.
\end{aligned}
\]
The restriction to transversal blocks does not affect the leading term since $n/u\sim W(2n)/2\to\infty$. Averaging over $U$ therefore gives
\[
\mathbb{E}\left[\sum_{i=1}^k\mathbf{1}_{A_i,B_i>0}|D_i|\right]
\sim \frac{\sqrt\pi}{2}\frac{n^{3/2}}{W(2n)}.
\]
This dominates the dot contribution, proving the claim.
\end{proof}

Hence, the average complexity of \autoref{alg:PaFact} is
\begin{gather*}
\operatorname{Avg}_n\sim 8n+2n-\frac{4n}{W(2n)}+2n+\frac{\sqrt{\pi}}{2}\frac{n^{3/2}}{W(2n)}
+2n\log(n)\cdot \frac{n!}{B_{2n}}
+\frac{n^2}{2}-\frac{n^2}{2W(2n)}-\frac{nW(2n)}{4}-\frac{n}{2}
\\
\sim \frac{1}{2}n^2-\frac{n^2}{2W(2n)}
+\frac{\sqrt{\pi}}{2}\frac{n^{3/2}}{W(2n)}
-\frac{1}{4}nW(2n)+\frac{23}{2}n.
\end{gather*}

\section{Experimental Average Complexity}\label{S:Experiment}

We use Monte Carlo methods to estimate the average complexity of each algorithm and compare with our theoretical averages in \autoref{tab:factorization-complexities}. Since the factorization length dominates the average, we compute the average length for each sample to avoid issues arising from the eccentricities of, and discrepancies between, different software and hardware. To implement Monte Carlo sampling for the diagram monoids in GAP, we use uniform random sampling methods produced by OpenAI's GPT~5.6 through Microsoft Copilot based on the algorithms in \cite[Chapters 10,13]{NW-comb-alg} and the Fisher--Yates shuffle \cite[Algorithm 3.4.2P]{art-computer-programming2}. We see in \autoref{fig:exp-average-plot} that the estimated results match our theoretical results in \autoref{tab:factorization-complexities}.

\begin{figure}[H]
\caption{These plots are the average factorization length (AFL) ratios for $n=100,200,\dots,1000,2000,\dots,5000$ (except for $Pa_n$ which stops at $n=1000$ due to excessive computation time) estimated by Monte Carlo sampling with $1000$ samples. The code and numerical results used to produce these are in \cite{St-github-FastFact}.}
\label{fig:exp-average-plot}
    \centering
    \begin{subfigure}{0.48\textwidth}
        \includegraphics[width=\textwidth]{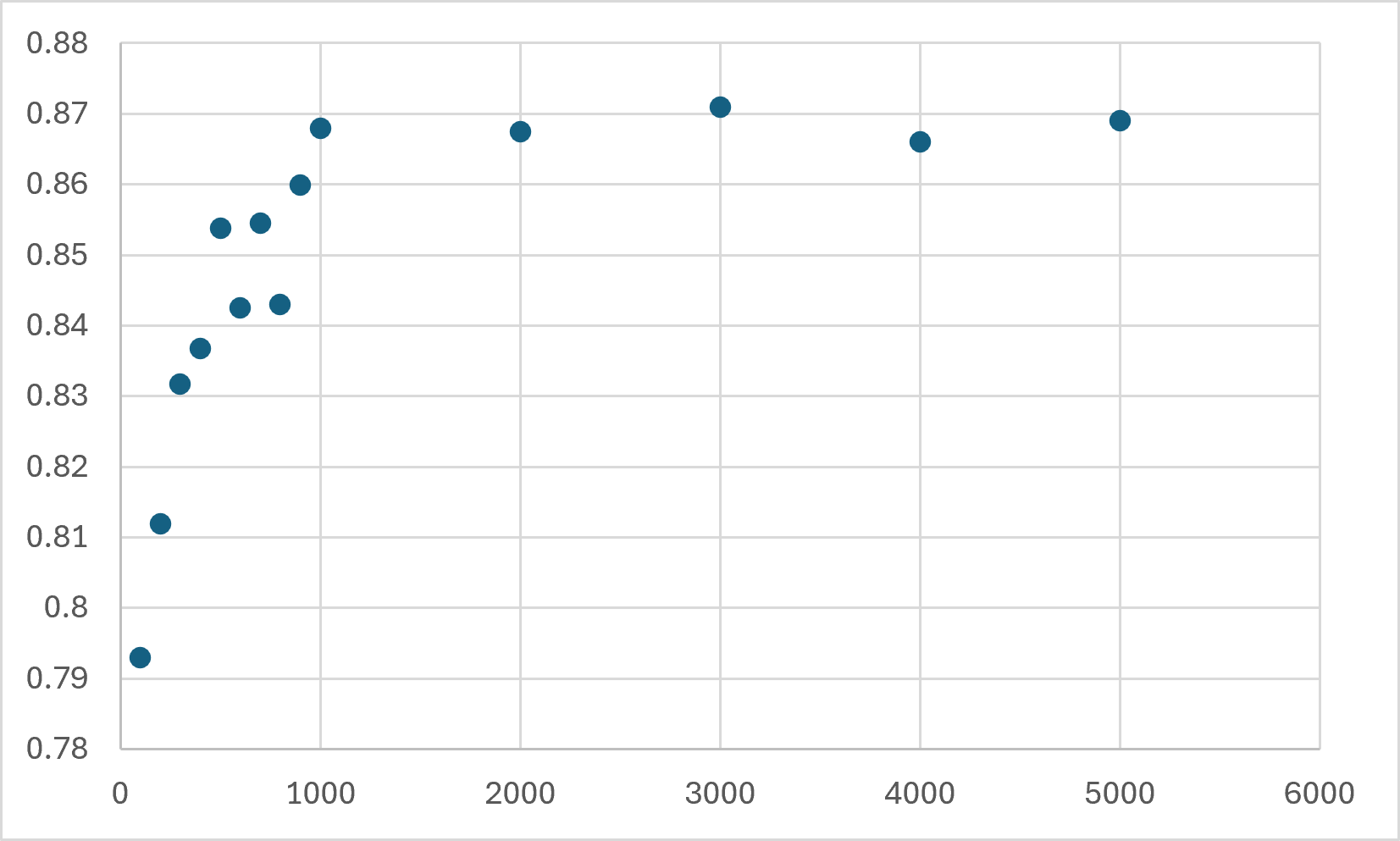} 
        \caption{AFL for $TL_n$ divided by $n^{3/2}$.}
    \end{subfigure}
    \hfill
    \begin{subfigure}{0.48\textwidth}
        \includegraphics[width=\textwidth]{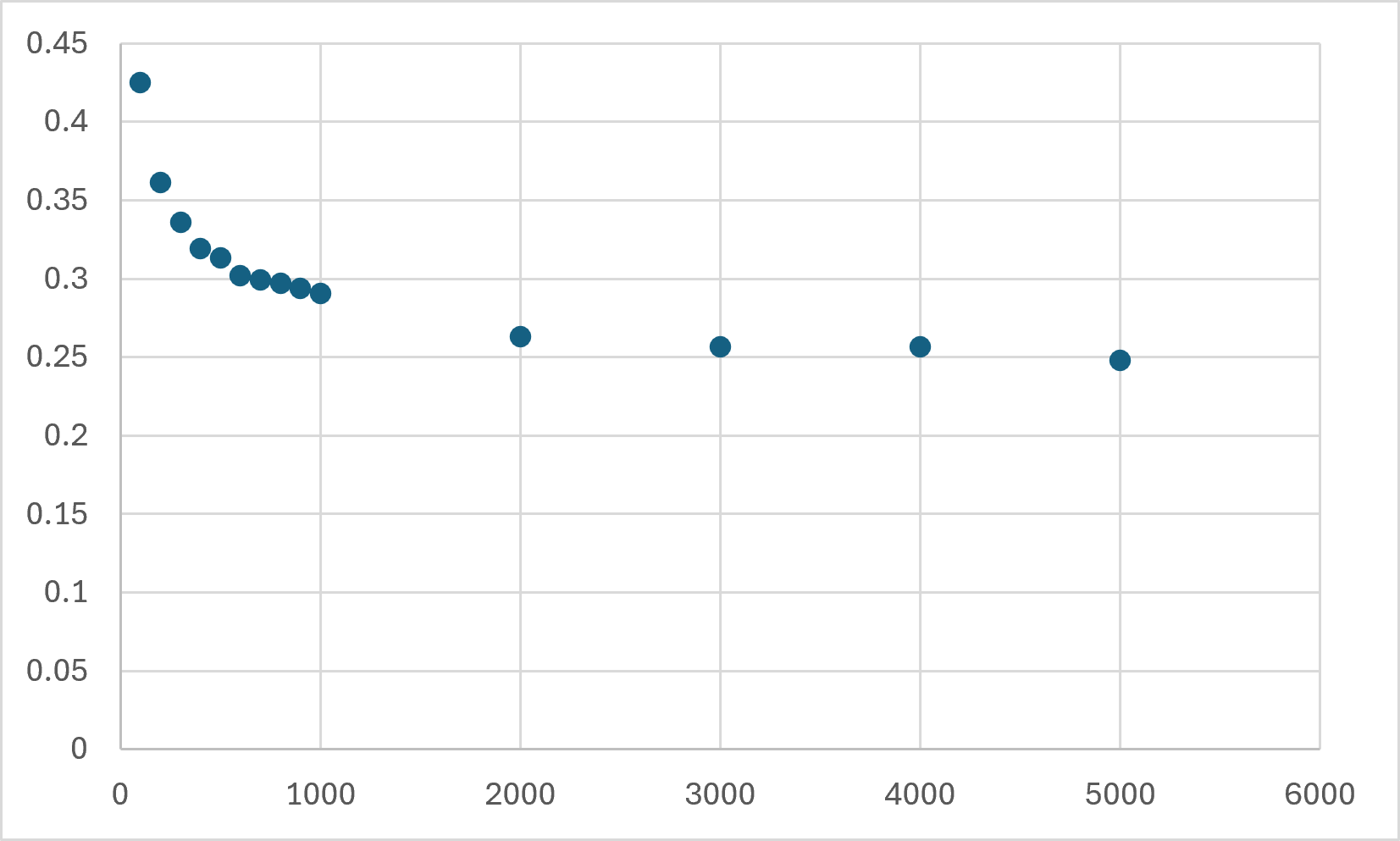} 
        \caption{AFL for $pRo_n$ divided by $n^{3/2}$.}
    \end{subfigure}
\end{figure}

\begin{figure}[H] 
\ContinuedFloat
    \centering
    \begin{subfigure}{0.48\textwidth}
        \includegraphics[width=\textwidth]{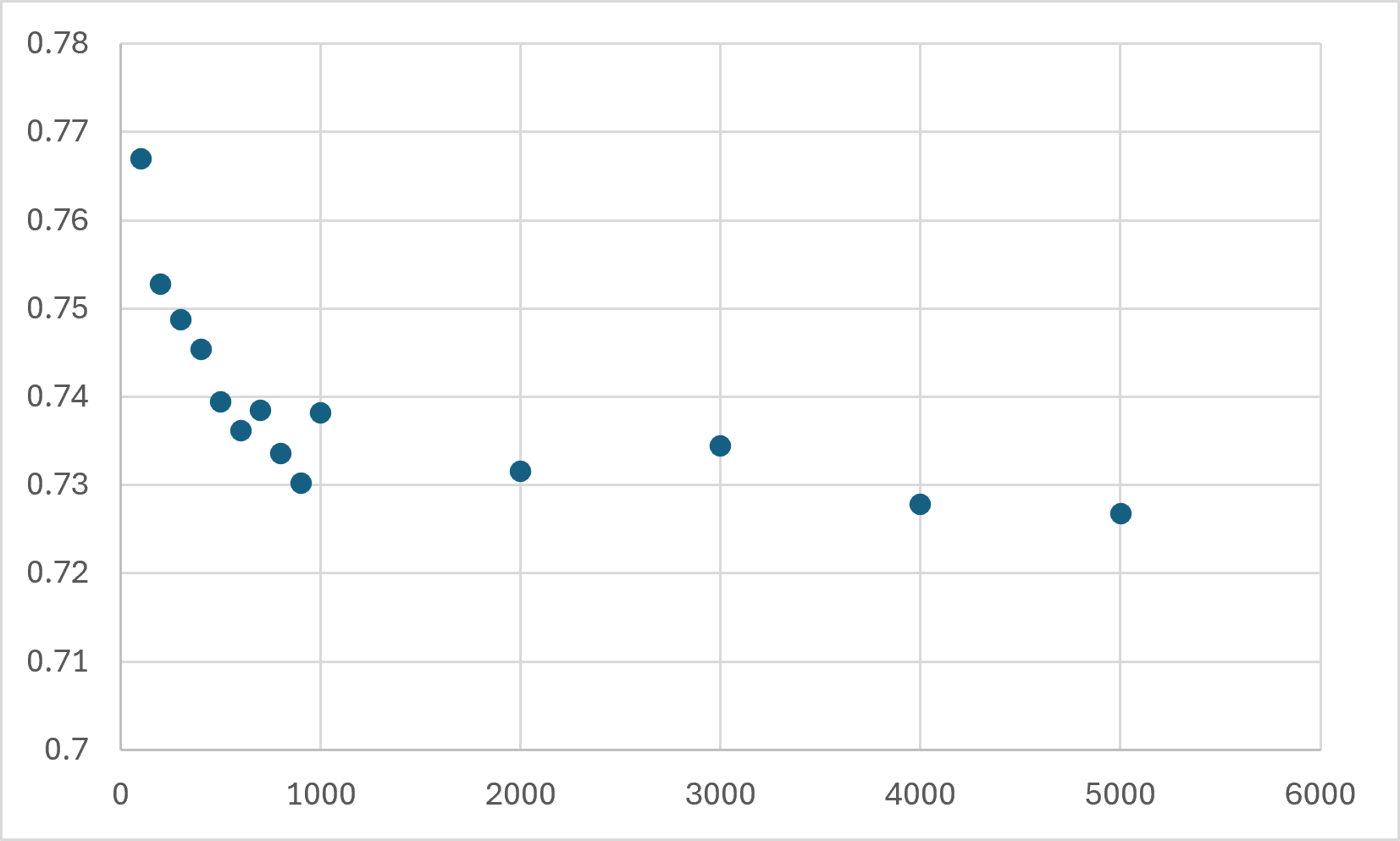} 
        \caption{AFL for $Mo_n$ divided by $n^{3/2}$.}
    \end{subfigure}
    \hfill
    \begin{subfigure}{0.48\textwidth}
        \includegraphics[width=\textwidth]{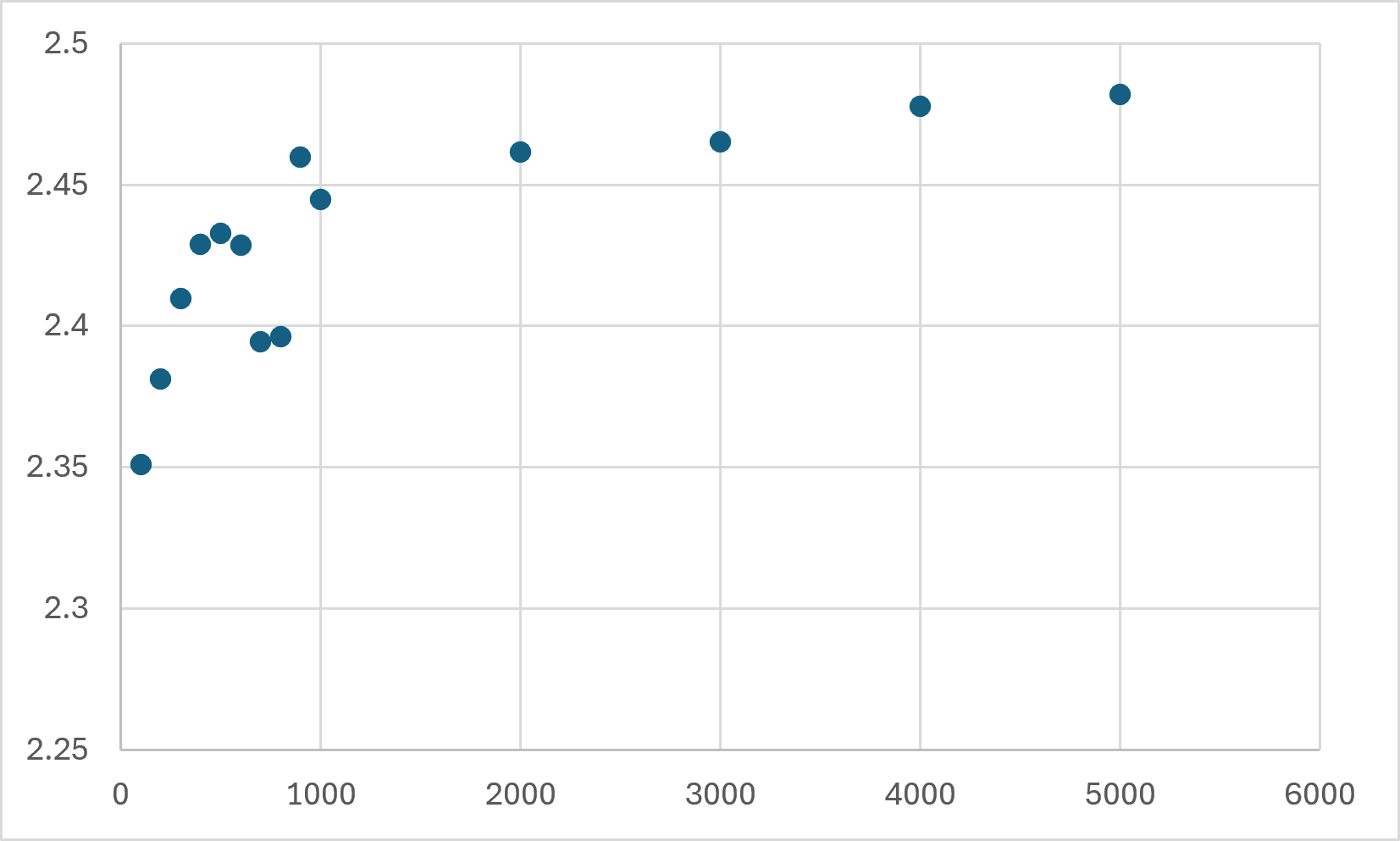} 
        \caption{AFL for $pPa_n$ divided by $n^{3/2}$.}
    \end{subfigure}
\end{figure}

\begin{figure}[H] 
\ContinuedFloat
    \centering
    \begin{subfigure}{0.48\textwidth}
        \includegraphics[width=\textwidth]{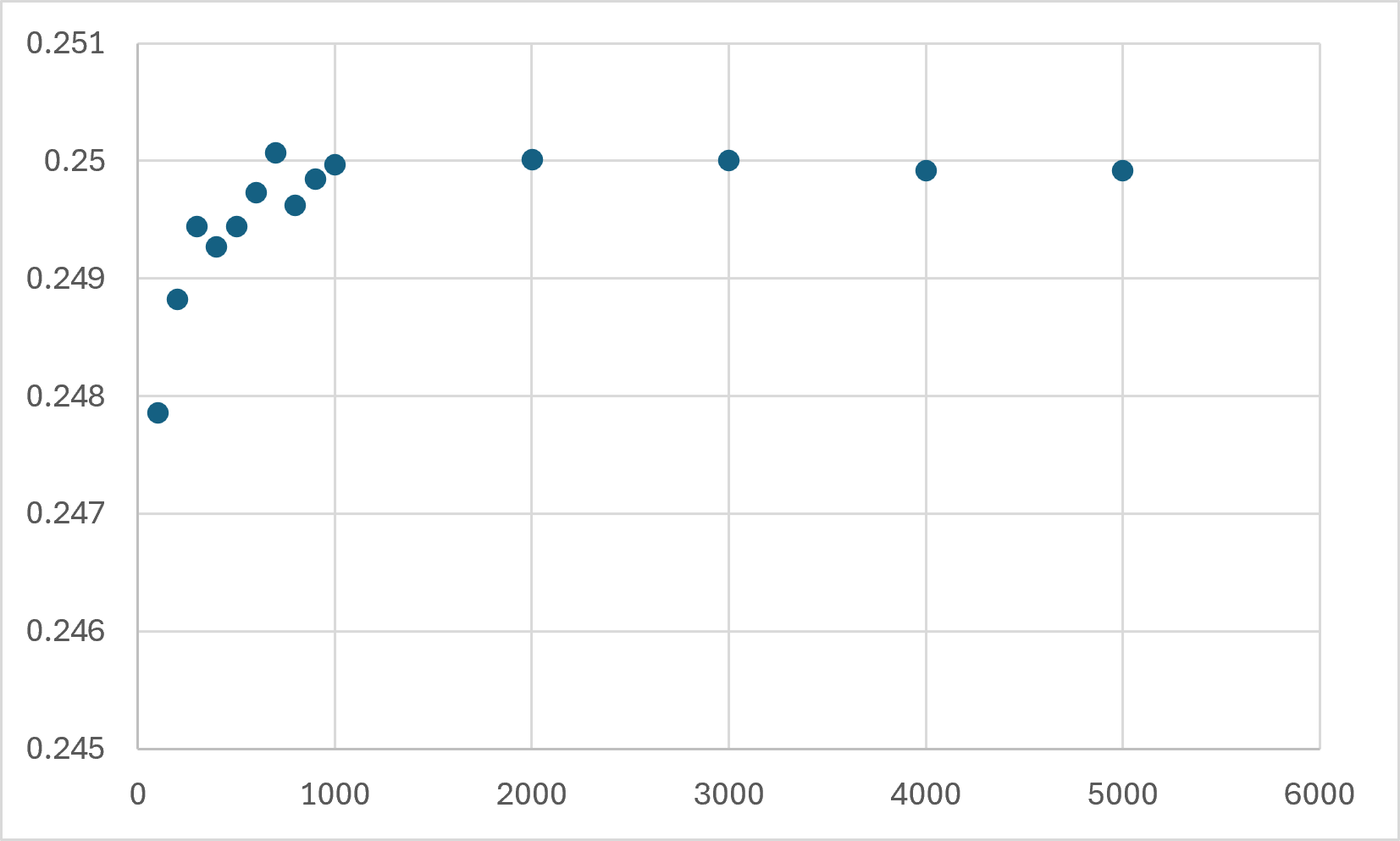} 
        \caption{AFL for $S_n$ divided by $n^{2}$.}
    \end{subfigure}
\end{figure}

\begin{figure}[H] 
\ContinuedFloat
    \centering
    \begin{subfigure}{0.5\textwidth}
        \includegraphics[width=\textwidth]{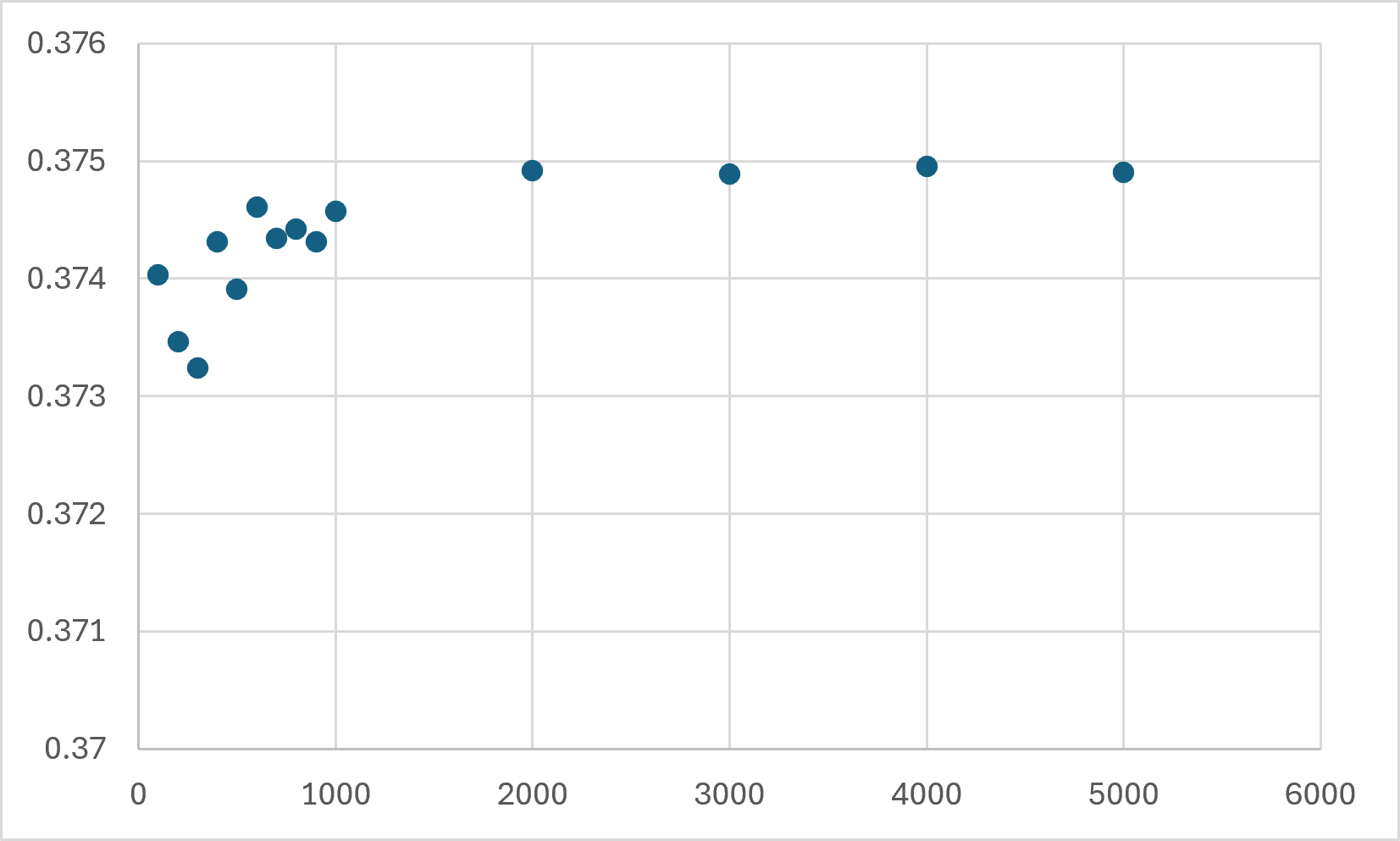} 
        \caption{AFL for $Br_n$ divided by $n^{2}$.}
    \end{subfigure}
    \hfill
    \begin{subfigure}{0.48\textwidth}
        \includegraphics[width=\textwidth]{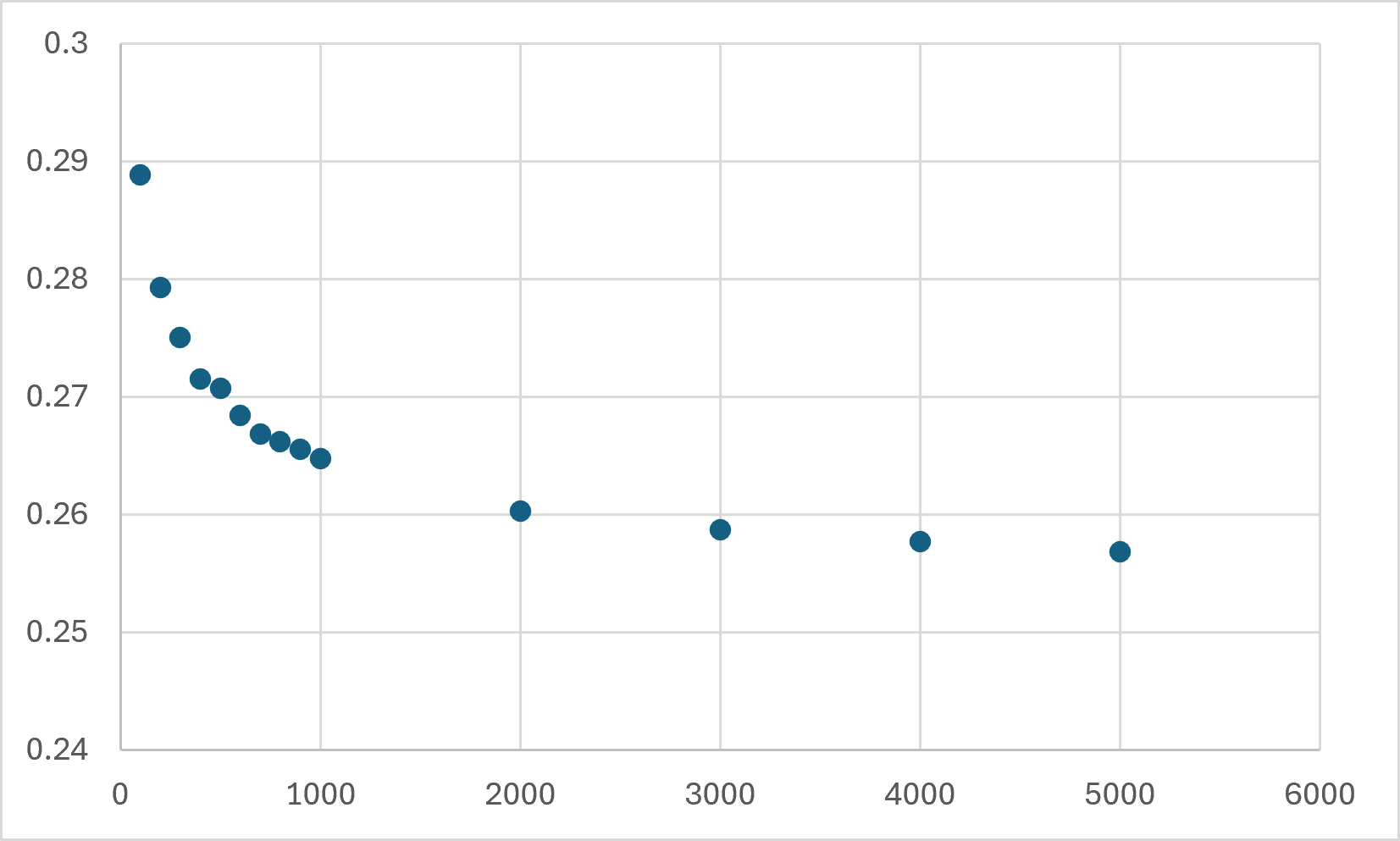} 
        \caption{AFL for $Ro_n$ divided by $n^{2}$.}
    \end{subfigure}
\end{figure}

\begin{figure}[H] 
\ContinuedFloat
    \centering
    \begin{subfigure}{0.48\textwidth}
        \includegraphics[width=\textwidth]{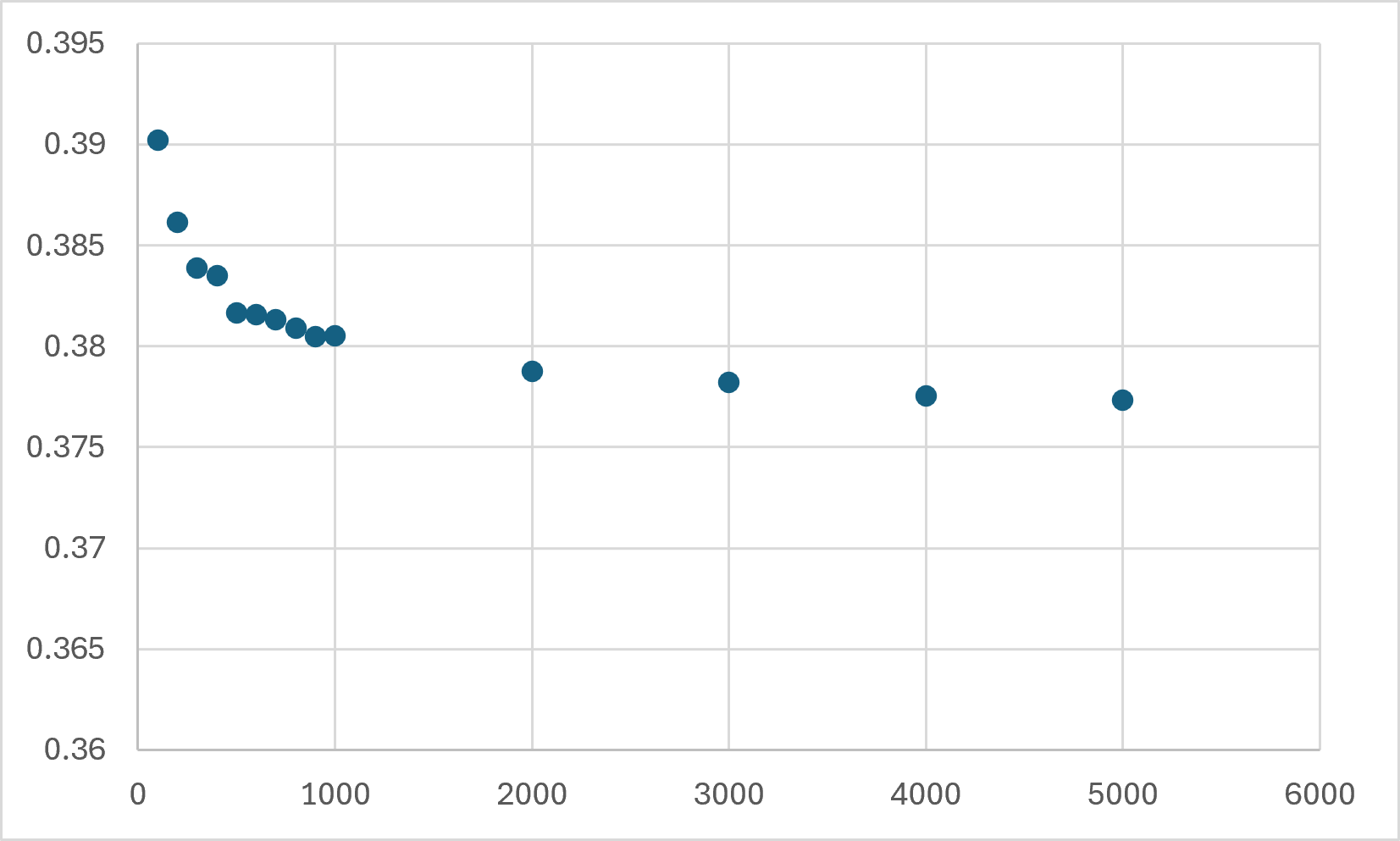} 
        \caption{AFL for $RoBr_n$ divided by $n^{2}$.}
    \end{subfigure}
    \hfill
    \begin{subfigure}{0.48\textwidth}
        \includegraphics[width=\textwidth]{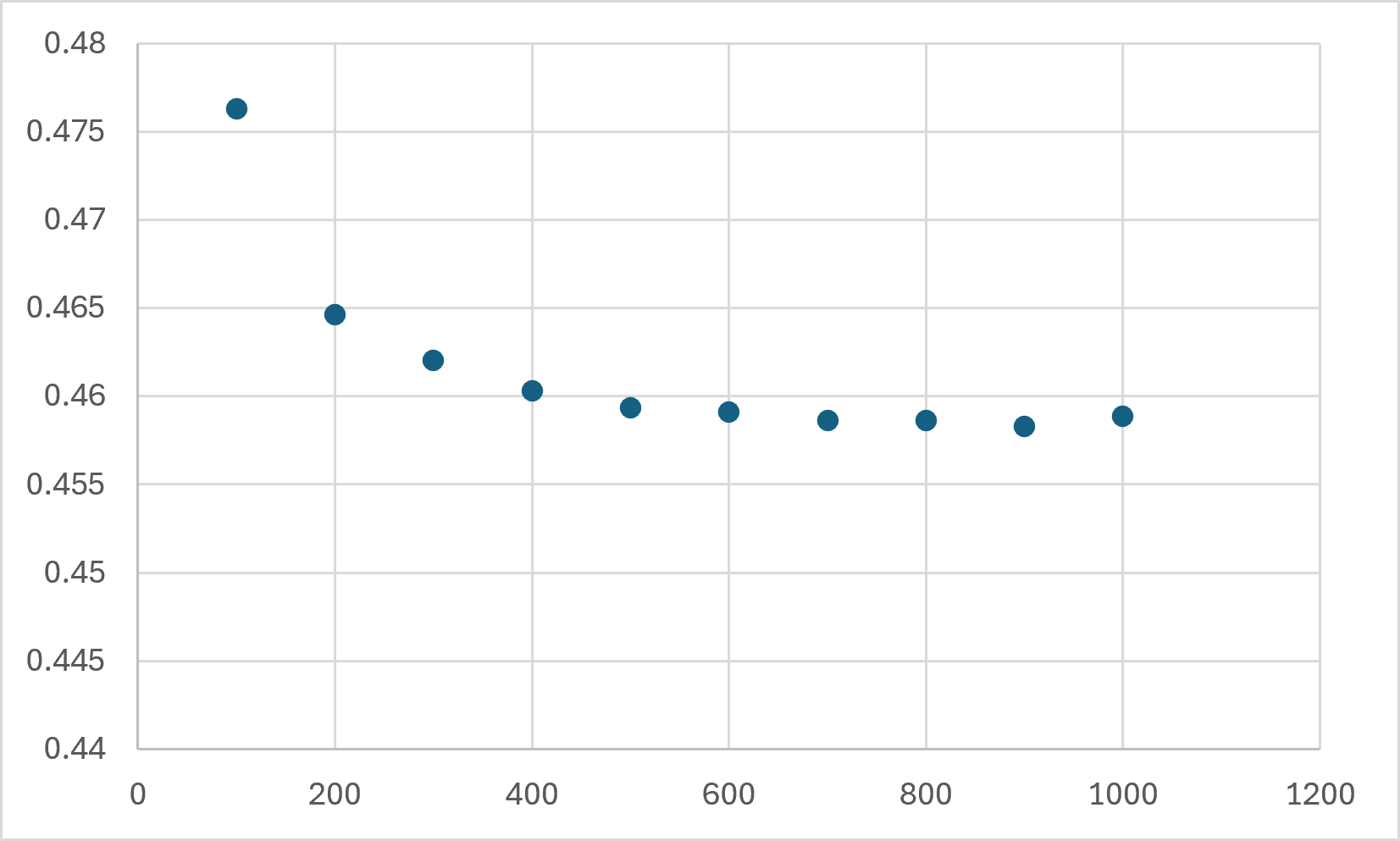} 
        \caption{AFL for $Pa_n$ divided by $n^{2}$.}
    \end{subfigure}
\end{figure}


\newcommand{\etalchar}[1]{$^{#1}$}


\begin{thebibliography}{MMM{\etalchar{+}}21}

\bibitem[AJP24]{AJP-Framization}
F.~Aicardi, J.~Juyumaya, and P.~Papi.
\newblock Framization and Deframization.
\newblock URL: \url{https://arxiv.org/abs/2405.10809}.

\bibitem[AST18]{AnStTu-cellular-tilting}
H.H.~Andersen, C.~Stroppel, and D.~Tubbenhauer.
\newblock Cellular structures using {$\mathrm{U}_q$}-tilting modules.
\newblock {\em Pacific J. Math.}, 292(1):21--59, 2018.
\newblock \url{https://arxiv.org/abs/1503.00224}, \href
{https://doi.org/10.2140/pjm.2018.292.21}
{\path{doi:10.2140/pjm.2018.292.21}}.

\bibitem[Ar25]{Arms-Motzkin-gap}
K.~Arms.
\newblock Representation gap of the Motzkin monoid.
\newblock 2025.
\newblock \url{https://arxiv.org/abs/2510.06707}.

\bibitem[BKL89]{BKL89}
L.~Babai, W.M.~Kantor, and A.~Lubotzky.
\newblock Small-diameter Cayley graphs for finite simple groups.
\newblock {\em European J. Combin.}, 10(6):507--522, 1989.
\newblock \href{https://doi.org/10.1016/S0195-6698(89)80067-8}{\path{doi:10.1016/S0195-6698(89)80067-8}}.

\bibitem[BH14]{BH-Motzkin}
G.~Benkart and T.~Halverson.
\newblock Motzkin algebras.
\newblock {\em European J. Combin.}, 36:473--502, 2014.
\newblock \url{https://arxiv.org/abs/1106.5277},
\href{https://doi.org/10.1016/j.ejc.2013.09.010}{\path{doi:10.1016/j.ejc.2013.09.010}}.

\bibitem[Be99]{Be-catalan-motzkin}
F.~Bernhart.
\newblock Catalan, Motzkin, and Riordan numbers.
\newblock {\em Discrete Mathematics}, 204(1--3):73--112, 1999.
\newblock \href{https://doi.org/10.1016/S0012-365X(99)00054-0}{\path{doi:10.1016/S0012-365X(99)00054-0}}.

\bibitem[Br37]{Br37}
R.~Brauer.
\newblock On algebras which are connected with the semisimple continuous groups.
\newblock {\em Ann. of Math.}, 38(4):857--872, 1937.
\newblock \href{https://doi.org/10.2307/1968843}{\path{doi:10.2307/1968843}}.

\bibitem[Bro55]{Br-gen-matrix-algebras}
W.P.~Brown.
\newblock Generalized matrix algebras.
\newblock {\em Canadian J. Math.}, 7:188--190, 1955.
\newblock \href {https://doi.org/10.4153/CJM-1955-023-2}
{\path{doi:10.4153/CJM-1955-023-2}}.

\bibitem[CEF25]{CEF-full-domain}
L.~Carroll, J.~East, and M.~Fresacher.
\newblock Presentations for semigroups of full-domain partitions.
\newblock 2025.
\newblock \url{https://arxiv.org/abs/2507.05497}.

\bibitem[CEF{\etalchar{+}}]{CEF-trivial-cokernel}
L.~Carroll, J.~East, M.~Fresacher, M.~Lyu, and P.A.A.~Muhammed.
\newblock Presentations for semigroups of trivial-cokernel partitions.
\newblock In preparation.

\bibitem[CDKR15]{CDKR-partition-limit}
B.~Chern, P.~Diaconis, D.~Kane, and R.~Rhoades.
\newblock Central limit theorems for some set partition statistics.
\newblock {\em Adv. in App. Math.}, 70:92--105, 2015.
\newblock \url{https://arxiv.org/abs/1502.00938}, \href{https://doi.org/10.1016/j.aam.2015.06.008}{\path{doi:10.1016/j.aam.2015.06.008}}.

\bibitem[DEG17]{DEG-MotzPartialBrauer}
I.~Dolinka, J.~East, and R.~Gray.
\newblock Motzkin monoids and partial Brauer monoids.
\newblock {\em J. Algebra}, 471:251--298, 2017.
\newblock \url{https://arxiv.org/abs/1512.02279v2}, \href{https://doi.org/10.1016/j.jalgebra.2016.09.018}{\path{doi:10.1016/j.jalgebra.2016.09.018}}.

\bibitem[Ea11]{Ea-partition-presentations}
J.~East.
\newblock Generators and relations for partition monoids and algebras.
\newblock {\em J. Algebra}, 339(1):1--26, 2011.
\newblock \href{https://doi.org/10.1016/j.jalgebra.2011.04.008}{\path{doi:10.1016/j.jalgebra.2011.04.008}}.

\bibitem[EKMW18]{EKMW-MaxSubsemigroup}
J.~East, J.~ Kumar, J.~Mitchell, and W.~Wilson.
\newblock Maximal subsemigroups of finite transformation and diagram monoids.
\newblock {\em J. Algebra}, 504:176--216, 2018.
\newblock URL: \url{https://arxiv.org/abs/1706.04967}, \href{https://doi.org/10.1016/j.jalgebra.2018.01.048}{\path{doi:10.1016/j.jalgebra.2018.01.048}}.

\bibitem[EHS17]{TL-factorization}
D.~Ernst, M.~Hastings, and S.~Salmon.
\newblock Factorization of Temperley--Lieb diagrams.
\newblock {\em Involve, a Journal of Mathematics}, 10(1):89--108, 2017.
\newblock \url{https://arxiv.org/abs/1509.01241}, \href{https://doi.org/10.2140/involve.2017.10.89}{\path{doi:10.2140/involve.2017.10.89}}.

\bibitem[Fe94]{Fenwick}
P.~Fenwick.
\newblock A new data structure for cumulative frequency tables.
\newblock {\em Softw: Pract. Exper.}, 24: 327--336, 1994.
\newblock URL: \href{https://doi.org/10.1002/spe.4380240306}{\path{doi:10.1002/spe.4380240306}}.

\bibitem[FMM24]{Br-factorization}
A.~Francis, D.~Marchei, and E.~Merelli.
\newblock Factorizing the Brauer monoid in polynomial time.
\newblock \url{https://arxiv.org/abs/2402.07874}.

\bibitem[FST25]{FST-GeneralizedRepGap}
M.~Fresacher, W.~Stewart, and D.~Tubbenhauer.
\newblock Generalized diagram categories and monoids, and their representations.
\newblock \url{https://arxiv.org/abs/2512.17177}.

\bibitem[GL96]{GrLe-cellular}
J.J.~Graham and G.~Lehrer.
\newblock Cellular algebras.
\newblock {\em Invent. Math.}, 123(1):1--34, 1996.
\newblock \href{https://doi.org/10.1007/BF01232365}{\path{doi:10.1007/BF01232365}}.

\bibitem[GKP94]{GKP-Concrete}
R.~Graham, D.~Knuth, and O.~Patashnik.
\newblock Concrete Mathematics, 2nd ed.
\newblock {\em Addison-Wesley}, 1994.

\bibitem[GT25]{GT-GrowthDiagCat}
J.~Gruber and D.~Tubbenhauer.
\newblock Growth problems in diagram categories.
\newblock {\em Bull. London Math. Soc.}, 57: 3454--3469, 2025.
\newblock \url{https://arxiv.org/abs/2503.00685}, \href{https://doi.org/10.1112/blms.70163}{\path{doi:10.1112/blms.70163}}.

\bibitem[HR05]{HaRa-partition-algebras}
T.~Halverson and A.~Ram.
\newblock Partition algebras.
\newblock {\em European J. Combin.}, 26(6):869--921, 2005.
\newblock \url{https://arxiv.org/abs/math/0401314}, \href{https://doi.org/10.1016/j.ejc.2004.06.005}{\path{doi:10.1016/j.ejc.2004.06.005}}.

\bibitem[HT25]{HT-AffineDiag}
D.~He and D.~Tubbenhauer.
\newblock Affine diagram categories, algebras and monoids.
\newblock \url{https://arxiv.org/abs/2512.05510}.

\bibitem[Hu20]{Hu-diagram-categories}
M.~Hu.
\newblock Presentations of diagram categories.
\newblock {\em PUMP J. Undergrad. Res.}, 3:1--25, 2020.
\newblock \url{https://arxiv.org/abs/1910.11784},
\href{https://doi.org/10.46787/pump.v3i0.2256}{\path{doi:10.46787/pump.v3i0.2256}}.

\bibitem[Ir37]{Ir-Skellam}
J.~Irwin.
\newblock The frequency distribution of the difference between two independent variates following the same Poisson distribution.
\newblock {\em J. Royal Stat. Soc. Series A}, 100(3):415--416.

\bibitem[Jo87]{Jo87}
V.F.R.~Jones.
\newblock Hecke algebra representations of braid groups and link polynomials.
\newblock {\em Ann. of Math.}, 126(2):335--388, 1987.
\newblock \href{https://doi.org/10.2307/1971403}{\path{doi:10.2307/1971403}}.

\bibitem[JL25]{JL-FramedBlob}
J.~Juyumaya and D.~Lobos.
\newblock Framed blob monoids.
\newblock URL: \url{https://arxiv.org/abs/2501.14125v1}.

\bibitem[Ka87]{Ka87}
L.H.~Kauffman.
\newblock State models and the Jones polynomial.
\newblock {\em Topology}, 26(3):395--407, 1987.
\newblock \href{https://doi.org/10.1016/0040-9383(87)90009-7}{\path{doi:10.1016/0040-9383(87)90009-7}}.

\bibitem[KST24]{khovanov-monoidal-2024}
M.~Khovanov, M.~Sitaraman, and D.~Tubbenhauer.
\newblock Monoidal Categories, representation Gap and cryptography.
\newblock {\em Trans. Amer. Math. Soc. Ser. B}, 11:329--395, 2024.
\newblock \url{https://arxiv.org/abs/2201.01805}, \href{https://doi.org/10.1090/btran/151}{\path{doi:10.1090/btran/151}}.

\bibitem[KT05]{KT-AlgDesign}
J.~Kleinberg and E.~Tardos.
\newblock Algorithm Design.
\newblock {\em Pearson Education}, 2006.
\newblock ISBN 0-321-29535-8.

\bibitem[Kn98]{art-computer-programming2}
D.~Knuth.
\newblock The Art of Computer Programming Vol. 2: Seminumerical Algorithms 3rd Ed.
\newblock {\em Reading, Massachusetts: Addison-Wesley}, 1998.
\newblock ISBN 0-201-89684-2.

\bibitem[Kn98]{art-computer-programming}
D.~Knuth.
\newblock The Art of Computer Programming Vol. 3: Sorting and Searching 2nd Ed.
\newblock {\em Reading, Massachusetts: Addison-Wesley}, 1998.
\newblock ISBN 0-201-89685-0.

\bibitem[Le60]{Lehmer-code}
D.~Lehmer.
\newblock Teaching combinatorial tricks to a computer.
\newblock {\em Combinatorial Analysis, Proceedings of Symposia in Applied Mathematics}, 10:179--193, 1960.
\newblock \href{https://doi.org/10.1090/psapm/010/0113289}{\path{doi:10.1090/psapm/010/0113289}}.

\bibitem[Li25]{Liu-cyclic-diagram}
J.~Liu.
\newblock Representations of cyclic diagram monoids.
\newblock 2025.
\newblock \url{https://arxiv.org/abs/2511.15945}.

\bibitem[MS94]{MS-BlobAlg}
P.~Martin and H.~Saleur.
\newblock The blob algebra and the periodic Temperley-Lieb algebra.
\newblock {\em Lett. Math. Phys.}, 30:189--206, 1994.
\newblock URL: \href{https://doi.org/10.1007/BF00805852}{\path{doi:10.1007/BF00805852}}.

\bibitem[Mit{\etalchar{+}}26]{Semigroups-GAP}
J.D.~Mitchell et al.
\newblock Semigroups -- a GAP package, Version 5.6.3.
\newblock 2026.
\newblock \url{https://semigroups.github.io/Semigroups/},
\href{https://doi.org/10.5281/zenodo.592893}{\path{doi:10.5281/zenodo.592893}}.

\bibitem[NW78]{NW-comb-alg}
A.~Nijenhuis and H.~Wilf.
\newblock Combinatorial algorithms for computers and calculators, 2nd ed.
\newblock {\em Elsevier: Academic Press}.
\newblock \href{https://doi.org/10.1016/C2013-0-11243-3}{\path{doi:10.1016/C2013-0-11243-3}}.

\bibitem[{OEI}23]{oeis}
{OEIS Foundation Inc.}
\newblock The {O}n-{L}ine {E}ncyclopedia of {I}nteger {S}equences, 2023.
\newblock Published electronically at \url{http://oeis.org}.

\bibitem[RTW32]{RTW-Valenztheorie}
G.~Rumer, E.~Teller, and H.~Weyl.
\newblock Eine für die Valenztheorie geeignete Basis der binären Vektorinvarianten.
\newblock {\em Nachrichten von der Gesellschaft der Wissenschaften zu Göttingen, Mathematisch-Physikalische Klasse}, 499--504, 1932.
\newblock \url{http://eudml.org/doc/59396}.

\bibitem[Sp08]{Sp-rec-bell}
M.~Spivey.
\newblock A Generalized Recurrence for Bell Numbers.
\newblock {\em J. Int. Seq.}, 11(2):Article 08.2.5, 3, 2008.
\newblock \href{https://cs.uwaterloo.ca/journals/JIS/VOL11/Spivey/spivey25.html}{\path{cs.uwaterloo.ca/journals/JIS/VOL11/Spivey/spivey25}}.

\bibitem[St83]{St-partition-probability}
A.~Stam.
\newblock Generation of a random partition of a finite set by an urn model.
\newblock {\em J. Comb. Theor.}, Series A 35(2):231--240, 1983.
\newblock \href{https://doi.org/10.1016/0097-3165(83)90009-2}{\path{doi:10.1016/0097-3165(83)90009-2}}.

\bibitem[SW86]{SW-constructive}
D.~Stanton and D.~White.
\newblock Constructive Combinatorics.
\newblock {\em Springer-Verlag, New York}, 1986.
\newblock \href{https://doi.org/10.1007/978-1-4612-4968-9}{\path{doi:10.1007/978-1-4612-4968-9}}.

\bibitem[ST26]{St-github-FastFact}
W.~Stewart and D.~Tubbenhauer.
\newblock GitHub page for the code used in the paper ``Fast factorization in diagram monoids.''
\newblock 2026.
\newblock \url{https://github.com/WillowStewart/FastFactDiagramMonoid}.

\bibitem[ST25]{ST-RepGapRigidPlanar}
W.~Stewart and D.~Tubbenhauer.
\newblock Representation gaps of rigid planar diagram monoids.
\newblock 2025.
\newblock \url{https://arxiv.org/abs/2505.05846}.

\bibitem[Sz75]{Sz-OrthogonalPolynomials}
G.~Szeg{\"o}.
\newblock Orthogonal Polynomials 4th Ed.
\newblock {\em American Mathematical Society}, 1975.
\newblock \href{https://doi.org/10.1090/coll/023}{\path{doi:10.1090/coll/023}}.

\bibitem[TL71]{TL71}
H.N.V.~Temperley and E.H.~Lieb.
\newblock Relations between the `percolation' and `colouring' problem and other
graph-theoretical problems associated with regular planar lattices: some exact
results for the `percolation' problem.
\newblock {\em Proc. Roy. Soc. London Ser. A}, 322(1549):251--280, 1971.
\newblock \href{https://doi.org/10.1098/rspa.1971.0067}{\path{doi:10.1098/rspa.1971.0067}}.

\bibitem[GAP25]{GAP}
The GAP Group.
\newblock GAP -- Groups, Algorithms, and Programming, Version 4.15.1.
\newblock 2025.
\newblock \url{https://www.gap-system.org}.

\bibitem[Tu24]{Tu-sandwich}
D.~Tubbenhauer.
\newblock Sandwich cellularity and a version of cell theory.
\newblock {\em Rocky Mountain J. Math.}, 54(6):1733--1773, 2024.
\newblock \url{https://arxiv.org/abs/2206.06678}, \href{https://doi.org/10.1216/rmj.2024.54.1733}{\path{doi:10.1216/rmj.2024.54.1733}}.

\end{thebibliography}
\end{document}